\documentclass[11pt]{amsart}

\usepackage[margin=2.7cm]{geometry}

\usepackage[table]{xcolor}

\usepackage{ytableau,tikz,varwidth, tikz-cd}
\usepackage{hyperref}

\usepackage{color, colortbl}
\usepackage{array, adjustbox}

\definecolor{Gray}{gray}{0.9}

\newcolumntype{g}{>{\columncolor{Gray}}c}
\newcolumntype{M}{V{6cm}} 
\newcolumntype{N}{V{12cm}} 

\usetikzlibrary{calc, snakes, positioning}
\usepackage[enableskew]{youngtab}
\usepackage{geometry}
\usepackage{enumerate}
\usepackage{amssymb, mathscinet}

\usepackage[all,cmtip]{xy}

\usepackage{multicol}
\usepackage{enumitem}

\usetikzlibrary{matrix}
\usetikzlibrary{shapes}
\usetikzlibrary{arrows, automata}
\usetikzlibrary{decorations,decorations.pathmorphing,decorations.pathreplacing,decorations.markings}
\usetikzlibrary{through}
\usetikzlibrary{decorations.pathreplacing,calligraphy}

\tikzset{every picture/.style={baseline=-.65ex} }
\tikzset{ext/.style={circle, draw,inner sep=1pt},int/.style={circle,draw,fill,inner sep=1pt},nil/.style={inner sep=1pt}}

\tikzset{every loop/.style={draw}}

  \tikzset{->-/.style={decoration={
    markings,
    mark=at position .5 with {\arrow{>}}},postaction={decorate}}}

\newtheorem{thm}{Theorem}[section]
\newtheorem{prop}[thm]{Proposition}

\newtheorem{cor}[thm]{Corollary}

\newcommand{\ZZ}{\mathbb{Z}}
\newcommand{\CC}{\mathbb{C}}
\newcommand{\QQ}{\mathbb{Q}}
\newcommand{\Q}{\mathbb{Q}}
\newcommand{\bS}{\mathbb{S}}
\newcommand{\cF}{\mathcal{F}}

\newcommand{\comment}[1]{}

\newcommand{\sgn}{\operatorname{sgn}}

\newcommand{\M}{\mathcal{M}}
\newcommand{\MM}{\overline{\mathcal{M}}}

\newcommand{\bbS}{\mathbb{S}}

\newcommand{\bGK}{\overline{\GK}{}}
\newcommand{\hide}[1]{}

\newtheorem{Definition}[thm]{Definition}
\newenvironment{definition}
  {\begin{Definition}}{\end{Definition}}

\newtheorem{Example}[thm]{Example}

\newtheorem{Fact}[thm]{Fact}

\newtheorem{Theorem}[thm]{Theorem}

\newtheorem{Lemma}[thm]{Lemma}
\newenvironment{lemma}
  {\begin{Lemma}}{\end{Lemma}}

\newtheorem{Remark}[thm]{Remark}
\newenvironment{remark}
  {\begin{Remark}\rm}{\end{Remark}}

\newtheorem{Proposition}[thm]{Proposition}

\newtheorem{Corollary}[thm]{Corollary}
\newenvironment{corollary}
  {\begin{Corollary}}{\end{Corollary}}

\newtheorem{Question}[thm]{Question}
\newenvironment{question}
  {\begin{Question}\rm}{\end{Question}}

\newtheorem{Conjecture}[thm]{Conjecture}
\newenvironment{conjecture}
  {\begin{Conjecture}\rm}{\end{Conjecture}}

\newcommand{\sX}{{\underline{X}}}

  \theoremstyle{definition} 

\theoremstyle{remark}

\newcommand{\gr}{\mathrm{gr}}
\newcommand{\GC}{\mathsf{GC}}
\newcommand{\fHGC}{\mathsf{fHGC}}
\newcommand{\HGC}{\mathsf{HGC}}
\newcommand{\Graphs}{\mathsf{Graphs}}
\newcommand{\GK}{\mathsf{GK}}
\newcommand{\lp}{\text{-loop}}

\DeclareMathOperator{\Sym}{Sym}

\title{Weight 11 compactly supported cohomology of moduli spaces of curves}

\author{Sam Payne}
\address{UT Department of Mathematics, 2515 Speedway, RLM 8.100, Austin, TX 78712, USA}
\email{sampayne@utexas.edu}

\author{Thomas Willwacher}
\address{Department of Mathematics\\ ETH Zurich\\ R\"amistrasse 101 \\ 8092 Zurich, Switzerland}
\email{thomas.willwacher@math.ethz.ch}

\begin{document}

\begin{abstract}
We study the weight 11 part of the compactly supported cohomology of the moduli space of curves $\M_{g,n}$, using graph complex techniques, with particular attention to the case $n = 0$.  As applications, we prove new nonvanishing results for the cohomology of $\M_g$, and exponential growth with $g$, in a wide range of degrees.  \end{abstract}

\thanks{
SP has been supported by NSF grants DMS--2001502 and DMS--2053261. TW has been supported by the NCCR SwissMAP, funded by the Swiss National Science Foundation.}

\maketitle

\section{Introduction}

The weight 0 and weight 2 graded parts of the compactly supported cohomology of the moduli spaces of curves $\M_{g,n}$ are naturally identified with the cohomology of combinatorially defined graph complexes \cite{CGP1, CGP2,PayneWillwacher} that resemble graph complexes arising in algebraic topology.  Meanwhile, the graded parts in weights 1, 3, 5, 7, and 9 all vanish, because the rational cohomology groups of the Deligne-Mumford compactifications $\MM_{g,n}$ vanish in these degrees \cite{BFP}.  This paper is devoted to studying the lowest nontrivial odd weight graded part of the cohomology of $\M_{g,n}$, in weight 11.  Our main technical result (Proposition~\ref{prop:Bgn intro}) identifies $\gr_{11} H^\bullet_c(\M_{g,n})$ with the cohomology of another combinatorial graph complex resembling those arising in the embedding calculus \cite{FTW2}.  This is similar in spirit to the aforementioned results in weights 0 and 2, and yet the details are substantially different in each weight.  

As an application of this construction, we give new nonvanishing results for the cohomology of $\M_g$ by showing that the $11$th weight graded piece is nonzero.  These results are proved by relating the weight 11 cohomology to the weight 0 cohomology.  Let $\Delta := H^{11}(\MM_{1,11})$, which we view as a 2-dimensional $\Q$-vector space with its Hodge structure or $\ell$-adic Galois representation of weight 11.  It follows from \cite{CanningLarsonPayne} that $\gr_{11} H^\bullet_c(\M_{g,n})$ is isomorphic to a direct sum of copies of $\Delta$.

\begin{thm}\label{thm:emb intro}
  Let $V_g^{r,k}$ denote the degree $k$ and genus $g$ part of the $r$-fold symmetric product
    \[
  \Sym^r\bigg(\bigoplus_{h\geq 3} W_0H_c^\bullet(\M_h)\bigg).  
  \]
Then there is an injective map
  \begin{equation} \label{eq:inj}
     \big( V^{10,k-21}_{g-1} 
  \oplus
  V^{10,k-22}_{g-2} 
  \oplus
  V^{9,k-22}_{g-3} 
  \oplus
  V^{6,k-22}_{g-5} 
  \oplus
  V^{3,k-22}_{g-7}  \big) \otimes \Delta 
   \to \gr_{11}H^k_c(\M_g).
  \end{equation}
\end{thm}

It was previously known that $\dim_\Q W_0H^{2g + k}_c(\M_g)$ grows at least exponentially with $g$ for $k \in \{ 0, 3 \}$ and is nonzero for $k = 7$ and $g = 10$ \cite{CGP1}.  From this, together with previously known nonvanishing results in weights 0 and 2, we have the following corollary.

\begin{corollary} \label{cor:exp-growth}
The dimension of $H^{2g+k}_c(\M_g)$ grows at least exponentially with $g$ 
for each fixed $0 \leq k \leq 53$, except possibly for $k \in \{ 1, 4, 7,20,51 \}$.
\end{corollary}

This corollary is proved using only what is already known about the nonvanishing and growth of $W_0H^\bullet_c(\M_g)$; it is expected that the weight zero cohomology is much larger than what is currently known. By \cite{PayneWillwacher} and Theorem~\ref{thm:emb intro}, each improvement in our understanding in weight zero will lead to corresponding improvements in weights 2 and 11, respectively.  Even a single new nonvanishing weight zero cohomology group could significantly extend the range of $k$ in which $H^{2g+k}_c(\M_g)$ is known to grow at least exponentially with $g$.  

\begin{conjecture}
The dimension of $H^{2g+k}_c(\M_g)$ grows at least exponentially with $g$ for all but finitely many non-negative  integers $k$.
\end{conjecture}
 
\noindent For $k = 4$ and $20$, we note that $H^{10}_c(\M_3)$ and $H^{34}_c(\M_7)$ are nonzero. Each is Poincar\'e dual to a corresponding $H^2$, which contains a nonzero class $\kappa$.  For $k = 7$, as noted above, $H^{27}_c(\M_{10})$ is nonzero in weight 0.  However, there is no $g$ for which $H^{2g+1}_c(\M_g)$ is known to be nonzero.

\begin{question} \label{q:vanish}
Does $H^{2g+1}_c(\M_g)$ vanish for all $g$? 
\end{question}

\begin{remark}
The cohomology group $H^{2g+1}(\GC_0^{g\lp})$ of the loop order $g$ part of Kontsevich's graph complex injects into $H^{2g+1}_c(\M_g)$ \cite{CGP1}, so a positive answer to Question~\ref{q:vanish} would imply that this cohomology group vanishes for all $g$. This is equivalent to the vanishing of $H^1(\GC_2)$, the first cohomology group of a degree shifted version of $\GC_0$. The vanishing of this cohomology group is a well-known open problem in homological algebra and algebraic topology; it is of significant interest because $H^1(\GC_2)$ is the space of obstructions to a variety of problems, including the existence of Drinfeld associators \cite{Drinfeld}, the existence of formality maps in deformation quantization \cite{KontsevichFormality}, and the rational intrinsic formality of the little disks operad \cite{FresseWillwacherIntrinsic}.  
\end{remark}

We now explain how Corollary~\ref{cor:exp-growth} follows from Theorem~\ref{thm:emb intro} before discussing the graph complexes that arise in our study of the weight 11 cohomology. 

\begin{proof}[Proof of Corollary~\ref{cor:exp-growth}]
Recall that $\dim_\Q W_0H^{2g + k}_c(\M_g)$ grows at least exponentially with $g$ for $k \in \{ 0, 3 \}$ and is equal to 1 for $k = 7$ and $g = 10$ \cite{CGP1, PayneWillwacher}.  It then follows from Theorem~\ref{thm:emb intro} that, for fixed $k$ and $r\geq 2$, $\dim_\Q V^{r, 2g + k}_g$ grows at least exponentially with $g$ whenever $k$ is in the set
\[
U_r =
\{ 0,3,6,\dots, 3r  \} 
\cup \{7,10,13,\dots, 4+3r \}. 
\]
Note that, by the sign conventions for graded vector spaces recalled in \S\ref{sec:dg preliminaries}, below, the image of $v \otimes v$ in $\Sym^2 V$ is zero when $\deg(v)$ is odd.  In particular, since $W_0 H_c^{27}(\M_{10})$ is $1$-dimensional and of odd degree, its symmetric powers vanish.

Taking into account the degree shifts, the injection \eqref{eq:inj} yields at least exponential growth of $\dim_\Q H^{2g+k}_c(\M_g)$ for $k$ in the set 
\[
  (19+ U_{10}) \cup 
  (18+ U_{10}) \cup
  (16+ U_{9}) \cup
  (12+ U_{6}) \cup
  (8+ U_{3})\, .
\]
Similarly, from \cite{PayneWillwacher} we know that the weight $2$ cohomology contributes at least exponential growth of $\dim_\Q H^{2g+k}_c(\M_g)$ for $k$ in the set 
\[
  (3+U_2) \cup (2+U_2),
\]
and the result follows.
\end{proof}

The following picture illustrates the values of $k$ for which $\dim_\Q H^{2g+k}_c(\M_g)$ is now known to grow at least exponentially with $k$, with dark grey boxes for the previously known cases (from weight 0 and 2) and light grey boxes for the new contributions from weight 11. 
\[
\begin{tikzpicture}
  \begin{scope}[yshift=.2cm,scale=.2]
    \foreach \x in {0, 1, 2, 3, 4, 5, 6, 7, 9, 10, 13, 20, 51, 54, 55, 56, 57, 58, 59, 60}
    \node[minimum size=2mm, draw=black, fill=white, inner sep=0] at (\x,0) {};
  \end{scope}
  \begin{scope}[yshift=.2cm, scale=.2]
    \foreach \x in {8, 11, 12, 14, 15, 16, 17, 18, 19, 21, 22, 23, 24, 25, 26, 27, 28, 29, 30, 31, 32, 33, 34, 35, 36, 37, 38, 39, 40, 41, 42, 43, 44, 45, 46, 47, 48, 49, 50, 52, 53}
    \node[minimum size=2mm, draw=black, fill=black!20, inner sep=0] at (\x,0) {};
    \end{scope}
    \begin{scope}[yshift=.2cm, scale=.2]
      \foreach \x in {2, 3, 5, 6, 8, 9, 10, 12, 13}
      \node[minimum size=2mm, draw=black, fill=black!50, inner sep=0] at (\x,0) {};
      \end{scope}
  \begin{scope}[yshift=.2cm, scale=.2]
    \foreach \x in {0,3}
    \node[minimum size=2mm, draw=black, fill=black!50, inner sep=0] at (\x,0) {};
    \end{scope}
\begin{scope}[scale=.2]
\foreach \x in {0,10,20,30,40,50,60}
\draw (\x,1) -- (\x,-.5) (\x,-1) node {\x};
\end{scope}
\end{tikzpicture}
\]

\begin{remark}
Corollary~\ref{cor:exp-growth} is only a rough summary of what one can deduce from Theorem~\ref{thm:emb intro} and previous known nonvanishing results in weights 0 and 2.  One also gets specific bounds on the genera for which $H^{2g+k}_c(\M_g)$ is nonzero, and lower bounds on the dimensions of these groups.\end{remark}

The results above are proved by identifying the weight 11 compactly supported cohomology of $\M_{g,n}$ with the tensor product of $\Delta = H^{11}(\MM_{1,11})$ with the cohomology of a graph complex that we now describe.   The graph complex $B_{g,n}$ is a differential graded vector space generated by genus $g$ graphs with $n$ legs numbered $1,\dots,n$, at least 11 legs labeled $\omega$, and an arbitrary number of legs labeled $\epsilon$.  Each connected component contains at least one $\epsilon$- or $\omega$-labeled leg.
The \emph{genus} of a generating graph is the loop order of the connected graph obtained by gluing together all $\epsilon$- and $\omega$-legs, plus one.
The cohomological degree is: 
\[
 22-\#\omega+\#\text{edges}-n .
\]
For example, the following graph is a degree 22 generator of $B_{9,1}$.
\[
  \begin{tikzpicture}
    \node (v) at (0,0) {$\epsilon$};
    \node (w) at (1,0) {$1$};
    \draw (v) edge (w);
  \end{tikzpicture}
  \begin{tikzpicture}
    \node[int] (i) at (0,.5) {};
    \node (v1) at (-.5,-.2) {$\omega$};
    \node (v2) at (0,-.2) {$\omega$};
    \node (v3) at (.5,-.2) {$\omega$};
  \draw (i) edge (v1) edge (v2) edge (v3);
  \end{tikzpicture}
  \begin{tikzpicture}
    \node[int] (i) at (0,.5) {};
    \node (v1) at (-.5,-.2) {$\omega$};
    \node (v2) at (0,-.2) {$\omega$};
    \node (v3) at (.5,-.2) {$\omega$};
  \draw (i) edge (v1) edge (v2) edge (v3);
  \end{tikzpicture}
  \begin{tikzpicture}
    \node[int] (i) at (0,.5) {};
    \node (v1) at (-.5,-.2) {$\omega$};
    \node (v2) at (0,-.2) {$\omega$};
    \node (v3) at (.5,-.2) {$\omega$};
  \draw (i) edge (v1) edge (v2) edge (v3);
  \end{tikzpicture}
  \begin{tikzpicture}
    \node[int] (i) at (0,.5) {};
    \node (v1) at (-.5,-.2) {$\omega$};
    \node (v2) at (0,-.2) {$\omega$};
    \node (v3) at (.5,-.2) {$\omega$};
  \draw (i) edge (v1) edge (v2) edge (v3);
  \end{tikzpicture}
\]
The differential $\delta$ on $B_{g,n}$ is a sum of three pieces, 
\[
\delta = \delta_\omega  + \delta_s^\bullet + \delta_s^\circ.
\]
The piece $\delta_\omega$ changes one $\omega$- to an $\epsilon$-label, the piece $\delta_s^\bullet$ splits vertices, and the piece $\delta_s^\circ$ joins together a subset of the $\epsilon$-legs with either 0 or 1 of the $\omega$-legs, and attaches a new leg labeled $\epsilon$ or $\omega$, respectively.  See \S\ref{sec:truncation} for details. 

\begin{prop} 
\label{prop:Bgn intro}
The weight 11 compactly supported cohomology of $\M_{g,n}$ is isomorphic to the tensor product of $\Delta$ with the cohomology of $B_{g,n}\colon$
\[
  \gr_{11} H_c^\bullet( \M_{g,n} )\cong H(B_{g,n},\delta) \otimes \Delta.
\]
\end{prop}

\noindent When $E(g,n):=3g+2n-25$ is small, the graph complex $B_{g,n}$ is sufficiently simple that its cohomology, and hence $\gr_{11}H^\bullet_c(\M_{g,n})$, can be computed by hand. We carry this through for $E(g,n) \leq 3$ in \S\ref{sec:low excess}. We show that $\gr_{11}H^\bullet_c(\M_{g,n})$ vanishes when $E(g,n) < 0$.  In particular, $\gr_{11}H^\bullet_c(\M_g) = 0$ for $g \leq 8$.  In the first nontrivial case without marked points, we find that 
\[
\gr_{11}H^{k}_c(\M_9) \cong \begin{cases} \Delta & \text{ for $k = 22$,} \\ 0 & \text{ otherwise.} \end{cases}
\] 
We also find large families of nontrivial graph cohomology classes for $n=0$ that give rise to Theorem~\ref{thm:emb intro}. These families are constructed in \S\S\ref{sec:first inj}-\ref{sec:second inj}.

As another application of Proposition~\ref{prop:Bgn intro}, we give a formula for a generating function for the $\bbS_n$-equivariant Euler characteristic of $\gr_{11}(H_c^\bullet(\M_{g,n}))$, analogous to the formulas in weights 0 and 2 proved in \cite{CFGP} and \cite{PayneWillwacherEuler}, respectively.
This formula, along with numerical results for a range of $g,n$, are presented in \S\ref{sec:euler}.  For $g = 2$ and $3$ and $n \leq 14$, our results agree with data obtained independently by Bergstr\"om and Faber using local systems and the trace of Frobenius \cite{BF}.  For $g = 3$, the computations of Bergstr\"om and Faber are conditional on a conjectural list of motives of weight at most 22 that can appear in moduli spaces of curves, based on the work of Chenevier and Lannes \cite{CL}.  Our results confirm the weight 11 part of these computations unconditionally.

\section{Preliminaries}

\subsection{Graded vector spaces, symmetric products, and differentials}
\label{sec:dg preliminaries}
Let $V=\bigoplus_{n\in \mathbb Z} V^n$ denote a graded vector space over $\Q$, with $V^n$ the subspace of degree $n$. We write $|v|=n$ for the degree of a homogeneous element $v\in V^n$.

We follow the usual Koszul sign convention. In other words, the preferred isomorphism exchanging the factors in the tensor product of graded vector spaces $V$ and $W$
\[
V \otimes W \xrightarrow{\sim} W\otimes V
\]
is given on homogeneous elements $v\in V$ and $w\in W$ by 
\[
  v\otimes w \mapsto (-1)^{|v||w|} w\otimes v.
\]
This convention induces an action of the symmetric group $\bbS_k$ on the tensor power
\[
V^{\otimes k} = \underbrace{V\otimes \cdots \otimes V}_{k}  
\]
of a graded vector space $V$. We define 
\[
\Sym^k V := V^{\otimes k} / \bbS_k
\]
to be the space of coinvariants.  For $v_1,\dots,v_k\in V$ we write $v_1\cdots v_k$ for the equivalence class of $v_1\otimes \cdots \otimes v_k$ in $\Sym^k V$. Note that, by the Koszul sign convention, when $k=2$ and $v_1,v_2\in V$ are homogeneous elements
\begin{align*}
v_1v_2 &= (-1)^{|v_1||v_2|} v_2v_1.
\end{align*}
In particular, if $v$ is homogeneous of odd degree then $v^2 = 0$ in $\Sym^2 V$.

Many of the graded vector spaces that we consider are also equipped with a differential. Throughout, we follow cohomological conventions for these dg vector spaces, i.e., the differentials increase the cohomological degree by 1.

\subsection{The Getzler-Kapranov graph complex}
The cohomology of the Deligne-Mumford compactifications $H^\bullet(\MM_{g,n})$ of the moduli spaces of curves assemble to form a modular cooperad $H(\MM)$. The modular cooperad structure encapsulates the symmetric group actions and the boundary-pullback operations 
\begin{align*}
\xi^* \colon H^\bullet(\MM_{g_1+g_2,n_1+n_2}) &\to H^\bullet(\MM_{g_1,n_1+1})\otimes H^\bullet(\MM_{g_2,n_2+1}) \\
\eta^* \colon H^\bullet(\MM_{g+1,n}) &\to H^\bullet(\MM_{g,n+2}),
\end{align*}
together with the natural compatibility relations among these pullbacks and group actions.  

For any modular cooperad one can define its Feynman transform, following Getzler and Kapranov \cite{GK}, see also \cite[\S2.4]{PayneWillwacher}. We define the Getzler-Kapranov complex $\GK$ to be the Feynman transform of the modular cooperad $H(\MM)$ 
\[
\GK := \cF H(\MM),
\]
and write $\GK_{g,n}$ for the part of genus $g$ and arity $n$, as in \cite[\S2.5]{PayneWillwacher}.
Generators of $\GK_{g,n}$ are dual graphs of stable curves of genus $g$ with $n$ numbered external legs, each of whose vertices $v$ is decorated by a copy of $H^{k_v}(\MM_{g_v, n_v})$, where $g_v$ and $n_v$ are the genus and valence of the vertex $v$, respectively.  The genus $g$ is the loop order of the graph plus the sum of the numbers $g_v$. The cohomological degree of a generator is the number of structural edges (not counting numbered legs) plus the sum of the degrees of decorations $\sum_v k_v$; the differential $\delta$ on $\GK_{g,n}$ is defined using the modular cooperad operations $\xi^*, \eta^*$ and increases the cohomological degree by 1.

There is an additional grading of $\GK_{g,n}$ by weight. The weight of a generator is $\sum_v k_v$, the sum of the degrees of the decorations, and the weight is preserved by the differential.  We write $\GK^k_{g,n}$ for the subcomplex generated by graphs of weight $k$, so $(\GK_{g,n}, \delta)$ splits as a direct sum
\[
\GK_{g,n} \cong \bigoplus_k \GK_{g,n}^k,
\]

The cohomology of the weight $k$ part of $\GK_{g,n}$ is identified with the weight $k$ graded part of the compactly supported cohomology of the open moduli space
\[
  H^\bullet(\GK^k_{g,n}, \delta) \cong \gr_k H^\bullet_c(\M_{g,n}).
\]
In this paper we study the weight 11 part $\GK^{11}_{g,n}$, whose cohomology computes $\gr_{11} H^\bullet_c(\M_{g,n})$.

\subsection{The weight 11 Getzler-Kapranov complex}

The complex $\GK^{11}_{g,n}$ has a relatively simple description because $H^k(\MM_{g,n})$ vanishes for all odd $k \leq 9$, by \cite[Theorem 1.1]{BFP}. It follows that, in each generator for $\GK^{11}_{g,n}$, there is one vertex $v$, which we call the ``special vertex" with $k_v = 11$, and all other vertices are decorated by $H^0$.  Since $H^{0}(\MM_{g_v, n_v})=\Q$ these latter decorations are essentially trivial and can be ignored.  

The possibilities for the decoration at the special vertex are as follows \cite{CanningLarsonPayne}.  Let $W_n:=V_{(n-10)1^{10}}$ be the irreducible $\bS_n$-representation corresponding to the Young diagram $(n-10)1^{10}$, and let $\Delta = H^{11}(\MM_{1,11})$. Then 
  \[
      H^{11}(\MM_{g,n})
    \cong
  \begin{cases}
     W_n \otimes \Delta & \text{for $g = 1$ and $n \geq 11$};\\
0 & \text{otherwise.}
\end{cases}
  \] 
In particular, the special vertex $v$ in each generator for $\GK^{11}_{g,n}$ has genus $g_v = 1$ and  valence $n_v \geq 11$. The following figure depicts a typical generator; the special vertex is indicated by a double circle, and $x \in H^{11}(\MM_{1,n})$ is the decoration at the special vertex.
\[
\begin{tikzpicture}  
  \node[ext,accepting, label=0:{$\scriptstyle x$}] (v) at (0,0){$\scriptstyle 1$};
  \node[ext] (v1) at (-3,1) {$\scriptstyle 0$};
  \node[ext] (v2) at (-2,1) {$\scriptstyle 2$};
  \node[ext] (v3) at (-1,1) {$\scriptstyle 0$};
  \node[ext] (v4) at (0,1) {$\scriptstyle 2$};
  \node[ext] (v5) at (1,1) {$\scriptstyle 0$};
  \node[ext] (v6) at (2,1) {$\scriptstyle 0$};
  \node[ext] (v7) at (3,1) {$\scriptstyle 1$};
  \node (e1) at (-4,1) {$\scriptstyle 1$};
  \node (e2) at (-1.5,-1) {$\scriptstyle 2$};
  \node (e3) at (-.5,-1) {$\scriptstyle 3$};
  \node (e4) at (.5,-1) {$\scriptstyle 4$};
  \node (e5) at (1.5,-1) {$\scriptstyle 5$};
  \draw (v) edge[loop left] (v) edge (v1) edge (v2) edge (v3) edge (v4) edge (v5) edge (v6) edge (v7) edge (e5) edge (e2) edge (e3) edge (e4)
  (v2) edge (v1) edge (v3) (v4) edge (v3) edge (v5) (v6) edge (v5) edge (v7)
  (v7) edge[loop above] (v7) (v1) edge (e1);
\end{tikzpicture}
\in \GK^{11}_{14,5}
\]
The genus $g_v$ of each vertex $v$ is inscribed in the corresponding node. Note that generators can have tadpoles, i.e., edges connecting a vertex to itself.

\section{A combinatorial graph complex for weight 11}

In this section, we give a more precise description of the combinatorial graph complex $B_{g,n}$ discussed in the introduction and prove Proposition~\ref{prop:Bgn intro}. The proof is a zig-zag of quasi-isomorphisms between $\GK_{g,n}^{11}$ and $B_{g,n} \otimes \Delta$.  The first step in our zig-zag is a surjective quasi-isomorphism from $\GK^{11}_{g,n}$ to a quotient complex whose generators do not have tadpoles, except at the special vertex, and whose non-special vertices are all of genus 0.

\begin{definition}
Let $I_{g,n}\subset \GK^{11}_{g,n}$ be the dg subspace spanned by graphs with at least one non-special vertex $v$ that carries a tadpole or a positive genus $g_v\geq 1$. Then we define  
\[
  \bGK^{11}_{g,n} = \GK^{11}_{g,n} / I_{g,n}.
\]
\end{definition}
\noindent In other words, in $\bGK^{11}_{g,n}$ we set to zero all generators with tadpoles or positive genera at non-special vertices. The special vertex nevertheless always has genus 1, and may also have tadpoles.

\begin{prop}
The quotient map 
\[
  \GK^{11}_{g,n} \to \bGK^{11}_{g,n}
\]
is a quasi-isomorphism of dg vector spaces.
\end{prop}
\begin{proof}
We endow both sides with the descending filtration on the number of vertices.
The differential on the associated graded of $\bGK^{11}_{g,n}$ is zero, while that on $\GK^{11}_{g,n}$ is given by the part that reduces the genus of a non-special vertex and adds a tadpole:
\[
\begin{tikzpicture}
  \node[ext] (v) at (0,0) {$\scriptstyle g_v$};
  \draw (v) edge +(-.5,-.5)   edge +(0,-.5)  edge +(.5,-.5);  
\end{tikzpicture}
\mapsto 
\begin{tikzpicture}
  \node[ext] (v) at (0,0) {$\scriptstyle g_v-1$};
  \draw (v) edge +(-.5,-.5)   edge +(0,-.5)  edge +(.5,-.5) edge [loop above] (v);  
\end{tikzpicture}
\]
The cohomology of this differential is given by graphs in which every non-special vertex has genus 0 and no tadpoles. The proof is similar to (and simpler than) the arguments in \cite[\S4]{PayneWillwacher}.
\end{proof}

In pictures of generators for $\bGK_{g,n}^{11}$, we omit the genus of the vertices. The special vertex is indicated by a double circle and has genus 1. All other vertices are of genus 0. For instance, the following depicts a generator for $\bGK_{8,6}^{11}$.

\[
\begin{tikzpicture}  
  \node[ext,accepting, label=0:{$\scriptstyle x$}] (v) at (0,0){};
  \node[int] (v1) at (-3,1) {};
  \node[int] (v2) at (-2,1) {};
  \node[int] (v3) at (-1,1) {};
  \node[int] (v4) at (0,1) {};
  \node[int] (v5) at (1,1) {};
  \node[int] (v6) at (2,1) {};
  \node[int] (v7) at (3,1) {};
  \node (e1) at (-4,1) {$\scriptstyle 1$};
  \node (e2) at (-1.5,-1) {$\scriptstyle 2$};
  \node (e3) at (-.5,-1) {$\scriptstyle 3$};
  \node (e4) at (.5,-1) {$\scriptstyle 4$};
  \node (e5) at (1.5,-1) {$\scriptstyle 5$};
  \node (e6) at (4,1) {$\scriptstyle 6$};
  \draw (v) edge[loop left] (v) edge (v1) edge (v2) edge (v3) edge (v4) edge (v5) edge (v6) edge (v7) edge (e5) edge (e2) edge (e3) edge (e4)
  (v2) edge (v1) edge (v3) (v4) edge (v3) edge (v5) (v6) edge (v5) edge (v7)
  (v7) edge (e6) (v1) edge (e1);
\end{tikzpicture}
\]

To summarize, each generator of $\bGK^{11}_{g,n}$ has the following form:

\begin{itemize}
\item A connected graph $\Gamma$ of loop order $g-1$ with one special vertex $v$ and $n$ numbered legs, in which all non-special vertices have valence at least 3.
\item The special vertex has valence $n_v \geq 11$ and is decorated by an element $x\in H^{11}(\MM_{1,n_v})$.  (The markings in $\MM_{1,n_v}$ are implicitly identified with the half-edges at the special vertex.)
\item There are no tadpoles at non-special vertices. 
\item The graph is equipped with an orientation $o$ given by an ordering on the set of structural edges (i.e., all edges other than the numbered legs). \end{itemize}
We suggestively denote the orientation by $o=e_1\wedge \cdots \wedge e_k$ with $e_1,\dots,e_k$ the structural edges of $\Gamma$.
We impose two relations on these generators. 

\begin{itemize}
\item First, we identify isomorphic graphs. That is, if $\phi\colon \Gamma\to \Gamma'$ is an isomorphism, we set 
\begin{equation}\label{equ:bGK gen rel 1}
(\Gamma, e_1\wedge\cdots \wedge e_k, x) 
=
(\Gamma',\phi(e_1)\wedge\cdots \wedge\phi(e_k), \phi(x)),
\end{equation}
with $\phi(x)$ relabeling the punctures in $\MM_{1,n_v}$ according to the isomorphism $\phi$.
\item Second, we identify orderings up to sign. That is, for 
a permutation $\sigma\in \bbS_k$ we set 
\begin{equation}\label{equ:bGK gen rel 2}
    (\Gamma, e_1\wedge\cdots \wedge e_k,x )
    =
    \sgn(\sigma)
    (\Gamma, e_{\sigma(1)}\wedge\cdots \wedge e_{\sigma(k)},x)\,.
\end{equation}
\end{itemize}
The differential acts by splitting vertices. The vertex split of the special vertex uses the map 
\[
\xi^* \colon H^{11}(\MM_{1,r})\to H^{11}(\MM_{1,r-s+1})\otimes H^0(\MM_{0,s+1}).
\]

Recall that the special vertex is decorated by $W_{n_v} \otimes \Delta$, and $\Delta = H^{11}(\MM_{1,11})$ is a $\Q$-vector space of rank 2.  This vector space does not have a canonical basis, but it does have a canonical Hodge structure of weight 11. Its complexification $\Delta \otimes \CC$ splits canonically as
\[
\Delta \otimes \CC \cong \Delta^{11,0} \oplus \Delta^{0,11},
\]
where $\Delta^{0,11}$ is the complex conjugate of $\Delta^{11,0}$, and $\Delta^{11,0}$ is spanned by a canonical element $\omega$ corresponding to the weight 12 cusp form for $\mathrm{SL}_2(\ZZ)$.  The pullback maps $\xi$ and $\eta$ respect complex conjugation. 
We can then decompose $\bGK_{g,n}^{11}$ as a tensor product
\[
\bGK_{g,n}^{11} \cong \bGK^{11, \circ}_{g,n} \otimes \Delta
\]
where $\bGK^{11,\circ}_{g,n}$ is a simpler and more combinatorial complex in which the special vertex is decorated by $W_{n_v}$. To make the differential on $\bGK^{11,\circ}_{g,n}$ explicit, we recall the description of generators, relations, and boundary pullback maps for $H^{11}(\MM_{g,n})$, from \cite[\S2]{CanningLarsonPayne}.

The symmetric group $\bS_{11}$ acts by the sign representation on $H^{11,0}(\MM_{1,11})$.  For $n > 11$, $H^{11,0}(\MM_{1,n})$ is generated by the pullbacks 
\[
\omega_A := \iota_A^* \omega    
\]
of the canonical generator $\omega$ of $H^{11,0}(\MM_{1,11})$ under the forgetful maps 
\[
    \iota_A\colon  \MM_{1,n}\to \MM_{1,11},
\]
given by forgetting all punctures except those in the set $A\subset \{1,\dots,n\}$ of cardinality 11.  Moreover, the pullbacks $\{\omega_A : 1 \in A\}$ form a basis.

Let $\xi_C\colon\MM_{1,B\cup \{p\}}\times \MM_{0,C\cup\{q\}}\to \MM_{1,B\cup C}$ be the boundary inclusion and let
\[
    \xi_C^*\colon  H^{11}(\MM_{1,B\cup C})
    \to 
    H^{11}(\MM_{1,B\cup \{p\}})\otimes H^{0}(\MM_{0,C\cup\{q\}})
    \cong H^{11}(\MM_{1,B\cup \{p\}})
\]
be the corresponding pullback operation.
Then
\begin{equation}
\label{equ:xistar}
\xi_C^* \omega_A = 
\begin{cases}
    \omega_A & \text{if $C\cap A=\emptyset $} \\
    \omega_{(A\setminus c)\cup p} & \text{if $C\cap A=\{c\}$}. \\
    0 & \text{otherwise.}
\end{cases}
\end{equation}

The differential on $\bGK_{g,n}^{11,\circ}$ hence has the form 
\[
\delta = \delta_s^\bullet + \delta_s^\circ,  
\]
with $\delta_s^\bullet$ splitting non-special vertices and $\delta_s^\circ$ splitting the special vertex. Concretely, we have 
\begin{equation}\label{equ:delta bullet def bGK}
  \delta_s^\bullet (\Gamma, e_1\wedge\cdots  \wedge e_k,x)
  =
  \sum_{v\in V_\bullet\Gamma} \sum_{\text{split v}} (\Gamma', e_0\wedge e_1\wedge\cdots  \wedge e_k ,x)    
  \end{equation}
with the outer sum running over non-special vertices of $\Gamma$.
The inner sum is over all admissible ways of replacing the vertex $v$ by two vertices connected by a new edge $e_0$, distributing the incident half-edges at $v$ on the new vertices, thus forming a graph $\Gamma'$. Pictorially:
\begin{align*}
    \delta_s^\bullet:
    \begin{tikzpicture}[baseline=-.65ex]
        \node[int] (v) at (0,0) {};
        \draw (v) edge +(-.3,-.3)  edge +(-.3,0) edge +(-.3,.3) edge +(.3,-.3)  edge +(.3,0) edge +(.3,.3);
        \end{tikzpicture}
        &\mapsto\sum
        \begin{tikzpicture}[baseline=-.65ex]
        \node[int] (v) at (0,0) {};
        \node[int] (w) at (0.5,0) {};
        \draw (v) edge (w) (v) edge +(-.3,-.3)  edge +(-.3,0) edge +(-.3,.3)
         (w) edge +(.3,-.3)  edge +(.3,0) edge +(.3,.3);
        \end{tikzpicture}        
\end{align*}
Similarly, the operation $\delta_s^\circ$ splits the special vertex,
\begin{equation}\label{equ:delta circ def bGK}
\delta_s^\circ (\Gamma, e_1\wedge\cdots  \wedge e_k, x)
=
\sum_{B\subset H_*, |B|\geq 2} (\text{split}_B\Gamma, e_0\wedge e_1\wedge\cdots  \wedge e_k, \xi_B^* x),    
\end{equation}
where the sum is running over subsets $B$ of the set of half-edges at the special vertex, and $\text{split}_B\Gamma$ is the graph obtained by adding an additional non-special vertex to the graph, to which we connect the half-edges in $B$, and a new edge to the special vertex.
Pictorially:
\begin{align*}
    \delta_s^\circ:
\begin{tikzpicture}
        \node[ext,accepting, label=90:{$\scriptstyle x$}] (v) at (0,0) {};
        \draw (v) edge +(-.3,-.3)  edge +(-.3,0) edge +(-.3,.3) edge +(.3,-.3)  edge +(.3,0) edge +(.3,.3);
        \end{tikzpicture}
        &\mapsto\sum_B
        \begin{tikzpicture}[baseline=-.65ex]
        \node[ext,accepting, label=90:{$\scriptstyle \xi_B^*x$}] (v) at (0,0) {};
        \node[int] (w) at (0.5,0) {};
        \draw (v) edge (w) (v) edge +(-.3,-.3)  edge +(-.3,0) edge +(-.3,.3) 
         (w)   edge +(.3,0) edge +(.3,.3) edge +(.3,-.3);
         \draw [decorate,decoration = {calligraphic brace}, thick] (.9,.35) --  (.9,-.35);
         \node at (1.2,0) {$\scriptstyle B$};
        \end{tikzpicture}
\end{align*}
To compute the pullback for the decoration at the special vertex one uses \eqref{equ:xistar}. In the definitions of both $\delta_s^\bullet$ and $\delta_s^\circ$, the newly added edge $e_0$ comes first in the edge ordering, and the relative order of the other edges is preserved.

\begin{remark}\label{rem:symdecomp}
Note that the formal linear combinations of expressions $\omega_A$ for $A\subset \{1,\dots,n\}$ with $|A|=11$ form a representation of the symmetric group $\bbS_n$ of the form 
\[
\mathrm{Ind}_{\bbS_{11}\times \bbS_{n-11}}^{\bbS_n} \sgn_{11} \otimes \Q ,
\]
i.e., the induced representation from the product of the sign representation $\sgn_{11}\cong V_{1^{11}}$ of $\bbS_{11}$ and the trivial representation $\Q\cong V_{n-11}$ of $\bbS_{n-11}$.
By Pieri's rule (or the more general Littlewood-Richardson rule) this representation decomposes into irreducibles as 
\[
    V_{(n-10)1^{10}}\oplus V_{(n-11)1^{11}}.
\]

The image of the subspace $V_{(n-11)1^{11}}$ in $H^{11,0}(\MM_{1,n})$ is zero, so  $H^{11,0}(\MM_{1,n})\cong V_{(n-10)1^{10}}$.
A complete set of relations spanning $V_{(n-11)1^{11}}$ is
\begin{equation} \label{eq:12rels}
 \sum_{j=1}^{12} (-1)^{j+1} \omega_{\{b_1,\dots,\hat b_j,\dots,b_{12}\}},
\end{equation}
with $B=\{b_1,\dots,b_{12}\}\subset \{1,\dots,n\}$ running over subsets of cardinality $12$.
\end{remark}

\subsection{An acyclic auxiliary graph complex}
We now describe an auxiliary graph complex $X_{g,n}$ in which each generator has a special vertex decorated by an arbitrary subset of its incident half-edges (not necessarily of size 11).  We include an ordering of these half-edges as part of the orientation, so permuting these half-edges induces a sign representation, consistent with the antisymmetric properties of the generators $\omega_A$ discussed above.  The resulting graph complex is acyclic, and hence gives rise to two resolutions of $\bGK_{g,n}^{11,\circ}$, by truncating according to the number of marked half-edges at the special vertex.  

The generators of $X_{g,n}$ are of the following form:
\begin{itemize}
    \item A connected graph $\Gamma$ of loop order $g - 1$, with one distinguished special vertex, a distinguished subset of $r$ half-edges at the special vertex, and $n$ numbered legs. 
    \item All non-special vertices have valence at least 3, the special vertex has valence at least one, and there are no tadpoles at the non-special vertices.
    \item  The cohomological degree of a generator is $\# \text{structural edges} - \# \text{marked half-edges}$.  
    \item The graph is equipped with an orientation $o$ consisting of a total ordering of the set $\{e_1,\dots,e_k, h_1,\dots,h_r\}$ of structural edges $e_1,\dots,e_k$ of $\Gamma$ and the distinguished subset of half-edges incident to the special vertex $h_1,\dots,h_r$. 
\end{itemize}    
Here, again, the structural edges are all edges other than the numbered legs.     
    We suggestively write $o=a_1\wedge\cdots \wedge a_{k+r}$ to indicate the order on the set of edges and the distinguished half-edges, with $a_1,\dots , a_{k+r}$ being some ordering of the elements of the set $\{e_1,\dots,e_k, h_1,\dots,h_r\}$. 
   We impose two relations. 
   
   \begin{itemize}
      \item First, we identify isomorphic graphs: if $\phi\colon \Gamma\to \Gamma'$ is an isomorphism, we set 
   \begin{equation}\label{equ:X gen rel 2}
    (\Gamma, e_1\wedge \cdots \wedge e_k \wedge h_1\wedge \cdots \wedge h_r) 
    =
    (\Gamma',\phi(e_1)\wedge \cdots \wedge \phi(e_k) \wedge \phi(h_1)\wedge \cdots \wedge \phi(h_r)).
   \end{equation}
    \item Second, we identify orderings up to sign: for a
    permutation $\sigma\in \bbS_{k+r}$ we set 
    \begin{equation}\label{equ:X gen rel 1}
        (\Gamma, a_1\wedge\cdots \wedge a_{k+r})
        =
        \sgn(\sigma) 
        (\Gamma, a_{\sigma(1)}\wedge\cdots \wedge a_{\sigma(k+r)} )\,.
    \end{equation}
        \end{itemize}
    This second relation allows one to put the edges before the half-edges, $e_1\wedge\cdots  \wedge e_k \wedge h_1\wedge \cdots \wedge h_r$. 

The following figure depicts a generator for $X_{5,1}$. The special vertex is indicated by a double ring, and
the marked half-edges at the special vertex are indicated by arrows:
\[
  \begin{tikzpicture}[scale=1]
    \node[ext,accepting] (v1) at (0,0){};
    \node[int] (v2) at (180:1){};
    \node[int] (v3) at (60:1){};
    \node[int] (v4) at (-60:1){};
    \draw (v1) 
    edge[->-](v2)  edge[->-] (v3) edge (v4)  edge[loop right] (v2)
    (v2) edge[bend left] (v3) edge[bend right] (v4)  -- +(180:1.3) 
    (v3) edge (v4);
    \node (w) at (180:2.5) {$1$};
    \end{tikzpicture} \quad .
\]

The differential on $X_{g,n}$ is the sum of two pieces 
\[
\delta=\delta_s +\delta_\omega.    
\]
The piece $\delta_\omega$ simply removes one distinguished half-edge from the orientation.
\[
    \delta_\omega (\Gamma, e_1\wedge\cdots \wedge e_k \wedge h_1\wedge \cdots \wedge h_r)
    =\sum_{j=1}^r (-1)^{k+j-1}
    (\Gamma, e_1\wedge \cdots \wedge e_k \wedge h_1 \wedge \cdots \hat h_j \cdots \wedge h_r)
\]
Pictorially:
\begin{align*}
    \delta_\omega:
    \begin{tikzpicture}
        \node[ext,accepting] (v) at (0,0){};
        \draw (v) edge[->-] +(0:.6) edge[->-] +(60:.6) edge +(120:.6) edge +(180:.6) edge +(240:.6) edge[->-] +(300:.6);
    \end{tikzpicture}
    &\mapsto
    \sum\pm
    \begin{tikzpicture}
        \node[ext,accepting] (v) at (0,0){};
        \draw (v) edge[->-] +(0:.6) edge +(60:.6) edge +(120:.6) edge +(180:.6) edge +(240:.6) edge[->-] +(300:.6);
    \end{tikzpicture} \quad .
\end{align*}
The piece $\delta_s$ acts by splitting vertices.
For convenience we shall further decompose $\delta_s=\delta_s^\bullet+\delta_s^\circ$ into a piece $\delta_s^\circ$ splitting the special vertex and $\delta_s^\bullet$ splitting the other vertices. Concretely, the operation $\delta_s^\bullet$ is defined analogously to \eqref{equ:delta bullet def bGK},
\begin{equation}\label{equ:delta bullet def}
\delta_s^\bullet (\Gamma, e_1\wedge\cdots \wedge e_k \wedge h_1\wedge \cdots \wedge h_r)
=
\sum_{v\in V_\bullet\Gamma} \sum_{\text{split v}} (\Gamma', e_0\wedge e_1\wedge\cdots \wedge e_k \wedge h_1\wedge \cdots\wedge h_r)    \, .
\end{equation}
The outer sum is again over all non-special vertices $v$ of $\Gamma$. The inner sum is over all admissible ways of replacing the vertex $v$ by two vertices connected by an edge, distributing the incident half-edges at $v$ on the new vertices, thus forming a graph $\Gamma'$. Pictorially:
\begin{align*}
    \delta_s^\bullet:
    \begin{tikzpicture}[baseline=-.65ex]
        \node[int] (v) at (0,0) {};
        \draw (v) edge +(-.3,-.3)  edge +(-.3,0) edge +(-.3,.3) edge +(.3,-.3)  edge +(.3,0) edge +(.3,.3);
        \end{tikzpicture}
        &\mapsto\sum
        \begin{tikzpicture}[baseline=-.65ex]
        \node[int] (v) at (0,0) {};
        \node[int] (w) at (0.5,0) {};
        \draw (v) edge (w) (v) edge +(-.3,-.3)  edge +(-.3,0) edge +(-.3,.3)
         (w) edge +(.3,-.3)  edge +(.3,0) edge +(.3,.3);
        \end{tikzpicture}        \quad .
\end{align*}
Similarly, we define $\delta_s^\circ$ analogously to \eqref{equ:delta circ def bGK},
\begin{equation}\label{equ:delta circ def}
\delta_s^\circ (\Gamma, e_1\wedge\cdots \wedge e_k \wedge h_1\wedge \cdots \wedge h_r)
=
\sum_{B\subset H_* \atop |B|\geq 2} (\text{split}_B\Gamma, e_0\wedge e_1\wedge\cdots \wedge e_k \wedge h_1\wedge \cdots \wedge h_r)   \, , 
\end{equation}
with the sum running over subsets $B$ of the set $H_*$ of half-edges incident at the special vertex, such that $|B|\geq 2$ and $B$ contains at most one of the distinguished half-edges.
The graph $\text{split}_B\Gamma$ is built by adding a new non-special vertex $v$ to the graph $\Gamma$, with an edge to the special vertex, and reconnecting the half-edges $B$ to $v$.
If $B$ contains a marked half-edge, then the marking is removed and put on the half-edge connecting the special vertex to $v$ instead. 
Pictorially:
\begin{align*}
    \delta_s^\circ:
\begin{tikzpicture}
        \node[ext,accepting] (v) at (0,0) {};
        \draw (v) edge[->-] +(-.3,-.3)  edge[->-] +(-.3,0) edge +(-.3,.3) edge[->-] +(.3,-.3)  edge +(.3,0) edge +(.3,.3);
        \end{tikzpicture}
        &\mapsto\sum
        \begin{tikzpicture}[baseline=-.65ex]
        \node[ext,accepting] (v) at (0,0) {};
        \node[int] (w) at (0.5,0) {};
        \draw (v) edge (w) (v) edge[->-] +(-.3,-.3)  edge[->-] +(-.3,0) edge +(-.3,.3) edge[->-] +(.3,-.3)
         (w)   edge +(.3,0) edge +(.3,.3);
        \end{tikzpicture}
        +\sum 
        \begin{tikzpicture}[baseline=-.65ex]
        \node[ext,accepting] (v) at (0,0) {};
        \node[int] (w) at (0.5,0) {};
        \draw (v) edge[->-] (w) (v) edge[->-] +(-.3,-.3)  edge[->-] +(-.3,0) edge +(-.3,.3)
            (w) edge +(.3,-.3)  edge +(.3,0) edge +(.3,.3);
        \end{tikzpicture} \quad .
\end{align*}

\begin{lemma}\label{lem:X well defined}
The differential $\delta$ satisfies $\delta^2=0$, and the dg vector space $(X_{g,n}, \delta)$ is acyclic.
\end{lemma}
\begin{proof}
The verification that $\delta^2 = 0$ is by direct computation, as follows. 
Expand $\delta^2$ as 
\[
  (\delta_s^\bullet+\delta_s^\circ +\delta_\omega)^2
  =
  (\delta_s^\bullet)^2 + (\delta_s^\circ)^2
  + (\delta_\omega)^2
  + [\delta_s^\bullet, \delta_s^\circ]
  + [\delta_s^\bullet, \delta_\omega]
  + [\delta_s^\circ, \delta_\omega],
\]
with $[-,-]$ denoting the anticommutator.

It is clear that $(\delta_\omega)^2=0$ since the operations of removing two different markings commute, and the terms come with opposite sign.
Similarly, $[\delta_s^\bullet, \delta_\omega]=0$ since the markings at the special vertex do not interfere with the edge splitting operation, and matching terms again come with opposite sign.

To check that $(\delta_s^\bullet)^2=0$ consider a graph $\Gamma\in X$ and compute $\delta_s^\bullet\delta_s^\bullet \Gamma$. 
Say the first application of $\delta_s^\bullet$ splits a vertex $v$ into vertices $v',v''$, and the second application split a vertex $w$. Clearly, if $w\neq v',v''$ then the splittings of first $v$ then $w$ cancels the similar term corresponding to first splitting $w$ and then $v$.
There remain the terms for which $w=v'$ and $w=v''$, schematically depicted as follows:
\[
  \begin{tikzpicture}
  \node[int,label=90:{$v$}] (v) at (0,0) {};
  \draw (v) edge +(-.5,.5) edge +(-.5,0) edge +(-.5,-.5) edge +(.5,.5) edge +(.5,0) edge +(.5,-.5);
  \end{tikzpicture}
  \, \xrightarrow{\delta_s^\bullet}\, 
  \sum 
  \begin{tikzpicture}
    \node[int,label=90:{$v'$}] (v) at (-.5,0) {};
    \node[int,label=90:{$v''$}] (vv) at (.5,0) {};
    \draw (v) edge +(-.5,.5) edge +(-.5,0) edge +(-.5,-.5) edge node[below] {$\scriptstyle 1$} (vv)
    (vv) edge +(.5,.5) edge +(.5,0) edge +(.5,-.5);
    \end{tikzpicture}
    \,  \xrightarrow{\delta_s^\bullet}\, 
  \sum 
    \begin{tikzpicture}
      \node[int] (v) at (-.5,0) {};
      \node[int] (vv) at (.5,0) {};
      \node[int] (x) at (-.1,0) {};
      \draw (v)  edge +(-.5,0) edge +(-.5,-.5) edge node[below] {$\scriptstyle 1$} (x)
      (x) edge +(-.5,.5)
      (vv) edge node[below] {$\scriptstyle 2$} (x) edge +(.5,.5) edge +(.5,0) edge +(.5,-.5);
    \end{tikzpicture}
    +
    \begin{tikzpicture}
      \node[int] (v) at (-.5,0) {};
      \node[int] (vv) at (.5,0) {};
      \node[int] (x) at (.1,0) {};
      \draw (v)  edge +(-.5,0) edge +(-.5,-.5) edge node[below] {$\scriptstyle 2$} (x) edge +(-.5,.5)
      (x) edge +(.5,.5)
      (vv) edge node[below] {$\scriptstyle 1$} (x) edge +(.5,0) edge +(.5,-.5);
    \end{tikzpicture} \quad .
\] 
The numbers below the edges indicate the position of the edge in the ordering that makes up the orientation of the graph. The two terms on the right (from splitting $v'$ and $v''$) are the same, up to the sign from swapping the edge order, and hence cancel.

By essentially the same argument, with one of the vertices replaced by the special vertex,
\[
  (\delta_s^\circ)^2+[\delta_s^\bullet, \delta_s^\circ]=0.
\]
It remains to check that $[\delta_s^\circ, \delta_\omega]=0$.
To this end fix a graph $\Gamma\in X$ and look at those terms in $[\delta_s^\circ, \delta_\omega]X$ in which a subset $B$ of the half-edges at the special vertex is split off, and the marking of half-edge $h$ is removed. Such terms can potentially be produced twice, corresponding to the two paths from top left to bottom right in the following diagram:
\[
  \begin{tikzcd}
  \left.
  \begin{tikzpicture}
    \node[ext, accepting] (v) at (0,0) {};
    \draw (v) edge +(-.5,.5) edge[->-] +(-.5,0) edge +(-.5,-.5) edge +(.5,.5) edge +(.5,0) edge +(.5,-.5);
    \node at (-.7,0) {$\scriptstyle h$};
    \end{tikzpicture}
    \right\} B
    \ar{r}{\text{split}_B}
    \ar{d}{\text{remove }h}
 &
\left.
    \begin{tikzpicture}
      \node[ext, accepting] (v) at (-.5,0) {};
      \node[int] (vv) at (.5,0) {};
      \draw (v) edge +(-.5,.5) edge[->-] +(-.5,0) edge +(-.5,-.5) edge (vv)
      (vv) edge +(.5,.5) edge +(.5,0) edge +(.5,-.5);
      \end{tikzpicture}
      \right\} B
  \ar{d}{\text{remove }h}
\\
    \left.
    \begin{tikzpicture}
      \node[ext, accepting] (v) at (0,0) {};
      \draw (v) edge +(-.5,.5) edge +(-.5,0) edge +(-.5,-.5) edge +(.5,.5) edge +(.5,0) edge +(.5,-.5);
      \end{tikzpicture}
      \right\} B
      \ar{r}{\text{split}_B}
   &
  \left.
      \begin{tikzpicture}
        \node[ext, accepting] (v) at (-.5,0) {};
        \node[int] (vv) at (.5,0) {};
        \draw (v) edge +(-.5,.5) edge +(-.5,0) edge +(-.5,-.5) edge (vv)
        (vv) edge +(.5,.5) edge +(.5,0) edge +(.5,-.5);
        \end{tikzpicture}
        \right\} B
    \end{tikzcd} \quad .
    \]
We distinguish 4 cases.
\begin{itemize}
\item If the half-edge $h$ is not in $B$, as in the picture above, the two terms produced from $\delta_s^\circ \delta_\omega \Gamma$ and $\delta_\omega \delta_s^\circ \Gamma$ are isomorphic but come with opposite signs and hence cancel.
\item If $h\in B$, and $B$ contains exactly one marked half-edge, then the two terms again cancel:
\[
  \begin{tikzcd}
  \left.
  \begin{tikzpicture}
    \node[ext, accepting] (v) at (0,0) {};
    \draw (v) edge +(-.5,.5) edge +(-.5,0) edge +(-.5,-.5) edge +(.5,.5) edge[->-] +(.5,0) edge +(.5,-.5);
    \end{tikzpicture}
    \right\} B
    \ar{r}{\text{split}_B}
    \ar{d}{\text{remove }h}
 &
\left.
    \begin{tikzpicture}
      \node[ext, accepting] (v) at (-.5,0) {};
      \node[int] (vv) at (.5,0) {};
      \draw (v) edge +(-.5,.5) edge +(-.5,0) edge +(-.5,-.5) edge[->-] (vv)
      (vv) edge +(.5,.5) edge +(.5,0) edge +(.5,-.5);
      \end{tikzpicture}
      \right\} B
  \ar{d}{\text{remove }h}
\\
    \left.
    \begin{tikzpicture}
      \node[ext, accepting] (v) at (0,0) {};
      \draw (v) edge +(-.5,.5) edge +(-.5,0) edge +(-.5,-.5) edge +(.5,.5) edge +(.5,0) edge +(.5,-.5);
      \end{tikzpicture}
      \right\} B
      \ar{r}{\text{split}_B}
   &
  \left.
      \begin{tikzpicture}
        \node[ext, accepting] (v) at (-.5,0) {};
        \node[int] (vv) at (.5,0) {};
        \draw (v) edge +(-.5,.5) edge +(-.5,0) edge +(-.5,-.5) edge (vv)
        (vv) edge +(.5,.5) edge +(.5,0) edge +(.5,-.5);
        \end{tikzpicture}
        \right\} B
    \end{tikzcd} \quad .
    \]
\item If $h\in B$ and $B$ contains at least $3$ marked half-edges, then none of the splitting terms yields a contribution (i.e., those $B$ do not appear in the sum \eqref{equ:delta circ def}).
\item Finally, if $h\in B$ and $B$ contains exactly two marked half-edges, say $h$ and $h'$, then the terms corresponding to removing $h$ and $h'$ match and cancel:
\[
  \begin{tikzcd}
  \left.
  \begin{tikzpicture}
    \node[ext, accepting] (v) at (0,0) {};
    \draw (v) edge +(-.5,.5) edge +(-.5,0) edge +(-.5,-.5) edge[->-] +(.5,.5) edge[->-] +(.5,0) edge +(.5,-.5);
    \end{tikzpicture}
    \right\} B
    \ar{r}{\text{split}_B}
    \ar{d}{\text{remove }h,h'}
 &
0
  \ar{d}
\\
    \begin{tikzpicture}
      \node[ext, accepting] (v) at (0,0) {};
      \draw (v) edge +(-.5,.5) edge +(-.5,0) edge +(-.5,-.5) edge +(.5,.5) edge[->-] +(.5,0) edge +(.5,-.5);
      \end{tikzpicture}
-
\begin{tikzpicture}
  \node[ext, accepting] (v) at (0,0) {};
  \draw (v) edge +(-.5,.5) edge +(-.5,0) edge +(-.5,-.5) edge[->-] +(.5,.5) edge +(.5,0) edge +(.5,-.5);
  \end{tikzpicture}
      \ar{r}{\text{split}_B}
   &
   (+1-1) \cdot 
      \begin{tikzpicture}
        \node[ext, accepting] (v) at (-.5,0) {};
        \node[int] (vv) at (.5,0) {};
        \draw (v) edge +(-.5,.5) edge +(-.5,0) edge +(-.5,-.5) edge[->-] (vv)
        (vv) edge +(.5,.5) edge +(.5,0) edge +(.5,-.5);
        \end{tikzpicture}=0
    \end{tikzcd} \quad .
    \]
\end{itemize}

To show acyclicity we consider the filtration on $X$ by the number of vertices.
The $E^0$-page of the associated spectral sequence may be identified with the complex $(X_{g,n},\delta_\omega)$.
Hence it suffices to check that $H(X_{g,n},\delta_\omega)=0$. 
To this end consider the degree $-1$-operation $h \colon X_{g,n}\to X_{g,n}$ that sums over all ways of adding a half-edge to the distinguished set, 
\[
\resizebox{.97\hsize}{!}{
$\displaystyle{
h(\Gamma, e_1\wedge\cdots \wedge e_k \wedge h_1\wedge \cdots \wedge h_r)
=
\frac1{|H_*|}
(-1)^k
\sum_{h\in H_*\setminus\{h_1, \dots, h_r\} } 
h(\Gamma, e_1\wedge\cdots \wedge e_k \wedge h\wedge h_1\wedge \cdots \wedge h_r).
}$}
\]
Then $\delta_\omega h +h\delta_\omega$ is the identity map. It follows that $h$ is a contracting homotopy for $\delta_\omega$, and $H(X_{g,n},\delta_\omega)=0$.
Since $X_{g,n}$ is finite dimensional, the filtration is bounded and our spectral sequence converges to the cohomology. Hence $H(X_{g,n},\delta_s+\delta_\omega)=0$ as claimed.
\end{proof}

\subsection{Blown-up picture}
We now introduce an alternative graphical depiction of generators for $X_{g,n}$ that we call the blown-up picture. This equivalent encoding of the same information is obtained by ``blowing-up" the special vertex, i.e., removing the special vertex and making the incident half-edges into external legs, which we label by a special symbol $\omega$ (resp. $\epsilon$) according to whether the half-edges are marked or not.
For example
\[
    \begin{tikzpicture}[scale=1]
        \node[ext,accepting] (v1) at (0,0){};
        \node[int] (v2) at (180:1){};
        \node[int] (v3) at (60:1){};
        \node[int] (v4) at (-60:1){};
        \draw (v1) 
        edge[->-](v2)  edge[->-] (v3) edge (v4)  edge[loop right] (v2)
        (v2) edge[bend left] (v3) edge[bend right] (v4)  -- +(180:1.3) 
        (v3) edge (v4);
        \node (w) at (180:2.5) {$1$};
        \end{tikzpicture}
\quad \quad \mapsto \quad \quad 
\begin{tikzpicture}[scale=1, node distance=.8]
    \node[int] (v2) at (180:1){};
    \node[int] (v3) at (60:1){};
    \node[int] (v4) at (-60:1){};
    \node[below=of v2] (e1) {$\omega$};
    \node[right=of v3] (e2) {$\omega$};
    \node[right=of v4] (e3) {$\epsilon$};
    \node (e4) at (2.5,0) {$\epsilon$};
    \node[right=of e4] (e5) {$\epsilon$};
    \draw
    (v2) edge (e1) edge[bend left] (v3) edge[bend right] (v4)  -- +(180:1.3) 
    (v3) edge (e2) edge (v4)
    (v4) edge (e3)
    (e4) edge (e5);
    \node (w) at (180:2.5) {$1$};
    \end{tikzpicture}\, .
\]
Note that the graph on the right may be disconnected, though every connected component must contain at least one $\epsilon$- or $\omega$-leg.
When we talk about the blown-up components of a graph in $X_{g,n}$ below, we refer to the connected components of the blown-up picture. 

The differential has an equivalent description in the blown-up picture, as follows. The piece 
\begin{align*}
\delta_\omega:
    \begin{tikzpicture}
        \node (v) at (0,0) {$\omega$};
        \node[ext] (w) at (1,0) {$\cdots$};
        \draw (v) edge (w);
    \end{tikzpicture}
    &\mapsto 
    \begin{tikzpicture}
        \node (v) at (0,0) {$\epsilon$};
        \node[ext] (w) at (1,0) {$\cdots$};
        \draw (v) edge (w);
    \end{tikzpicture}
\end{align*}
replaces one $\omega$-decoration by $\epsilon$. The piece $\delta_{s}^\circ$ joins together a subset $S$ of the $\epsilon$- and $\omega$-legs, containing at most one $\omega$-leg, and attaches a new leg that is decorated by $\omega$ if $S$ contains an $\omega$ leg and $\epsilon$ otherwise: 
\begin{align*}
  \delta_{s}^\circ 
 \begin{tikzpicture}[baseline=-.8ex]
 \node[draw,circle] (v) at (0,.3) {$\Gamma$};
 \node (w1) at (-.7,-.5) {};
 \node (w2) at (-.25,-.5) {};
 \node (w3) at (.25,-.5) {};
 \node (w4) at (.7,-.5) {};
 \draw (v) edge (w1) edge (w2) edge (w3) edge (w4);
 \end{tikzpicture} 
 = 
 \sum_{B} 
 \begin{tikzpicture}[baseline=-.8ex]
 \node[draw,circle] (v) at (0,.3) {$\Gamma$};
 \node (w1) at (-.7,-.5) {};
 \node (w2) at (-.25,-.5) {};
 \node[int] (i) at (.4,-.5) {};
 \node (w4) at (.4,-1.3) {$\epsilon$ or $\omega$};
 \draw (v) edge (w1) edge (w2) edge[bend left] (i) edge (i) edge[bend right] (i) (w4) edge (i);
 \end{tikzpicture} \, .
 \end{align*}

\subsection{Truncations of $X_{g,n}$ and resolutions of $\bGK^{11,\circ}_{g,n}$} \label{sec:truncation}
We now show that two truncations of $X_{g,n}$ give natural resolutions of $\bGK^{11,\circ}_{g,n}$.
\begin{definition}
Let $\widetilde C_{g,n}\subset X_{g,n}$ be the dg subspace spanned by graphs that have at most $10$ distinguished half-edges at the special vertex, and define
\[
\widetilde B_{g,n}:= X_{g,n}/\widetilde C_{g,n}.    
\]
to be the quotient complex. We also denote appropriate degree shifted versions by 
\begin{align*}
    B_{g,n}&=\widetilde B_{g,n}[-22] \\
    C_{g,n}&=\widetilde C_{g,n}[-21].
\end{align*}
\end{definition}
\noindent We now show that $B_{g,n}$ and $C_{g,n}$ are resolutions of $\bGK^{11,\circ}_{g,n}$.
Consider $P \colon B_{g,n} \to  \bGK^{11,\circ}_{g,n}$, given by
\[
P(\Gamma,e_1\wedge\cdots\wedge e_k\wedge h_1\wedge \cdots \wedge h_r) =
\begin{cases}
(\Gamma, e_1\wedge\cdots\wedge e_k, \omega_{\{h_1,\dots,h_{11}\}}) & \text{if $r=11$,}\\
0 & \text{otherwise;} 
\end{cases}
\]
and  $ I \colon \bGK^{11,\circ}_{g,n} \to C_{g,n}$, given by
\[
    I(\Gamma, e_1\wedge\cdots\wedge e_k, \omega_{\{h_1,\dots,h_{11}\}})
    =
    \sum_{j=1}^{11}
    (-1)^{k+j-1}
    (\Gamma,e_1\wedge\cdots\wedge e_k\wedge h_1\wedge \cdots \hat h_j \cdots\wedge h_{11}).
\]

\begin{prop}
The maps $P$ and $I$ above are well-defined maps of dg vector spaces and induce isomorphisms on cohomology.
\end{prop}
\begin{proof}
We start by checking that $P$ intertwines the differentials.  
First, we show that $P(\delta_\omega\Gamma) = 0$. If $\Gamma$ has 11 or more than 12 distinguished half-edges then this is clear by degree reasons. If $\Gamma$ has exactly 12 distinguished half-edges then $P(\delta_\omega\Gamma) = 0$ by Remark \ref{rem:symdecomp}, since $\delta_\omega$ produces expressions of the form \eqref{eq:12rels}, which vanish in $\bGK^{11,\circ}_{g,n}$.
Next, we have $P(\delta_s^\bullet\Gamma)=\delta_s^\bullet P(\Gamma)$, since the splitting of the non-special vertices is the same on both sides; see \eqref{equ:delta bullet def bGK} and \eqref{equ:delta bullet def}.
Finally, comparing \eqref{equ:delta circ def bGK} and \eqref{equ:delta circ def} we see that $P(\delta_s^\circ\Gamma)=\delta_s^\circ(\Gamma)$, because the handling of the marked half-edges in $\delta_s^\circ$ on $B_{g,n}$ reflects the pullback operation $\xi_B^*$ of \eqref{equ:xistar}.  It follows that $P$ intertwines the differentials, as required.

Next, note that the maps $P$ and $I$ fit into a commutative diagram 
\[
\begin{tikzcd}    
B_{g,n} \ar{r}{P} \ar{dr}[below]{\delta_\omega} & \bGK^{11,\circ}_{g,n} \ar{d}{I}\\
& C_{g,n} 
\end{tikzcd}.
\] 
Since $P$ is surjective and both $P$ and $\delta_\omega$ intertwine the differentials, so does $I$.

It remains to show that $P$ and $I$ are quasi-isomorphisms. Since $X_{g,n}$ is acyclic, $\delta_\omega \colon B_{g,n}\to C_{g,n}$ is a quasi-isomorphism. Since the above diagram commutes, it therefore suffices to show that $P$ is a quasi-isomorphism.  
To this end we first note that by Remark \ref{rem:symdecomp}
\[
  \bGK^{11,\circ} \cong B_{g,n}/(B_{g,n,\geq 12\omega} \oplus \delta_\omega B_{g,n,12\omega} ),
\]
where $B_{g,n,12\omega}$ (resp. $B_{g,n,\geq 12\omega}$) is the subspace of $B_{g,n}$ spanned by graphs with 12 (resp. $\geq 12$) $\omega$-legs.
We hence need to show that the projection 
\begin{equation}\label{equ:P on B}
  B_{g,n} \to B_{g,n}/(B_{g,n,\geq 12\omega} \oplus \delta_\omega B_{g,n,12\omega} )
\end{equation}
is a quasi-isomorphism. To this end we follow the argument for acyclicity of $X_{g,n}$ in the proof of Lemma~\ref{lem:X well defined}.
We consider on both sides of \eqref{equ:P on B} the spectral sequences from the filtration by the numbers of vertices in graphs.
On the first page of the spectral sequence, the differential on the left-hand side of \eqref{equ:P on B} is given by $\delta_\omega$, and on the right-hand side it is zero.
Since $(X_{g,n},\delta_\omega)$ is acyclic and $B_{g,n}$ is the a truncation at 11 $\omega$-legs, we have
\[
  H(B_{g,n},\delta_\omega) 
  =
  B_{g,n,11\omega}
  /
  \mathrm{im}(B_{g,n,12\omega}\xrightarrow{\delta_\omega} B_{g,n,1\omega})
  \cong B_{g,n}/(B_{g,n,\geq 12\omega} \oplus \delta_\omega B_{g,n,12\omega} ).
\]
Hence \eqref{equ:P on B} induces an isomorphism on the $E^1$-page of the spectral sequence, and is a quasi-isomorphism by the spectral sequence comparison lemma.
\end{proof}

\begin{corollary}
There are natural isomorphisms
\[
H^k(B_{g,n}) \otimes \Delta \cong \gr_{11}H^k_c(\M_{g,n}) \cong H^k(C_{g,n}) \otimes \Delta.
\]
\end{corollary}

\section{Explicit computations in low excess}\label{sec:low excess}
In order to study the cohomology of $B_{g,n}$, we introduce a statistic on graph generators that we call excess. Most importantly for our purposes, the excess is non-negative, additive on blown-up components, and graphs with small excess are relatively simple and easy to classify.

\begin{definition}
The excess of a generator $\Gamma$ of $X_{g,n}$ is 
\[
E(\Gamma) = 3(g-1) +2n - 2\# \omega,  
\] 
where $\# \omega$ is the number of $\omega$-legs of $\Gamma$.
\end{definition}

We also define
\[
E(g,n) := 3g +2n - 25.
\]
Any generator $\Gamma$ for $B_{g,n}$ has $\# \omega\geq 11$, and hence
\[
E(\Gamma) \leq E(g,n). 
\]
Note that $E(\Gamma) \equiv E(g,n) \mod 2$, so all generators for $B_{g,n}$ have even or odd excess, when $g$ is odd or even, respectively.  Also, for each fixed $k$, there are only finitely many pairs $(g,n)$ such that $E(g,n) = k$.

\medskip

Write each generator $\Gamma$ for $B_{g,n}$ as a union of its blown-up components:
\[
\Gamma = C_1\cup\cdots \cup C_k.
\]
Let $g_i$ be the contribution of $C_i$ to the genus of $\Gamma$.  More precisely,
\[
g_i = h^1(C_i) + \# \epsilon + \#\omega-1,
\]
i.e., the loop order of $C_i$ plus the number of its $\epsilon$ and $\omega$ labeled legs minus one.  Then the excess of $C_i$ is 
\[
E(C_i)
:= 3g_i +2n_i - 2\# \omega
\]

\begin{lemma}\label{lem:excess}
The excess is additive over blown-up components, i.e.
\begin{equation}\label{equ:E additive}
E(\Gamma) = E(C_1\cup\cdots \cup C_k)= E(C_1)+ \cdots + E(C_k),  
\end{equation}
and the excess of each blown-up component is nonnegative.  
\end{lemma}
\begin{proof}
The formula \eqref{equ:E additive} for $E(\Gamma)$ is evident since the genus of $\Gamma$ is the sum over the genus contributions of the blown-up components, plus one to take into account that the special vertex has genus one.

If either $h^1(C_i) \geq 1$ or $\#\epsilon\geq 1$ then 
\[
  E(C_i) = 3 h^1(C_i) + 3\#\epsilon + 2n_i +\# \omega - 3 \geq 0.
\]
It remains to show that $E(C_i) \geq 0$ when $h^1(g_i) =\#\epsilon=0$. Suppose $C_i$ is a tree with $m$ leaves that can be either numbered or $\omega$-decorated. 
If the tree has at least 3 leaves,  then $E(C_i)$ is at least $2n_i +\# \omega -3 \geq 0$. Any tree has at least two leaves, and the remaining cases are:
\[
  \begin{tikzpicture}
    \node (v) at (0,0) {$\omega$};
    \node (w) at (1,0) {$\omega$};
    \draw (v) edge (w);
  \end{tikzpicture}
\quad\text{or}\quad  
\begin{tikzpicture}
  \node (v) at (0,0) {$\omega$};
  \node (w) at (1,0) {$j$};
  \draw (v) edge (w);
\end{tikzpicture}\, .
\] 
The first graph vanishes by symmetry and for the second we have $E(C_i)=0$.
\end{proof}

\begin{cor} \label{cor:Eneg}
If $E(g,n) <0$ then $
\gr_{11} H^\bullet_c(\M_{g,n}) =0.$
\end{cor}

\begin{proof}
By Lemma~\ref{lem:excess}, if $E(g,n)$ is negative, the complex $B_{g,n}$ is $0$.
\end{proof}

\begin{lemma} \label{lem:excessloop}
If $h^1(C_i) \geq 1$ then $E(C_i)\geq 5$.
\end{lemma}
\begin{proof}
The argument is similar to the proof of Lemma~\ref{lem:excess}.  First, note that if $h^1(C_i) \geq 3$ then $E(C_i)\geq 5$.
If $h^1(C_i) = 2$, then the only graphs that would produce $E(C_i)< 5$ need to have $\#\epsilon=2n_i=0$ and $\# \omega=1$.
But there is no such (non-vanishing) loop order 2 graph.
Finally, suppose $h^1(C_i) = 1$. The general loop order one graph has the form 
\[
  \begin{tikzpicture}[baseline=-.65ex]
  \node (v0) at (0:1) {$\cdots$};
  \foreach \a [count=\ai, evaluate=\prev using int(\ai -1)] in {60,120,180,240,300}
    { 
      \node[int] (v\ai) at (\a:1) {};
      \node (w\ai) at (\a:1.5) {$?$};
      \draw (v\ai) edge (w\ai);
      \draw (v\ai) edge (v\prev);
    }
  \draw (v0) edge (v5);
  \end{tikzpicture}
\] 
with the $?$ representing an $\epsilon$, $\omega$, a numbered leg or a tree or forest to be attached.
Clearly, if the graph has $\geq 5$ legs then $E(C_i)\geq 5$.
Also note that the length of the inner loop must be at least three, otherwise the graph has a double edge and vanishes.
If the graph has loop length 4, the only case to be considered is that of all 4 legs being $\omega$-legs. 
\[
  \begin{tikzpicture}[baseline=-.65ex]
  \foreach \a [count=\ai, evaluate=\prev using int(\ai -1)] in {0,90,180,270}
    { 
      \node[int] (v\ai) at (\a:1) {};
      \node (w\ai) at (\a:1.5) {$\omega$};
      \draw (v\ai) edge (w\ai);
      \draw (v\ai) edge (v\prev);
    }
  \draw (v1) edge (v4);
  \end{tikzpicture}
\]
This graph has an odd symmetry and vanishes in $B_{g,n}$.
For loop length three we have the graph
\[
  \begin{tikzpicture}[baseline=-.65ex]
  \foreach \a [count=\ai, evaluate=\prev using int(\ai -1)] in {0,120,240}
    { 
      \node[int] (v\ai) at (\a:1) {};
      \node (w\ai) at (\a:1.5) {$\omega$};
      \draw (v\ai) edge (w\ai);
      \draw (v\ai) edge (v\prev);
    }
  \draw (v1) edge (v3);
  \end{tikzpicture},
\]
and its variants in which one $\omega$ is replaced by a number or a forest with 2 $\omega$-legs. In each case, the graph has an odd symmetry and vanishes in $B_{g,n}$.
\end{proof}

\noindent Using Lemma~\ref{lem:excessloop}, the cohomology of $B_{g,n}$ can be computed relatively easily as long as $E(g,n)\leq 4$, since the generating graphs are forests. We now carry through the details for $E(g,n)\leq 3$. 

\subsection{Excess 0}
By Lemma \ref{lem:excess}, the blown-up picture of a graph of excess zero is a union of connected components of excess zero.
The only such components are of the form 
\[
  \begin{tikzpicture}
    \node (v) at (0,0) {$\omega$};
    \node (w) at (1,0) {$j$};
    \draw (v) edge (w);
  \end{tikzpicture}
\quad\text{or}\quad  
\begin{tikzpicture}
  \node[int] (i) at (0,.5) {};
  \node (v1) at (-.5,-.2) {$\omega$};
  \node (v2) at (0,-.2) {$\omega$};
  \node (v3) at (.5,-.2) {$\omega$};
\draw (i) edge (v1) edge (v2) edge (v3);
\end{tikzpicture}\, .
\] 

Thus, if $E(g,n) = 0$, the generators of $B_{g,n}$ have the following form:
\[
  \Gamma^{(0)}=
  \begin{tikzpicture}
    \node (v) at (0,0) {$\omega$};
    \node (w) at (1,0) {$1$};
    \draw (v) edge (w);
  \end{tikzpicture}
  \cdots 
  \begin{tikzpicture}
    \node (v) at (0,0) {$\omega$};
    \node (w) at (1,0) {$n$};
    \draw (v) edge (w);
  \end{tikzpicture}
  \begin{tikzpicture}
    \node[int] (i) at (0,.5) {};
    \node (v1) at (-.5,-.2) {$\omega$};
    \node (v2) at (0,-.2) {$\omega$};
    \node (v3) at (.5,-.2) {$\omega$};
  \draw (i) edge (v1) edge (v2) edge (v3);
  \end{tikzpicture}
  \cdots 
  \begin{tikzpicture}
    \node[int] (i) at (0,.5) {};
    \node (v1) at (-.5,-.2) {$\omega$};
    \node (v2) at (0,-.2) {$\omega$};
    \node (v3) at (.5,-.2) {$\omega$};
  \draw (i) edge (v1) edge (v2) edge (v3);
  \end{tikzpicture}
\]
Note that there are $n$ $(\omega-j)$-edges and $\frac {g-1}{2}$ tripods with three $\omega$-legs each.
The cohomological degree of such a generator is $k=11+\frac32 (g-1)$.

One hence arrives at the following list of cases in which $\gr_{11} H_c^k(\M_{g,n})$ is concentrated in a single degree $k$, and isomorphic to $\Delta$. The $\bbS_n$-action is by the sign representation in each case.
\begin{align*}   H^k(B_{1,11}) &=
    \begin{cases}
        V_{1^{11}} & \text{for $k=11$} \\
        0 & \text{otherwise}
    \end{cases}
&
H^k(B_{3,8}) &=
    \begin{cases}
        V_{1^{8}} & \text{for $k=14$} \\
        0 & \text{otherwise}
    \end{cases}
\\
H^k(B_{5,5}) &=
    \begin{cases}
        V_{1^5} & \text{for $k=17$} \\
        0 & \text{otherwise}
    \end{cases}     
    &
    H^k(B_{7,2}) &=
    \begin{cases}
        V_{1^2} & \text{for $k=20$} \\
        0 & \text{otherwise}
    \end{cases}     \end{align*}

\subsection{Excess 1} \label{sec:exc1}
Suppose $E(g,n) = 1$.  Each generator for $B_{g,n}$ has all connected components of excess 0, except for one of excess 1.
The connected graphs of excess 1 are of the form:
\begin{equation}\label{equ:exc 1 trees}
  \begin{tikzpicture}
    \node (v) at (0,0) {$\omega$};
    \node (w) at (1,0) {$\epsilon$};
    \draw (v) edge (w);
  \end{tikzpicture}
\quad\text{or}\quad  
\begin{tikzpicture}
  \node[int] (i) at (0,.5) {};
  \node (v1) at (-.5,-.2) {$j$};
  \node (v2) at (0,-.2) {$\omega$};
  \node (v3) at (.5,-.2) {$\omega$};
\draw (i) edge (v1) edge (v2) edge (v3);
\end{tikzpicture}
\quad\text{or}\quad  
\begin{tikzpicture}
  \node[int] (i) at (0,.5) {};
  \node (v1) at (-.6,-.2) {$\omega$};
  \node (v2) at (-.2,-.2) {$\omega$};
  \node (v3) at (.2,-.2) {$\omega$};
  \node (v4) at (.6,-.2) {$\omega$};
\draw (i) edge (v1) edge (v2) edge (v3) edge (v4);
\end{tikzpicture}
\quad\text{or}\quad  
\begin{tikzpicture}
  \node[int] (i) at (0,.5) {};
  \node[int] (j) at (1,.5) {};
  \node (v1) at (-.5,-.2) {$\omega$};
  \node (v2) at (0,-.2) {$\omega$};
  \node (v3) at (1,-.2) {$\omega$};
  \node (v4) at (1.5,-.2) {$\omega$};
\draw (i) edge (v1) edge (v2) edge (j) (j) edge (v3) edge (v4);
\end{tikzpicture}
\, .
\end{equation}
The third graph maps to the fourth under the vertex splitting differential, and hence graphs with these components do not contribute to cohomology; we may simplify $B_{g,n}$ by killing these terms. 

The remaining excess 1 graphs are of the form 
\[
  \Gamma^{(1)}=
  \begin{tikzpicture}
    \node (v) at (0,0) {$\omega$};
    \node (w) at (1,0) {$1$};
    \draw (v) edge (w);
  \end{tikzpicture}
  \cdots 
  \begin{tikzpicture}
    \node (v) at (0,0) {$\omega$};
    \node (w) at (1,0) {$n$};
    \draw (v) edge (w);
  \end{tikzpicture}
  \begin{tikzpicture}
    \node[int] (i) at (0,.5) {};
    \node (v1) at (-.5,-.2) {$\omega$};
    \node (v2) at (0,-.2) {$\omega$};
    \node (v3) at (.5,-.2) {$\omega$};
  \draw (i) edge (v1) edge (v2) edge (v3);
  \end{tikzpicture}
  \cdots 
  \begin{tikzpicture}
    \node[int] (i) at (0,.5) {};
    \node (v1) at (-.5,-.2) {$\omega$};
    \node (v2) at (0,-.2) {$\omega$};
    \node (v3) at (.5,-.2) {$\omega$};
  \draw (i) edge (v1) edge (v2) edge (v3);
  \end{tikzpicture}
  \begin{tikzpicture}
    \node (v) at (0,0) {$\omega$};
    \node (w) at (1,0) {$\epsilon$};
    \draw (v) edge (w);
  \end{tikzpicture}
\]
or 
\[
  \Gamma^{(1)}_j=
  \begin{tikzpicture}
    \node (v) at (0,0) {$\omega$};
    \node (w) at (1,0) {$1$};
    \draw (v) edge (w);
  \end{tikzpicture}
  \cdots 
  \begin{tikzpicture}
    \node (v) at (0,0) {$\omega$};
    \node (w) at (1,0) {$n$};
    \draw (v) edge (w);
  \end{tikzpicture}
  \begin{tikzpicture}
    \node[int] (i) at (0,.5) {};
    \node (v1) at (-.5,-.2) {$\omega$};
    \node (v2) at (0,-.2) {$\omega$};
    \node (v3) at (.5,-.2) {$\omega$};
  \draw (i) edge (v1) edge (v2) edge (v3);
  \end{tikzpicture}
  \cdots 
  \begin{tikzpicture}
    \node[int] (i) at (0,.5) {};
    \node (v1) at (-.5,-.2) {$\omega$};
    \node (v2) at (0,-.2) {$\omega$};
    \node (v3) at (.5,-.2) {$\omega$};
  \draw (i) edge (v1) edge (v2) edge (v3);
  \end{tikzpicture}
  \begin{tikzpicture}
    \node[int] (i) at (0,.5) {};
    \node (v1) at (-.5,-.2) {$j$};
    \node (v2) at (0,-.2) {$\omega$};
    \node (v3) at (.5,-.2) {$\omega$};
  \draw (i) edge (v1) edge (v2) edge (v3);
  \end{tikzpicture},
\]
with the understanding that there is no $(\omega-j)$-edge in $\Gamma^{(1)}_j$.
Modulo terms involving the fourth graph in \eqref{equ:exc 1 trees}, which we ignore, the differential is given by 
\begin{equation}\label{eq:diff-excess1}
  \Gamma^{(1)}\mapsto \pm \sum_{j=1}^n (-1)^j\Gamma^{(1)}_j.
\end{equation}
It follows that the cohomology of $B_{g,n}$ (with $E(g,n)=1$) is one copy of the irreducible $\bbS_n$-representation $V_{21^{n-2}}$ in degree $k=10+\frac 32 g$, given by the cokernel of \eqref{eq:diff-excess1}.
Concretely, this applies to the cases:
\begin{align*}
    H^k(B_{2,10}) &=
    \begin{cases}
        V_{21^{8}} & \text{for $k=13$} \\
        0 & \text{otherwise}
    \end{cases}
&
    H^k(B_{4,7}) &=
    \begin{cases}
        V_{21^{5}} & \text{for $k=16$} \\
        0 & \text{otherwise}
    \end{cases}
\\
    H^k(B_{6,4}) &=
    \begin{cases}
        V_{21^2} & \text{for $k=19$} \\
        0 & \text{otherwise}
    \end{cases}
    &
    H^k(B_{8,1}) &= 0.
\end{align*}

\subsection{Excess 2} \label{sec:exc2}
Suppose $E(g,n) = 2$.  A generator for $B_{g,n}$ of excess 2 has either two connected components of excess 1, or one of excess 2.
The connected components of excess 2 are
\begin{equation}\label{equ:exc 2 trees}
    \begin{tikzpicture}
      \node (v) at (0,0) {$j$};
      \node (w) at (1,0) {$\epsilon$};
      \draw (v) edge (w);
    \end{tikzpicture}
  \quad\text{or}\quad  
  \begin{tikzpicture}
    \node[int] (i) at (0,.5) {};
    \node (v1) at (-.5,-.2) {$\epsilon$};
    \node (v2) at (0,-.2) {$\omega$};
    \node (v3) at (.5,-.2) {$\omega$};
  \draw (i) edge (v1) edge (v2) edge (v3);
  \end{tikzpicture}
  \quad\text{or}\quad  
  \begin{tikzpicture}
    \node[int] (i) at (0,.5) {};
    \node (v1) at (-.5,-.2) {$i$};
    \node (v2) at (0,-.2) {$j$};
    \node (v3) at (.5,-.2) {$\omega$};
  \draw (i) edge (v1) edge (v2) edge (v3);
  \end{tikzpicture}
  \quad\text{or trees with at least 4 leaves}\quad 
  \, .
  \end{equation}
  As in the case $E(g,n) = 1$, one readily checks that trees with at least 4 leaves do not contribute to the cohomology of $B_{g,n}$, and can be killed by a chain homotopy.  Thus we ignore such terms.  The remaining reduced version of $B_{g,n}$ is generated by the single graph $\Gamma^{(0)}$ of excess $0$ in degree $10+\frac32(g-1)$, along with the following graphs of excess 2:
\[
    \Gamma^{(2)}_{\epsilon}:=
  \begin{tikzpicture}
    \node (v) at (0,0) {$\omega$};
    \node (w) at (1,0) {$1$};
    \draw (v) edge (w);
  \end{tikzpicture}
  \cdots 
  \begin{tikzpicture}
    \node (v) at (0,0) {$\omega$};
    \node (w) at (1,0) {$n$};
    \draw (v) edge (w);
  \end{tikzpicture}
  \begin{tikzpicture}
    \node[int] (i) at (0,.5) {};
    \node (v1) at (-.5,-.2) {$\omega$};
    \node (v2) at (0,-.2) {$\omega$};
    \node (v3) at (.5,-.2) {$\omega$};
  \draw (i) edge (v1) edge (v2) edge (v3);
  \end{tikzpicture}
  \cdots 
  \begin{tikzpicture}
    \node[int] (i) at (0,.5) {};
    \node (v1) at (-.5,-.2) {$\omega$};
    \node (v2) at (0,-.2) {$\omega$};
    \node (v3) at (.5,-.2) {$\omega$};
  \draw (i) edge (v1) edge (v2) edge (v3);
  \end{tikzpicture}
  \begin{tikzpicture}
    \node[int] (i) at (0,.5) {};
    \node (v1) at (-.5,-.2) {$\epsilon$};
    \node (v2) at (0,-.2) {$\omega$};
    \node (v3) at (.5,-.2) {$\omega$};
  \draw (i) edge (v1) edge (v2) edge (v3);
  \end{tikzpicture}
\]
\[
    \Gamma^{(2)}_{\epsilon j}:=
    \begin{tikzpicture}
        \node (v) at (0,0) {$\epsilon$};
        \node (w) at (1,0) {$j$};
        \draw (v) edge (w);
    \end{tikzpicture}
  \begin{tikzpicture}
    \node (v) at (0,0) {$\omega$};
    \node (w) at (1,0) {$1$};
    \draw (v) edge (w);
  \end{tikzpicture}
  \cdots 
  \begin{tikzpicture}
    \node (v) at (0,0) {$\omega$};
    \node (w) at (1,0) {$n$};
    \draw (v) edge (w);
  \end{tikzpicture}
  \begin{tikzpicture}
    \node[int] (i) at (0,.5) {};
    \node (v1) at (-.5,-.2) {$\omega$};
    \node (v2) at (0,-.2) {$\omega$};
    \node (v3) at (.5,-.2) {$\omega$};
  \draw (i) edge (v1) edge (v2) edge (v3);
  \end{tikzpicture}
  \cdots 
  \begin{tikzpicture}
    \node[int] (i) at (0,.5) {};
    \node (v1) at (-.5,-.2) {$\omega$};
    \node (v2) at (0,-.2) {$\omega$};
    \node (v3) at (.5,-.2) {$\omega$};
  \draw (i) edge (v1) edge (v2) edge (v3);
  \end{tikzpicture}
\]
\[
    \Gamma^{(2)}_{ij}:=
  \begin{tikzpicture}
    \node (v) at (0,0) {$\omega$};
    \node (w) at (1,0) {$1$};
    \draw (v) edge (w);
  \end{tikzpicture}
  \cdots 
  \begin{tikzpicture}
    \node (v) at (0,0) {$\omega$};
    \node (w) at (1,0) {$n$};
    \draw (v) edge (w);
  \end{tikzpicture}
  \begin{tikzpicture}
    \node[int] (i) at (0,.5) {};
    \node (v1) at (-.5,-.2) {$\omega$};
    \node (v2) at (0,-.2) {$\omega$};
    \node (v3) at (.5,-.2) {$\omega$};
  \draw (i) edge (v1) edge (v2) edge (v3);
  \end{tikzpicture}
  \cdots 
  \begin{tikzpicture}
    \node[int] (i) at (0,.5) {};
    \node (v1) at (-.5,-.2) {$\omega$};
    \node (v2) at (0,-.2) {$\omega$};
    \node (v3) at (.5,-.2) {$\omega$};
  \draw (i) edge (v1) edge (v2) edge (v3);
  \end{tikzpicture}
  \begin{tikzpicture}
    \node[int] (i) at (0,.5) {};
    \node (v1) at (-.5,-.2) {$i$};
    \node (v2) at (0,-.2) {$j$};
    \node (v3) at (.5,-.2) {$\omega$};
  \draw (i) edge (v1) edge (v2) edge (v3);
  \end{tikzpicture}
\]
\[
  \Gamma^{(2)}_{\omega\epsilon\omega\epsilon}=
  \begin{tikzpicture}
    \node (v) at (0,0) {$\omega$};
    \node (w) at (1,0) {$1$};
    \draw (v) edge (w);
  \end{tikzpicture}
  \cdots 
  \begin{tikzpicture}
    \node (v) at (0,0) {$\omega$};
    \node (w) at (1,0) {$n$};
    \draw (v) edge (w);
  \end{tikzpicture}
  \begin{tikzpicture}
    \node[int] (i) at (0,.5) {};
    \node (v1) at (-.5,-.2) {$\omega$};
    \node (v2) at (0,-.2) {$\omega$};
    \node (v3) at (.5,-.2) {$\omega$};
  \draw (i) edge (v1) edge (v2) edge (v3);
  \end{tikzpicture}
  \cdots 
  \begin{tikzpicture}
    \node[int] (i) at (0,.5) {};
    \node (v1) at (-.5,-.2) {$\omega$};
    \node (v2) at (0,-.2) {$\omega$};
    \node (v3) at (.5,-.2) {$\omega$};
  \draw (i) edge (v1) edge (v2) edge (v3);
  \end{tikzpicture}
  \begin{tikzpicture}
    \node (v) at (0,0) {$\omega$};
    \node (w) at (1,0) {$\epsilon$};
    \draw (v) edge (w);
  \end{tikzpicture}
  \begin{tikzpicture}
    \node (v) at (0,0) {$\omega$};
    \node (w) at (1,0) {$\epsilon$};
    \draw (v) edge (w);
  \end{tikzpicture}
\]
\[
  \Gamma^{(2)}_{\omega\epsilon j}=
  \begin{tikzpicture}
    \node (v) at (0,0) {$\omega$};
    \node (w) at (1,0) {$1$};
    \draw (v) edge (w);
  \end{tikzpicture}
  \cdots 
  \begin{tikzpicture}
    \node (v) at (0,0) {$\omega$};
    \node (w) at (1,0) {$n$};
    \draw (v) edge (w);
  \end{tikzpicture}
  \begin{tikzpicture}
    \node[int] (i) at (0,.5) {};
    \node (v1) at (-.5,-.2) {$\omega$};
    \node (v2) at (0,-.2) {$\omega$};
    \node (v3) at (.5,-.2) {$\omega$};
  \draw (i) edge (v1) edge (v2) edge (v3);
  \end{tikzpicture}
  \cdots 
  \begin{tikzpicture}
    \node[int] (i) at (0,.5) {};
    \node (v1) at (-.5,-.2) {$\omega$};
    \node (v2) at (0,-.2) {$\omega$};
    \node (v3) at (.5,-.2) {$\omega$};
  \draw (i) edge (v1) edge (v2) edge (v3);
  \end{tikzpicture}
  \begin{tikzpicture}
    \node[int] (i) at (0,.5) {};
    \node (v1) at (-.5,-.2) {$j$};
    \node (v2) at (0,-.2) {$\omega$};
    \node (v3) at (.5,-.2) {$\omega$};
  \draw (i) edge (v1) edge (v2) edge (v3);
  \end{tikzpicture}
  \begin{tikzpicture}
    \node (v) at (0,0) {$\omega$};
    \node (w) at (1,0) {$\epsilon$};
    \draw (v) edge (w);
  \end{tikzpicture}
\]
\[
    \Gamma^{(2)}_{i;j}:=
  \begin{tikzpicture}
    \node (v) at (0,0) {$\omega$};
    \node (w) at (1,0) {$1$};
    \draw (v) edge (w);
  \end{tikzpicture}
  \cdots 
  \begin{tikzpicture}
    \node (v) at (0,0) {$\omega$};
    \node (w) at (1,0) {$n$};
    \draw (v) edge (w);
  \end{tikzpicture}
  \begin{tikzpicture}
    \node[int] (i) at (0,.5) {};
    \node (v1) at (-.5,-.2) {$\omega$};
    \node (v2) at (0,-.2) {$\omega$};
    \node (v3) at (.5,-.2) {$\omega$};
  \draw (i) edge (v1) edge (v2) edge (v3);
  \end{tikzpicture}
  \cdots 
  \begin{tikzpicture}
    \node[int] (i) at (0,.5) {};
    \node (v1) at (-.5,-.2) {$\omega$};
    \node (v2) at (0,-.2) {$\omega$};
    \node (v3) at (.5,-.2) {$\omega$};
  \draw (i) edge (v1) edge (v2) edge (v3);
  \end{tikzpicture}
  \begin{tikzpicture}
    \node[int] (i) at (0,.5) {};
    \node (v1) at (-.5,-.2) {$i$};
    \node (v2) at (0,-.2) {$\omega$};
    \node (v3) at (.5,-.2) {$\omega$};
  \draw (i) edge (v1) edge (v2) edge (v3);
  \end{tikzpicture}
  \begin{tikzpicture}
    \node[int] (i) at (0,.5) {};
    \node (v1) at (-.5,-.2) {$j$};
    \node (v2) at (0,-.2) {$\omega$};
    \node (v3) at (.5,-.2) {$\omega$};
  \draw (i) edge (v1) edge (v2) edge (v3);
  \end{tikzpicture}
\]
After killing terms involving trees with at least 4 leaves, the differential maps $\Gamma^{(2)}_\epsilon$, $\Gamma^{(2)}_{ij}$, and $\Gamma^{(2)}_{i;j}$ to $0$. On the remaining generators, it is given by:
\begin{align*}
\Gamma^{(0)} &\mapsto \sum_j \pm  \Gamma^{(2)}_{\epsilon j} + (const)\Gamma^{(2)}_\epsilon &
\Gamma^{(2)}_{\epsilon j} &\mapsto \sum_i \pm \Gamma^{(2)}_{ij} \\ 
\Gamma^{(2)}_{\omega\epsilon\omega\epsilon} &\mapsto \pm \Gamma^{(2)}_\epsilon + \sum_j \pm \Gamma^{(2)}_{\omega\epsilon j} &
\Gamma^{(2)}_{\omega\epsilon j}&\mapsto \sum_i \pm \Gamma^{(2)}_{i;j}. 
\end{align*}
Here one needs to take care that when $n = 0$ or $g = 1$, some of these generators are absent.  More precisely, the generators $\Gamma^{(2)}_{\epsilon j}$, $\Gamma^{(2)}_{\omega\epsilon j}$,$\Gamma^{(2)}_{ij}$, $\Gamma^{(2)}_{i;j}$ are not present when $n = 0$, nor are $\Gamma^{(2)}_\epsilon$, $\Gamma^{(2)}_{i;j}$, $\Gamma^{(2)}_{\omega\epsilon\omega\epsilon}$, $\Gamma^{(2)}_{\omega\epsilon j}$ when $g = 1$.  When all of the generators are present, the cohomology consists of one copy of the sign representation $V_{1^n}$ of $\bbS_n$, represented by $\Gamma^{(2)}_\epsilon+\cdots$, and two copies of the irreducible representation $V_{31^{n-3}}$, represented by linear combinations of graphs of the form $\Gamma^{(2)}_{ij}$ and $\Gamma^{(2)}_{i;j}$ respectively.
Taking into account the special cases $n=0$ and $g=1$, we arrive at the following: 
\begin{align*}
    H^k(B_{1,12}) &=
    \begin{cases}
        V_{31^{9}} & \text{for $k=12$} \\
        0 & \text{otherwise}
    \end{cases}
    &
    H^k(B_{3,9}) &=
    \begin{cases}
        V_{1^{9}} & \text{for $k=14$} \\
        V_{31^{6}}\oplus V_{31^{6}} & \text{for $k=15$} \\
        0 & \text{otherwise}
    \end{cases}
    \\
    H^k(B_{5,6}) &=
    \begin{cases}
        V_{1^6} & \text{for $k=17$} \\
        V_{31^{3}}\oplus V_{31^{3}} & \text{for $k=18$} \\
        0 & \text{otherwise}
    \end{cases}
    &
    H^k(B_{7,3}) &=
    \begin{cases}
        V_{1^3} & \text{for $k=20$} \\
        V_{3}\oplus V_{3} & \text{for $k=21$} \\
        0 & \text{otherwise}
    \end{cases}\\
    H^k(B_{9,0}) &=
    \begin{cases}
        \Q & \text{for $k=22$} \\
        0 & \text{otherwise.}
    \end{cases}
\end{align*}

\subsection{Excess 3}
Suppose $E(g,n) = 3$.  Then each generator $\Gamma$ for $B_{g,n}$ has excess 1 or 3, and the contributions of the blown up components determine a partition $\lambda$ of $E(\Gamma)$.  We filter $B_{g,n}$ according to the lexicographic ordering on these partitions, and consider the associated spectral sequence.  On the first page, the differential only relates generators with the same partition.
\begin{itemize}
\item
$\bf{\lambda = 1}$. The graphs of excess 1 and the complex that they generate on the first page are exactly as in \S\ref{sec:exc1}.   In particular, the cohomology is one copy of $V_{21^{n-2}}$ in degree $\frac32(g-4)+15$, represented by linear combinations of the graphs of type $\Gamma^{(1)}_j$.
\item
$\bf{\lambda = 1^3}$. 
The contribution of graphs with the blown-up components of excess 1 is similar to the excess 2 case computed in \S\ref{sec:exc2}. The resulting cohomology on the first page is one copy of $V_{41^{n-4}}$ in degree $\frac32(g-3)+17$ represented by linear combinations of graphs 
\[
    \Gamma^{(3)}_{i;j;k}:=
  \begin{tikzpicture}
    \node (v) at (0,0) {$\omega$};
    \node (w) at (1,0) {$1$};
    \draw (v) edge (w);
  \end{tikzpicture}
  \cdots 
  \begin{tikzpicture}
    \node (v) at (0,0) {$\omega$};
    \node (w) at (1,0) {$n$};
    \draw (v) edge (w);
  \end{tikzpicture}
  \begin{tikzpicture}
    \node[int] (i) at (0,.5) {};
    \node (v1) at (-.5,-.2) {$\omega$};
    \node (v2) at (0,-.2) {$\omega$};
    \node (v3) at (.5,-.2) {$\omega$};
  \draw (i) edge (v1) edge (v2) edge (v3);
  \end{tikzpicture}
  \cdots 
  \begin{tikzpicture}
    \node[int] (i) at (0,.5) {};
    \node (v1) at (-.5,-.2) {$\omega$};
    \node (v2) at (0,-.2) {$\omega$};
    \node (v3) at (.5,-.2) {$\omega$};
  \draw (i) edge (v1) edge (v2) edge (v3);
  \end{tikzpicture}
  \begin{tikzpicture}
    \node[int] (i) at (0,.5) {};
    \node (v1) at (-.5,-.2) {$i$};
    \node (v2) at (0,-.2) {$\omega$};
    \node (v3) at (.5,-.2) {$\omega$};
  \draw (i) edge (v1) edge (v2) edge (v3);
  \end{tikzpicture}
  \begin{tikzpicture}
    \node[int] (i) at (0,.5) {};
    \node (v1) at (-.5,-.2) {$j$};
    \node (v2) at (0,-.2) {$\omega$};
    \node (v3) at (.5,-.2) {$\omega$};
  \draw (i) edge (v1) edge (v2) edge (v3);
  \end{tikzpicture}
  \begin{tikzpicture}
    \node[int] (i) at (0,.5) {};
    \node (v1) at (-.5,-.2) {$k$};
    \node (v2) at (0,-.2) {$\omega$};
    \node (v3) at (.5,-.2) {$\omega$};
  \draw (i) edge (v1) edge (v2) edge (v3);
  \end{tikzpicture}
\]
\item  $\bf{\lambda = 21}$. Here we have 6 types of graphs to consider, coming from 3 types of components of excess 2 \eqref{equ:exc 2 trees} and 2 types of components of excess 1  \eqref{equ:exc 1 trees}.  
The cohomology consists of:
\begin{itemize}
\item A copy of $V_{41^{n-4}}\oplus V_{321^{n-5}}$ in degree 
$\frac32(g-4)+17$ represented by linear combinations of graphs 
\[
    \Gamma^{(3)}_{ij;k}:=
  \begin{tikzpicture}
    \node (v) at (0,0) {$\omega$};
    \node (w) at (1,0) {$1$};
    \draw (v) edge (w);
  \end{tikzpicture}
  \cdots 
  \begin{tikzpicture}
    \node (v) at (0,0) {$\omega$};
    \node (w) at (1,0) {$n$};
    \draw (v) edge (w);
  \end{tikzpicture}
  \begin{tikzpicture}
    \node[int] (i) at (0,.5) {};
    \node (v1) at (-.5,-.2) {$\omega$};
    \node (v2) at (0,-.2) {$\omega$};
    \node (v3) at (.5,-.2) {$\omega$};
  \draw (i) edge (v1) edge (v2) edge (v3);
  \end{tikzpicture}
  \cdots 
  \begin{tikzpicture}
    \node[int] (i) at (0,.5) {};
    \node (v1) at (-.5,-.2) {$\omega$};
    \node (v2) at (0,-.2) {$\omega$};
    \node (v3) at (.5,-.2) {$\omega$};
  \draw (i) edge (v1) edge (v2) edge (v3);
  \end{tikzpicture}
  \begin{tikzpicture}
    \node[int] (i) at (0,.5) {};
    \node (v1) at (-.5,-.2) {$i$};
    \node (v2) at (0,-.2) {$j$};
    \node (v3) at (.5,-.2) {$\omega$};
  \draw (i) edge (v1) edge (v2) edge (v3);
  \end{tikzpicture}
  \begin{tikzpicture}
    \node[int] (i) at (0,.5) {};
    \node (v1) at (-.5,-.2) {$k$};
    \node (v2) at (0,-.2) {$\omega$};
    \node (v3) at (.5,-.2) {$\omega$};
  \draw (i) edge (v1) edge (v2) edge (v3);
  \end{tikzpicture}
\]
\item Two copies of $V_{21^{n-2}}$ in degree $\frac32(g-4)+16$ one represented by graphs of the form
\[
  \Gamma^{(3)}_{\epsilon; j}=
  \begin{tikzpicture}
    \node (v) at (0,0) {$\omega$};
    \node (w) at (1,0) {$1$};
    \draw (v) edge (w);
  \end{tikzpicture}
  \cdots 
  \begin{tikzpicture}
    \node (v) at (0,0) {$\omega$};
    \node (w) at (1,0) {$n$};
    \draw (v) edge (w);
  \end{tikzpicture}
  \begin{tikzpicture}
    \node[int] (i) at (0,.5) {};
    \node (v1) at (-.5,-.2) {$\omega$};
    \node (v2) at (0,-.2) {$\omega$};
    \node (v3) at (.5,-.2) {$\omega$};
  \draw (i) edge (v1) edge (v2) edge (v3);
  \end{tikzpicture}
  \cdots 
  \begin{tikzpicture}
    \node[int] (i) at (0,.5) {};
    \node (v1) at (-.5,-.2) {$\omega$};
    \node (v2) at (0,-.2) {$\omega$};
    \node (v3) at (.5,-.2) {$\omega$};
  \draw (i) edge (v1) edge (v2) edge (v3);
  \end{tikzpicture}
  \begin{tikzpicture}
    \node[int] (i) at (0,.5) {};
    \node (v1) at (-.5,-.2) {$j$};
    \node (v2) at (0,-.2) {$\omega$};
    \node (v3) at (.5,-.2) {$\omega$};
  \draw (i) edge (v1) edge (v2) edge (v3);
  \end{tikzpicture}
  \begin{tikzpicture}
    \node[int] (i) at (0,.5) {};
    \node (v1) at (-.5,-.2) {$\epsilon$};
    \node (v2) at (0,-.2) {$\omega$};
    \node (v3) at (.5,-.2) {$\omega$};
  \draw (i) edge (v1) edge (v2) edge (v3);
  \end{tikzpicture}
\]
and another by graphs of the form:
\[
  \Gamma^{(3)}_{\epsilon i;j}=
  \begin{tikzpicture}
    \node (v) at (0,0) {$\omega$};
    \node (w) at (1,0) {$1$};
    \draw (v) edge (w);
  \end{tikzpicture}
  \cdots 
  \begin{tikzpicture}
    \node (v) at (0,0) {$\omega$};
    \node (w) at (1,0) {$n$};
    \draw (v) edge (w);
  \end{tikzpicture}
  \begin{tikzpicture}
    \node[int] (i) at (0,.5) {};
    \node (v1) at (-.5,-.2) {$\omega$};
    \node (v2) at (0,-.2) {$\omega$};
    \node (v3) at (.5,-.2) {$\omega$};
  \draw (i) edge (v1) edge (v2) edge (v3);
  \end{tikzpicture}
  \cdots 
  \begin{tikzpicture}
    \node[int] (i) at (0,.5) {};
    \node (v1) at (-.5,-.2) {$\omega$};
    \node (v2) at (0,-.2) {$\omega$};
    \node (v3) at (.5,-.2) {$\omega$};
  \draw (i) edge (v1) edge (v2) edge (v3);
  \end{tikzpicture}
  \begin{tikzpicture}
    \node[int] (i) at (0,.5) {};
    \node (v1) at (-.5,-.2) {$j$};
    \node (v2) at (0,-.2) {$\omega$};
    \node (v3) at (.5,-.2) {$\omega$};
  \draw (i) edge (v1) edge (v2) edge (v3);
  \end{tikzpicture}
  \begin{tikzpicture}
    \node (v) at (0,0) {$i$};
    \node (w) at (1,0) {$\epsilon$};
    \draw (v) edge (w);
  \end{tikzpicture}
\]
\end{itemize}
\item $\bf{\lambda = 3}$.
The relevant trees of excess 3 are:
\begin{equation*}
    \begin{tikzpicture}
      \node (v) at (0,0) {$\epsilon$};
      \node (w) at (1,0) {$\epsilon$};
      \draw (v) edge (w);
    \end{tikzpicture}
  \quad\text{or}\quad  
  \begin{tikzpicture}
    \node[int] (i) at (0,.5) {};
    \node (v1) at (-.5,-.2) {$\epsilon$};
    \node (v2) at (0,-.2) {$i$};
    \node (v3) at (.5,-.2) {$\omega$};
  \draw (i) edge (v1) edge (v2) edge (v3);
  \end{tikzpicture}
    \quad\text{or}\quad 
    \begin{tikzpicture}
    \node[int] (i) at (0,.5) {};
    \node[int] (j) at (1,.5) {};
    \node (v1) at (-.5,-.2) {$i$};
    \node (v2) at (0,-.2) {$\omega$};
    \node (v3) at (1,-.2) {$\omega$};
    \node (v4) at (1.5,-.2) {$j$};
  \draw (i) edge (v1) edge (v2) edge (j) (j) edge (v3) edge (v4);
  \end{tikzpicture}
\quad  .
  \end{equation*}
  The cohomology of the resulting 3-term complex is $V_{31^{n-3}}$ in degree $\frac32(g-4)+17$ represented by linear combinations of graphs of the form 
  \[
  \Gamma^{(3)}_{i\cdot j}=
  \begin{tikzpicture}
    \node (v) at (0,0) {$\omega$};
    \node (w) at (1,0) {$1$};
    \draw (v) edge (w);
  \end{tikzpicture}
  \cdots 
  \begin{tikzpicture}
    \node (v) at (0,0) {$\omega$};
    \node (w) at (1,0) {$n$};
    \draw (v) edge (w);
  \end{tikzpicture}
  \begin{tikzpicture}
    \node[int] (i) at (0,.5) {};
    \node (v1) at (-.5,-.2) {$\omega$};
    \node (v2) at (0,-.2) {$\omega$};
    \node (v3) at (.5,-.2) {$\omega$};
  \draw (i) edge (v1) edge (v2) edge (v3);
  \end{tikzpicture}
  \cdots 
  \begin{tikzpicture}
    \node[int] (i) at (0,.5) {};
    \node (v1) at (-.5,-.2) {$\omega$};
    \node (v2) at (0,-.2) {$\omega$};
    \node (v3) at (.5,-.2) {$\omega$};
  \draw (i) edge (v1) edge (v2) edge (v3);
  \end{tikzpicture}
  \begin{tikzpicture}
    \node[int] (i) at (0,.5) {};
    \node[int] (j) at (1,.5) {};
    \node (v1) at (-.5,-.2) {$i$};
    \node (v2) at (0,-.2) {$\omega$};
    \node (v3) at (1,-.2) {$\omega$};
    \node (v4) at (1.5,-.2) {$j$};
  \draw (i) edge (v1) edge (v2) edge (j) (j) edge (v3) edge (v4);
  \end{tikzpicture}
\]
\end{itemize}

The only possible cancellations on later pages of the spectral sequence are between the copies of $V_{21^{n-2}}$. And indeed $\delta_\omega \Gamma^{(1)}_i=\sum_j\pm \Gamma^{(3)}_{\epsilon i;j}+(\cdots)$, and thus the two corresponding copies of $V_{21^{n-2}}$ do cancel on the second page of the spectral sequence.  Taking into account the cases of low $g\leq 2$ and low $n\leq 3$ one arrives at the following expressions for the cohomology of $B_{g,n}$:
\begin{align*}
    H^k(B_{2,11}) &=
    \begin{cases}
        V_{41^{7}}\oplus V_{321^{6}}\oplus V_{31^{8}} & \text{for $k=14$} \\
        0 & \text{otherwise}
    \end{cases}
    \\
    H^k(B_{4,8}) &=
    \begin{cases}
        V_{21^{6}} & \text{for $k=16$} \\
        V_{41^{4}}\oplus V_{41^{4}}\oplus V_{321^{3}}\oplus V_{31^{5}} & \text{for $k=17$} \\
        0 & \text{otherwise}
    \end{cases}
    \\
    H^k(B_{6,5}) &=
    \begin{cases}
        V_{21^{3}} & \text{for $k=19$} \\
        V_{41}\oplus V_{41}\oplus V_{32}\oplus V_{31^{2}} & \text{for $k=20$} \\
        0 & \text{otherwise}
    \end{cases}
    \\
    H^k(B_{8,2}) &=
    \begin{cases}
        V_{2} & \text{for $k=22$} \\
        0 & \text{otherwise.}
    \end{cases}
\end{align*}

\section{The case $n=0$: first injection} \label{sec:first inj}
We now restrict attention to the special case where $n = 0$, allowing the genus $g$ (and hence the excess $E$) to be arbitrarily large.  Following the standard notational convention for moduli of curves, we write: 
\begin{align*}
X_g &:= X_{g,0},
&
B_g &:= B_{g,0},
&
C_g &:= C_{g,0}.
\end{align*}
In this section and the next, we identify several nontrivial families of cohomology classes in $C_g$ and $B_g$, respectively, built from the weight 0 compactly supported cohomology of $\M_{h}$ for $h<g$.

\medskip

Let $\GC_0$ be the graph complex generated by connected graphs without tadpoles in which every vertex has valence at least 3.  Each generator comes with an orientation, which is a total ordering of the edges.  As before, we identify isomorphic graphs, and we identify two orientations up to sign, cf. \eqref{equ:bGK gen rel 1} and \eqref{equ:bGK gen rel 2}, and  \eqref{equ:X gen rel 2} and \eqref{equ:X gen rel 1}.
The cohomological degree of a generator is the number of edges, and the differential $\delta$ is given by vertex splitting, as in \eqref{equ:delta bullet def}.  The differential preserves the loop order (first Betti number) and we write $\GC_0^{(g)}$ for the part of loop order $g$.

Recall that $\GC_0^{(g)}$ is quasi-isomorphic to the Feynman transform of the modular co-operad $H^0(\MM)$, evaluated at $(g,0)$; more precisely, it is the quotient of $\cF H^0(\MM)(g,0)$ by the acyclic subcomplex generated by graphs with tadpoles or vertices of positive genus \cite{CGP1}.  As a consequence, there is a canonical isomorphism
\[
W_0H^\bullet_c(\M_{g}) \cong H^\bullet(\GC_0^{(g)}).
\]
\noindent The symmetric product 
$
\Sym^k(\GC_0)
$
inherits a loop order grading from $\GC_0$ and we denote by $\Sym^k(\GC_0)^{(g)}\subset \Sym^k(\GC_0)$ the part of loop order $g$.

\begin{thm}\label{thm:emb1}
There is a map of complexes
\[
F \colon \Sym^{10}(\GC_0)^{(g-1)}[-21] \oplus \Sym^{10}(\GC_0)^{(g-2)}[-22] \to C_{g}
\]
that induces an injective map on the level of cohomology
\[
H^{k-21}(\Sym^{10}(\GC_0)^{(g-1)}) \oplus H^{k-22}(\Sym^{10}(\GC_0)^{(g-2)})
\to H^k(C_g).
\]
\end{thm}

\begin{corollary}
There is a natural injection 
\[
\big(H^{k-21}(\Sym^{10}(\GC_0)^{(g-1)}) \oplus H^{k-22}(\Sym^{10}(\GC_0)^{(g-2)})\big) \otimes \Delta 
\to \gr_{11}H^k(\M_g).
\]
\end{corollary}

\subsection{A homotopy trivial Lie bracket on $\GC_0$}
The graph complex $\GC_0$ is a dg Lie algebra with the Lie bracket $[-,-]$ defined by inserting one graph into a vertex of the other. This primary Lie bracket has degree $0$; it leaves the number of edges invariant and removes one vertex.

There is also a secondary Lie bracket of degree $+1$ on $\GC_0$ which, we denote $\{-,-\}$, defined by gluing two graphs together by attaching a new edge between them:
\[
\{\gamma_1,\gamma_2\} := \sum_{v\in V\gamma_1 \atop w\in V\gamma_2} (v,w) \cup \gamma_1 \cup \gamma_2 
= 
\sum 
\begin{tikzpicture}
    \node[ext] (v) at (0,0) {$\gamma_1$};
    \node[ext] (w) at (1,0) {$\gamma_2$};
    \draw (v) edge (w);
\end{tikzpicture}
\]
To fix the sign, the newly added edge comes first in the ordering, followed by the edges of $\gamma_1$, and then those of $\gamma_2$.  This secondary Lie bracket $\{-,-\}$ is homotopy trivial.
\begin{prop}[\cite{Willwachertrivial}]\label{prop:pois trivial}
There is an $L_\infty$-isomorphism 
\[
(\GC_0, \delta, 0) \to (\GC_0, \delta, \{-,-\})
\]
between the abelian Lie algebra $\GC_0$ and the dg Lie algebra $\GC_0$ equipped with the Lie bracket $\{-,-\}$.
This $L_\infty$-isomorphism preserves the grading by loop order on both sides. 
\end{prop}

The Chevalley-Eilenberg complex of the dg Lie algebra $\GC_0$ with the homotopy trivial Lie bracket $\{-,-\}$ of degree +1 is the graded vector space
\[
  CE(\GC_0,  \delta, \{-,-\}) = \Sym(\GC_0)
\]
with the differential $\delta+\delta_{\{,\}}$ such that
\[
\delta_{\{,\} } (\gamma_1\cdots \gamma_k)
=
\sum_{i<j}
(-1)^{|\gamma_i|(|\gamma_1|+\cdots +|\gamma_{i-1}|) + |\gamma_j|(|\gamma_1|+\cdots +|\gamma_{j-1}|) -|\gamma_i||\gamma_j|} 
\{\gamma_i,\gamma_j\} 
\gamma_1\cdots \hat \gamma_i \cdots \hat \gamma_j \cdots \gamma_k.
\]
Combinatorially, if we think of the product $\gamma_1\cdots \gamma_k$ as the union of the graphs $\gamma_1,\dots, \gamma_k$, then $\delta_{\{,\} }$ adds a new edge between any pair of vertices that belong to different connected components.
Any $L_\infty$-morphism induces a morphism on the Chevalley-Eilenberg complexes. Hence we obtain from Proposition \ref{prop:pois trivial} the following result.

\begin{cor}\label{cor:CE map}
    There is an isomorphism between the Chevalley-Eilenberg complexes of the abelian dg Lie algebra $\GC_0$ and the dg Lie algebra $\GC_0$ with the bracket $\{-,-\}$.
\[
\Phi\colon CE(\GC_0, \delta, 0) = (\Sym(\GC_0), \delta) \to CE(\GC_0, \delta, \{-,-\})
=
(\Sym(\GC_0), \delta + \delta_{\{,\}}).
\]
\end{cor}

\subsection{Some combinatorial operations}
We consider the tadpole-free version of the Kontsevich graphical operad $\Graphs_0$ \cite{KMotives}.
As for any operad, its unary operations $G_1:=\Graphs_0(1)$ form a dg associative algebra with the operadic composition $\circ$ as the product.
Concretely, elements of $G_1$ are linear combinations of pairs $(\Gamma,o)$ with $\Gamma$ a connected graph with one special ``external" vertex, and an orientation $o=e_1\wedge \cdots \wedge e_k$ given by ordering the edges.
\[
\begin{tikzpicture}
\node[int] (v1) at (0,0) {};
\node[int] (v2) at (1,0) {};
\node[int] (v3) at (0,1) {};
\node[int] (v4) at (1,1) {};
\node[ext] (e) at (0.5,-.7) {};
\draw (v1) edge (v2) edge (v3) edge (v4) edge (e)
(v2) edge (v3) edge (v4) edge (e)
(v3) edge (v4) edge (e) 
(v4) edge (e);
\end{tikzpicture} 
\]
The external vertex may have any valence; all other vertices must have valence at least $3$. 

The product (i.e., the operadic composition) is defined by inserting one graph into the other and summing over all ways of reconnecting the incident edges to vertices of $\Gamma_2$.
\[
\begin{tikzpicture}
\node[ext] (v) at (0,.5) {$\Gamma_1$};
\node[ext] (w) at (0,-.5) {};
\draw (v) edge (w) (v.west) edge (w) (v.east) edge (w);
\end{tikzpicture}
\circ 
\begin{tikzpicture}
    \node[ext] (v) at (0,.5) {$\Gamma_2$};
    \node[ext] (w) at (0,-.5) {};
    \draw (v) edge (w) (v.west) edge (w) (v.east) edge (w);
\end{tikzpicture}
=
\sum 
\begin{tikzpicture}
    \node[ext] (v) at (0,.8) {$\Gamma_1$};
    \node[ext, dotted, minimum size=10mm] (w) at (0,-.5) {};
    \draw (v) edge (w) (v.west) edge (w) (v.east) edge (w);
    \begin{scope} [yshift=-.6cm, scale=.4]
        \node[ext] (v) at (0,.5) {$\scriptstyle \Gamma_2$};
        \node[ext] (w) at (0,-.5) {};
        \draw (v) edge (w) (v.west) edge (w) (v.east) edge (w);
        \node[ext, fill=white] (v) at (0,.5) {$\scriptstyle \Gamma_2$};
    \end{scope}
\end{tikzpicture}
\]
The differential $\delta$ on $G_1$ is given by splitting vertices of graphs, as in  \eqref{equ:delta bullet def}, \eqref{equ:delta circ def}.
  \begin{align*}
    \delta:
    \begin{tikzpicture}[baseline=-.65ex]
        \node[int] (v) at (0,0) {};
        \draw (v) edge +(-.3,-.3)  edge +(-.3,0) edge +(-.3,.3) edge +(.3,-.3)  edge +(.3,0) edge +(.3,.3);
        \end{tikzpicture}
        &\mapsto\sum
        \begin{tikzpicture}[baseline=-.65ex]
        \node[int] (v) at (0,0) {};
        \node[int] (w) at (0.5,0) {};
        \draw (v) edge (w) (v) edge +(-.3,-.3)  edge +(-.3,0) edge +(-.3,.3)
         (w) edge +(.3,-.3)  edge +(.3,0) edge +(.3,.3);
        \end{tikzpicture}   
        &
        \begin{tikzpicture}[baseline=-.65ex]
          \node[ext] (v) at (0,0) {};
          \draw (v) edge +(-.3,-.3)  edge +(-.3,0) edge +(-.3,.3) edge +(.3,-.3)  edge +(.3,0) edge +(.3,.3);
          \end{tikzpicture}
          &\mapsto\sum
          \begin{tikzpicture}[baseline=-.65ex]
          \node[ext] (v) at (0,0) {};
          \node[int] (w) at (0.5,0) {};
          \draw (v) edge (w) (v) edge +(-.3,-.3)  edge +(-.3,0) edge +(-.3,.3)
           (w) edge +(.3,-.3)  edge +(.3,0) edge +(.3,.3);
          \end{tikzpicture}         
\end{align*}

Next, we define the vector space
$$
\sX := \Q \, \emptyset \, \oplus \, \bigoplus_{g\geq 1} X_{g}.
$$
It is generated by the generators of the $X_{g}$ of arbitrary genus, plus an additional generator $\emptyset$. We think of $\emptyset$ as the empty graph in the blown-up picture, or the graph with a single special vertex in the original picture. 

There is a left action of the dg algebra $G_1$ on the complex $\sX$
defined as follows:
\begin{equation}\label{equ:circ action def}
\begin{gathered}
G_1 \otimes \sX \to \sX \\
(\Gamma, \nu) \mapsto \Gamma \circ \nu = \sum 
\begin{tikzpicture}
    \node[ext] (v) at (0,.8) {$\Gamma$};
    \node[ext, dotted, minimum size=10mm] (w) at (0,-.5) {};
    \draw (v) edge (w) (v.west) edge (w) (v.east) edge (w);
    \begin{scope} [yshift=-.6cm, scale=.4]
        \node[ext] (v) at (0,.5) {$\scriptstyle \nu$};
          \end{scope}
\end{tikzpicture}
\end{gathered}
\end{equation}
Here again one inserts the graph $\nu$ into the external vertex of $\Gamma$, and reconnects the incident half-edges in all possible ways to vertices of $\nu$.
If a half-edge of $\Gamma$ is reconnected to the special vertex of $\nu$, then the rule is that the half-edge becomes non-marked. Equivalently, in the blown-up picture, each half-edge of $\Gamma$ incident to the special vertex is either reconnected to an internal vertex of $\nu$ or else decorated by $\epsilon$. In particular, if $\nu=\emptyset$, then one just labels all of the half-edges of $\Gamma$ incident to the external vertex by $\epsilon$.

\begin{lemma}\label{lem:circ well def}
The operation $\circ$ above is a well defined left action of the dg associative algebra $G_1$ on the dg vector space $\sX$.
\end{lemma}
\begin{proof}
  First, let us check that for graphs $\Gamma_1,\Gamma_2\in G_1$ and $\nu\in \sX$ we have
  \[
    \Gamma_1\circ(\Gamma_2\circ \nu) 
    =
    (\Gamma_1\circ\Gamma_2)\circ \nu. 
  \]
Each side is the sum of graphs obtained by reconnecting the half-edges incident at the external vertex of $\Gamma_1$ to internal vertices of $\Gamma_2$ or vertices of $\nu$, and half-edges incident to the external vertex of $\Gamma_2$ to vertices of $\nu$.
  Hence both sides agree.

  Next, we have to check compatibility with the differentials, that is 
  \[
  \delta(\Gamma\circ \nu) = (\delta\Gamma)\circ \nu + (-1)^{|\Gamma|} \Gamma \circ \delta \nu.
  \]
  This verification is straightforward. We outline the argument.  First, note that
  \[
    \delta_\omega(\Gamma\circ \nu) = (-1)^{|\Gamma|} \Gamma \circ \delta_\omega \nu,
  \]
  since $\Gamma \circ(-)$ does not interact with the decorations $\epsilon$ or $\omega$.
  It remains to show that 
  \begin{equation}\label{equ:deltas com tbs}
    \delta_s(\Gamma\circ \nu) = (\delta\Gamma)\circ \nu + (-1)^{|\Gamma|} \Gamma \circ \delta_s \nu.
  \end{equation}
  The terms appearing on either side are of one of four types:
  \begin{itemize}
    \item (A) Terms arising from splitting an internal vertex of $\Gamma$.
    Those are the same on both sides and can be ignored.
    \item (B) Terms arising from splitting a non-special vertex $v$ of $\nu$.
    On the left-hand side of \eqref{equ:deltas com tbs} the splitting is performed after (possibly) some edges of $\Gamma$ have been connected to $v$. Pictorially,
    \[
      \begin{tikzpicture}
        \node[int] (v) at (0,0) {};
        \node[int] (w) at (1,0) {};
        \draw (v) edge (w) edge +(-.5,.5) edge +(-.5, -.5) edge[dotted] +(0,.5)
        (w) edge +(+.5,.5) edge +(+.5, -.5) edge[dotted] +(0,.5);
      \end{tikzpicture}
    \]
    with the new edges drawn dotted.
    The same terms are produced on the right-hand side of \eqref{equ:deltas com tbs}, except for graphs for which there are 0 (say type (B0)) or 1 (say type (B1)) old edges incident to one of the vertices.
    \begin{align*}
      \text{(B0)}:
      &\begin{tikzpicture}
        \node[int] (v) at (0,0) {};
        \node[int] (w) at (1,0) {};
        \draw (v) edge (w) edge[dotted] +(-.5,.5) edge[dotted] +(-.5, -.5) edge[dotted] +(0,.5)
        (w) edge +(+.5,.5) edge +(+.5, -.5) edge[dotted] +(0,.5);
      \end{tikzpicture}
      &
      \text{(B1)}:
      &\begin{tikzpicture}
        \node[int] (v) at (0,0) {};
        \node[int] (w) at (1,0) {};
        \draw (v) edge (w) edge[dotted] +(-.5,.5) edge +(-.5, -.5) edge[dotted] +(0,.5)
        (w) edge +(+.5,.5) edge +(+.5, -.5) edge[dotted] +(0,.5);
      \end{tikzpicture}
    \end{align*}
    \item (C) There are terms arising from splitting the special vertex of $\nu$ on the left- and right-hand side of \eqref{equ:deltas com tbs}.
    They are handled just as those of type (B) above, and match except for terms of types
    \begin{align*}
      \text{(C0)}:
      &\begin{tikzpicture}
        \node[int] (v) at (0,0) {};
        \node[ext, accepting] (w) at (1,0) {};
        \draw (v) edge (w) edge[dotted] +(-.5,.5) edge[dotted] +(-.5, -.5) edge[dotted] +(0,.5)
        (w) edge +(+.5,.5) edge +(+.5, -.5) edge[dotted] +(0,.5);
      \end{tikzpicture}
      &
      \text{(C1)}:
      &\begin{tikzpicture}
        \node[int] (v) at (0,0) {};
        \node[ext, accepting] (w) at (1,0) {};
        \draw (v) edge (w) edge[dotted] +(-.5,.5) edge +(-.5, -.5) edge[dotted] +(0,.5)
        (w) edge +(+.5,.5) edge +(+.5, -.5) edge[dotted] +(0,.5);
      \end{tikzpicture}
    \end{align*}
    that appear on the left-hand side of \eqref{equ:deltas com tbs}, but (a priori) not on the right.
    \item (D) Finally, we have terms on the right-hand side of \eqref{equ:deltas com tbs} from splitting the external vertex of $\Gamma$. Those are of course absent from the left-hand side, because the operation $\Gamma\circ (-)$ removes the external vertex of $\Gamma$.
    However, upon inspection these terms of type (D) on the right-hand side exactly match the terms of type (B0) and (C0) above on the left-hand side.
  \end{itemize}
It remains to show that the yet unmatched terms (B1) and (C1) on the left-hand side of \eqref{equ:deltas com tbs} all cancel in pairs.
To see this mind that each such term appears twice, once from either side of the ``old" edge attached to the vertex that was split off:
\begin{align*}
  &\begin{tikzpicture}
    \node[int] (v) at (0,0) {};
    \node[int] (w) at (1,0) {};
    \node[int] (m) at (.3,0) {};
    \draw (v) edge (m) edge +(-.5,.5) edge +(-.5, -.5) 
    (w) edge (m) edge +(+.5,.5) edge +(+.5, -.5) 
    (m) edge[dotted] +(-.2,.5) edge[dotted] +(.2,.5);
  \end{tikzpicture}
  &\text{and}&
  &\begin{tikzpicture}
    \node[int] (v) at (0,0) {};
    \node[int] (w) at (1,0) {};
    \node[int] (m) at (.7,0) {};
    \draw (v) edge (m) edge +(-.5,.5) edge +(-.5, -.5) 
    (w) edge (m) edge +(+.5,.5) edge +(+.5, -.5) 
    (m) edge[dotted] +(-.2,.5) edge[dotted] +(.2,.5);
  \end{tikzpicture} 
  \\
  \text{or}\quad&\begin{tikzpicture}
    \node[int] (v) at (0,0) {};
    \node[ext, accepting] (w) at (1,0) {};
    \node[int] (m) at (.3,0) {};
    \draw (v) edge (m) edge +(-.5,.5) edge +(-.5, -.5) 
    (w) edge (m) edge +(+.5,.5) edge +(+.5, -.5) 
    (m) edge[dotted] +(-.2,.5) edge[dotted] +(.2,.5);
  \end{tikzpicture}
  &\text{and}&
  &\begin{tikzpicture}
    \node[int] (v) at (0,0) {};
    \node[ext, accepting] (w) at (1,0) {};
    \node[int] (m) at (.7,0) {};
    \draw (v) edge (m) edge +(-.5,.5) edge +(-.5, -.5) 
    (w) edge (m) edge +(+.5,.5) edge +(+.5, -.5) 
    (m) edge[dotted] +(-.2,.5) edge[dotted] +(.2,.5);
  \end{tikzpicture} \,.
\end{align*}
The two terms cancel in each case.  To see this, recall that the newly produced edge from the splitting is always the first in the ordering.  From this, one deduces that the two terms have opposite signs, and the lemma follows.
\end{proof}

Next, let $\gamma\in \GC_0$ be a graph.
Then we define the element $\gamma^1\in G_1$ by summing over all ways of making one vertex of $\gamma$ external.  For example:
\[
\gamma =
\begin{tikzpicture}[rotate=-45, yshift=.5cm, scale=.8]
    \node[int] (v1) at (0,0) {};
    \node[int] (v2) at (1,0) {};
    \node[int] (v3) at (0,1) {};
    \node[int] (v4) at (1,1) {};
    \draw (v1) edge (v2) edge (v3) edge (v4)
    (v2) edge (v3) edge (v4)
    (v3) edge (v4)
    (v4);
\end{tikzpicture} 
\ \mapsto  \ 
\gamma^1 = 
4\
\begin{tikzpicture}[rotate=-45, yshift=.5cm, scale=.8]
    \node[int] (v1) at (0,0) {};
    \node[ext] (v2) at (1,0) {};
    \node[int] (v3) at (0,1) {};
    \node[int] (v4) at (1,1) {};
    \draw (v1) edge (v2) edge (v3) edge (v4)
    (v2) edge (v3) edge (v4)
    (v3) edge (v4)
    (v4);
\end{tikzpicture} 
\]
We also define $\gamma^\circ\in G_1$ by attaching one edge and the external vertex.  Similarly, we define $\gamma^\epsilon\in \sX$ by adding one $\epsilon$-leg and $\gamma^\omega\in \sX$ by adding one $\omega$-leg.  For example, with $\gamma$ as above:
\[
    \gamma^\circ = 
    4\
    \begin{tikzpicture}[rotate=-45, yshift=.5cm, scale=.8]
        \node[int] (v1) at (0,0) {};
        \node[int] (v2) at (1,0) {};
        \node[int] (v3) at (0,1) {};
        \node[int] (v4) at (1,1) {};
        \node[ext] (e) at (1.7,-.7) {};
        \draw (v1) edge (v2) edge (v3) edge (v4)
        (v2) edge (v3) edge (v4) edge (e)
        (v3) edge (v4)
        (v4);
    \end{tikzpicture}    \quad, \quad 
        \gamma^\epsilon = 
        4\
        \begin{tikzpicture}[rotate=-45, yshift=.5cm, scale=.8]
            \node[int] (v1) at (0,0) {};
            \node[int] (v2) at (1,0) {};
            \node[int] (v3) at (0,1) {};
            \node[int] (v4) at (1,1) {};
            \node[] (e) at (1.7,-.7) {$\epsilon$};
            \draw (v1) edge (v2) edge (v3) edge (v4)
            (v2) edge (v3) edge (v4) edge (e)
            (v3) edge (v4)
            (v4);
        \end{tikzpicture} 
        \quad, \quad
\gamma^\omega = 
4\
\begin{tikzpicture}[rotate=-45, yshift=.5cm, scale=.8]
    \node[int] (v1) at (0,0) {};
    \node[int] (v2) at (1,0) {};
    \node[int] (v3) at (0,1) {};
    \node[int] (v4) at (1,1) {};
    \node[] (e) at (1.7,-.7) {$\omega$};
    \draw (v1) edge (v2) edge (v3) edge (v4)
    (v2) edge (v3) edge (v4) edge (e)
    (v3) edge (v4)
    (v4);
\end{tikzpicture}    \ .
\]
To fix the signs, the newly added edge becomes the first in the ordering of edges.
In the case of $\gamma^\omega$ the new (distinguished half-edge) $\omega$ stands right after the newly added edge in the ordering of edges and $\omega$s. Note that $\delta_\omega \gamma^\omega = -\gamma^\epsilon$.

\begin{lemma}\label{lem:some ops}
The above operations satisfy the following compatibility relations for $\gamma,\nu\in \GC_0$:
\begin{align}
    \label{equ:deltagamma1}
\delta (\gamma^1)- (\delta\gamma)^1 &= \gamma^\circ \\
\label{equ:poisbra1}
(\{\gamma_1,\gamma_2\})^1 &= \gamma_1^\circ \circ \gamma_2^1 + (-1)^{|\gamma_1||\gamma_2|}\gamma_2^\circ \circ \gamma_1^1
\end{align}
\end{lemma}
\begin{proof}
For \eqref{equ:deltagamma1} note that marking one vertex commutes with vertex splitting, except that splittings of the external vertex can be such that the external vertex has valence 1 or 2 afterwards.
The terms with valence 1 are $\gamma^\circ$. The terms of valence 2 cancel by an argument similar to that at the end of the proof of Lemma \ref{lem:circ well def}.

For \eqref{equ:poisbra1} note that 
$(\{\gamma_1,\gamma_2\})^1$ is a linear combination of graphs obtained by connecting $\gamma_1$ and $\gamma_1$ by one edge, and marking one vertex of $\gamma_1$ or $\gamma_2$ as external.
The terms in which the external vertex is in $\gamma_2$ are the same as $\gamma_1^\circ \circ \gamma_2^1$, and the terms in which the external vertex is in $\gamma_1$ are the same as $(-1)^{|\gamma_1||\gamma_2|}\gamma_2^\circ \circ \gamma_1^1$.
\end{proof}

For graphs $x_1,\dots, x_k\in \sX$ let us denote by $x_1\cup \cdots \cup x_k$ their union. 
For $\gamma\in \GC_0$ and $x\in \sX$ define the operation 
\[
\lambda(\gamma, x) := \gamma^\circ \circ x - \gamma^\epsilon \cup x.     
\]
Concretely, this operation connects $\gamma$ and $x$ by an edge $(v,w)$, summing over all vertices of $v$ and all \emph{non-special} vertices $w$ of $x$.
\[
\lambda(\gamma, x) =
\sum
\begin{tikzpicture}
\node[ext] (v) at (0,0.4) {$\gamma$};
\node[ext] (w) at (1,0.4) {$x$};
\node[ext, accepting] (ws) at (1,-.4) {};
\draw (v) edge (w) (ws) edge (w) edge (w.west) edge (w.east);
\end{tikzpicture}
\]

\begin{lemma}
The operation $\lambda(\gamma,-)$ is a derivation, i.e., 
\begin{equation}\label{equ:la derivation}
 \lambda(\gamma,x_1\cup \cdots \cup x_k) 
 =
 \sum_{j=1}^k \pm x_1\cup \cdots \cup \lambda(\gamma,x_j) \cup \cdots\cup x_k.
\end{equation}
Furthermore, for $\gamma,\nu \in \GC_0$ 
\begin{align}
  \label{equ:la bracket1}
   \{\gamma,\nu\}^\epsilon &=  
-\lambda(\gamma,\nu^\epsilon) - (-1)^{|\gamma||\nu|}\lambda(\nu,\gamma^\epsilon), \mbox{ and }    \\  
\label{equ:la bracket2}
\{\gamma,\nu\}^\omega &=  
\lambda(\gamma,\nu^\omega) + (-1)^{|\gamma||\nu|}\lambda(\nu,\gamma^\omega)\, . 
\end{align}
\end{lemma}
\begin{proof}
The derivation property \eqref{equ:la derivation} is clear from the pictorial description of $\lambda$ above the lemma.

We hence focus on \eqref{equ:la bracket1}, \eqref{equ:la bracket2}.
The element $\{\gamma,\nu\}^\epsilon$ is obtained by connecting $\gamma$ and $\nu$ by one edge and adding an  $\epsilon$-leg at a vertex of either $\gamma$ or $\nu$.
\[
  \{\gamma,\nu\}^\epsilon
  =
  \begin{tikzpicture}
    \node[ext] (v) at (0,0.4) {$\gamma$};
    \node[ext] (w) at (1,0.4) {$\nu$};
    \node[] (ws) at (1,-.4) {$\epsilon$};
    \draw (v) edge (w) (ws) edge (w);
    \end{tikzpicture}
    +
    \begin{tikzpicture}
      \node[ext] (v) at (0,0.4) {$\gamma$};
      \node[ext] (w) at (1,0.4) {$\nu$};
      \node[] (ws) at (0,-.4) {$\epsilon$};
      \draw (v) edge (w) (ws) edge (v);
      \end{tikzpicture}
\]
To fix the sign, note that in the implicit ordering of edges, the $\epsilon$-edge is first, then the horizontal edge, then the edges of $\gamma$, then those of $\nu$.
On the other hand 
\begin{align*}
  \lambda(\gamma,\nu^\epsilon)  
  =
  \begin{tikzpicture}
    \node[ext] (v) at (0,0.4) {$\gamma$};
    \node[ext] (w) at (1,0.4) {$\nu$};
    \node[] (ws) at (1,-.4) {$\epsilon$};
    \draw (v) edge (w) (ws) edge (w);
  \end{tikzpicture}
\end{align*}
with the ordering of edges such that the horizontal edge is first and the $\epsilon$-edge second, then the edges of $\gamma$, then those of $\nu$. Finally 
\begin{align*}
  \lambda(\nu,\gamma^\epsilon)  
  =
  \begin{tikzpicture}
    \node[ext] (v) at (0,0.4) {$\nu$};
    \node[ext] (w) at (1,0.4) {$\gamma$};
    \node[] (ws) at (1,-.4) {$\epsilon$};
    \draw (v) edge (w) (ws) edge (w);
  \end{tikzpicture}
\end{align*}
with the analogous ordering of edges.
Hence \eqref{equ:la bracket1} follows.
Equation \eqref{equ:la bracket2} is shown similarly, with the caveat that there is an extra sign due to the $\omega$-decoration being of odd degree, i.e., the assignment $\gamma\mapsto \gamma^\omega$ is of even degree zero.
\end{proof}

\subsection{The map $G$} Let $J = \{j_1, \ldots, j_k\}$ denote an ordered subset of $\{1, \ldots, n\}$, i.e., with indices chosen so that $j_1 < \cdots < j_k$.
We define a linear map
\[
G \colon (\Sym(\GC_0),\delta+\delta_{\{,\}}) \to \sX
\]
by the formula 
\begin{align*}
G(\gamma_1\cdots\gamma_n) 
&=
\sum_{|J| = k} \,
\sum_{\sigma\in \bS_{k}} \pm 
\gamma_{j_{\sigma(1)}}^1 \circ   \cdots  \circ \gamma_{j_{\sigma(k)}}^1 \circ
\cup_{i \notin J} \gamma_{i}^\omega
\\&=
\cup_{i=1}^n \gamma_{i}^\omega
+
\sum_{\ell=1}^n \pm 
\gamma_\ell^1\circ (\cup_{i=1\atop i\neq \ell}^n \gamma_{i}^\omega)  
+\cdots
.
\end{align*}
\noindent The sign reflects the permutation of the symbols $\gamma_j$ in the formula. The term with $J=\{1,\dots,k\}$ and $\sigma$ the identity comes with sign $+$. This, together with the conventions on permuting factors in the symmetric product discussed in \S\ref{sec:dg preliminaries}, determines all of the other signs.

\begin{prop}
The linear map $G$ is a map of dg vector spaces, i.e. $(\delta_s + \delta_\omega)\circ G=G\circ (\delta+\delta_{\{,\}}).$
\end{prop}
\begin{proof}
We have that $\delta_\omega \gamma^\omega = -\gamma^\epsilon$ and hence
\begin{align*}
\delta_\omega
G(\gamma_1\cdots\gamma_n) 
&=
\delta_\omega
\sum_{|J| = k } \, 
\sum_{\sigma\in \bS_{k}} \pm 
\gamma_{j_{\sigma(1)}}^1 \circ   \cdots  \circ \gamma_{j_{\sigma(k)}}^1 \circ
\cup_{i\notin J} \gamma_{i}^\omega
\\&=-
\sum_{|J| = k } \, 
\sum_{\sigma\in \bS_{k}} \pm
\gamma_{j_{\sigma(1)}}^1 \circ   \cdots  \circ \gamma_{j_{\sigma(k)}}^1 \circ
\sum_{\ell \notin J}
\gamma_\ell ^\epsilon \cup_{i \notin J\atop i \neq \ell}  \gamma_{i}^\omega.
\end{align*}
On the other hand, using \eqref{equ:deltagamma1} above
\[
\resizebox{.97\hsize}{!}{
$\displaystyle{
\delta_s
G(\gamma_1\cdots\gamma_n)  
-
\sum_{r=1}^n \pm   G(\gamma_1\cdots\delta\gamma_r\cdots\gamma_n)
=
\sum_{|J| = k } \, 
\sum_{\sigma\in \bS_{k}} \pm
\sum_{\ell=1}^k
\pm 
\gamma_{j_{\sigma(1)}}^1 \circ   \cdots \gamma_{j_{\sigma(\ell)}}^\circ \cdots   \circ \gamma_{j_{\sigma(k)}}^1 \circ
\cup_{i\notin J} \gamma_{i}^\omega\,.
}$}
\]
Next, using \eqref{equ:poisbra1} and \eqref{equ:la bracket2}, we have 
\[
\begin{split}
    G(\delta_{\{,\}}(\gamma_1\cdots\gamma_n) )
&=
   \frac12 \sum_{|J| = k} \, \sum_{r=1}^{k-1} 
\sum_{\sigma\in \bS_{k}} \pm 
\gamma_{j_{\sigma(1)}}^1 \circ   \cdots \{\gamma_{j_{\sigma(r)}},\gamma_{j_{\sigma(r+1)}} \}^1 \cdots  \circ \gamma_{j_{\sigma(k)}}^1 \circ
\cup_{i\notin J}  \gamma_{i}^\omega \\
& \quad +
   \frac12 \sum_{|J| = k} \, 
\sum_{\sigma\in \bS_{k}} \pm 
\gamma_{j_{\sigma(1)}}^1 \circ \cdots \circ \gamma_{j_{\sigma(k)}}^1 \circ
\sum_{a,b\notin J} \{\gamma_a,\gamma_b\}^\omega
\cup_{i\notin J \atop i\neq a,b}  \gamma_{i}^\omega
\\&=
 \sum_{|J| = k} \, 
\sum_{\sigma\in \bS_{k}} \pm 
\gamma_{j_{\sigma(1)}}^1 \circ   \cdots \gamma_{j_{\sigma(r)}}^\circ \circ \gamma_{j_{\sigma(r+1)}}^1 \cdots  \circ \gamma_{j_{\sigma(k)}}^1 \circ
\cup_{i\notin J }  \gamma_{i}^\omega \\ 
& \quad
+
 \sum_{|J| = k} \, 
\sum_{\sigma\in \bS_{k}} \pm 
\gamma_{j_{\sigma(1)}}^1 \circ \cdots \circ \gamma_{j_{\sigma(k)}}^\circ \circ
\cup_{i\notin J}  \gamma_{i}^\omega \\
& \quad -
 \sum_{|J| = k} \, 
\sum_{\sigma\in \bS_{k}}  \pm 
\gamma_{j_{\sigma(1)}}^1 \circ \cdots \circ \gamma_{j_{\sigma(k)}}^1 \circ
\sum_{a\notin J}
\gamma_a^\epsilon
\cup_{i\notin J \atop a\neq i }   \gamma_{i}^\omega
.
\end{split}
\]
The proposition follows by summing the expressions above. 
\end{proof}

\subsection{Operation $\nabla$}
Consider the degree +1 operation 
\begin{gather*}
\nabla \colon \sX[-1] \to \sX \\
\nu \mapsto \sum 
\begin{tikzpicture}
  \node[ext, dotted, minimum size=7mm] (w) at (0,-.2) {$\scriptstyle \nu$};
  \draw (w) edge[loop above] (w);
\end{tikzpicture},
\end{gather*}
defined by summing over all ways of attaching one new edge to the graph $\nu$, between an arbitrary pair of vertices.
Comparing to \eqref{equ:circ action def} this can be identified as the extension of the action $\circ$ to the tadpole graph 
\[
    \tau' = 
\begin{tikzpicture}
\node[ext] (v) {};
\draw (v) edge[loop above] (v);
\end{tikzpicture}
\]
so that 
\[
  \nabla \nu = \tau' \circ \nu.
\]
As in Lemma \ref{lem:circ well def} we then have that
$$
\delta \nabla\nu + \nabla \delta \nu=
(\delta \tau')\circ \nu =0,
$$
since $\delta \tau'=
\begin{tikzpicture}
    \node[ext] (v) at (0,-.3) {};
    \node[int] (w) at (0,.3) {};
    \draw (w) edge[loop above] (w) edge (v);
\end{tikzpicture}
$ has a tadpole at an internal vertex and thus acts as zero.
Hence the operation $\nabla$ anti-commutes with the differential on $\sX$ and is a map of complexes.

\subsection{Construction of the map $F$ of Theorem \ref{thm:emb1}}

We define a map  $\Sym(\GC_0) \to X$ as the composition of the isomorphism $\Phi$ of Corollary~\ref{cor:CE map} and the map $G$ above:
\[
   (\Sym(\GC_0),\delta) 
   \xrightarrow{\Phi}
   (\Sym(\GC_0),\delta+\delta_{\{ , \} })  
   \xrightarrow{G} \sX.
\]
Note that if the argument on the left-hand side is in the subspace $S^{10}(\GC_0)$, then the image lies in the subspace $C[21]\subset X$ spanned by graphs with at most $10$ $\omega$-legs.
We hence can define the map $F$ of Theorem \ref{thm:emb1} by restricting $G\circ \Phi$. More precisely, $F$ is the composition 
\[
F : \Sym^{10}(\GC_0)^{(g-1)}[-11] \oplus \Sym^{10}(\GC_0)^{(g-2)}[-12]
\xrightarrow{(G\circ \Phi) \oplus (G\circ \Phi)}
C_{g} \oplus C_{g-1}[-1]
\xrightarrow{\mathit{id} + \nabla} C_{g}.
\]

\subsection{Hairy graph complex and $(\sX,\delta_s)$}
Our final goal is to show the injectivity claim of Theorem \ref{thm:emb1}.
To this end it will be necessary to study the cohomology of the complex $(X,\delta_s)$. This can equivalently be identified with the associated graded complex of $(\sX,\delta_\omega+\delta_s)$ under the filtration by number of $\omega$-legs.

We consider a graph complex $\fHGC$ generated by linear combinations of pairs $(\Gamma,o)$ with $\Gamma$ a possibly disconnected tadpole-free graph with (non-numbered) external legs. 
\[
  \begin{tikzpicture}
    \node[int] (v1) at (45:.6) {};
    \node[int] (v2) at (135:.6) {};
    \node[int] (v3) at (-135:.6) {};
    \node[int] (v4) at (-45:.6) {};
    \draw (v1) edge +(45:.5) edge (v2) edge (v3) edge (v4) (v2) edge (v3) edge (v4) (v3) edge (v4) (v4) edge +(-45:.5);
  \end{tikzpicture} 
  \,
  \begin{tikzpicture}
    \node[int] (v1) at (0,0) {};
    \draw (v1)  edge +(90:.5)  edge +(210:.5)  edge +(-30:.5) ;
  \end{tikzpicture}   
\]
We require that each vertex has valence $\geq 3$. The orientation $o=e_1\wedge\cdots \wedge e_k$ is an ordering of the set of structural (i.e., non-leg) edges, and we again identify isomorphic graphs and orderings up to sign, cf. \eqref{equ:bGK gen rel 1}, \eqref{equ:bGK gen rel 2}. By convention, we allow the graph $\Gamma$ to be the empty graph for convenience, but we forbid connected components that are just a single edge and do not contain a vertex.
The differential $\delta$ on $\fHGC$ is again given by vertex splitting 
  \begin{align*}
    \delta:
    \begin{tikzpicture}[baseline=-.65ex]
        \node[int] (v) at (0,0) {};
        \draw (v) edge +(-.3,-.3)  edge +(-.3,0) edge +(-.3,.3) edge +(.3,-.3)  edge +(.3,0) edge +(.3,.3);
        \end{tikzpicture}
        &\mapsto\sum
        \begin{tikzpicture}[baseline=-.65ex]
        \node[int] (v) at (0,0) {};
        \node[int] (w) at (0.5,0) {};
        \draw (v) edge (w) (v) edge +(-.3,-.3)  edge +(-.3,0) edge +(-.3,.3)
         (w) edge +(.3,-.3)  edge +(.3,0) edge +(.3,.3);
        \end{tikzpicture}  \, .       
\end{align*}
The degree of a graph is the number of structural edges.
The graph complex $(\fHGC,\delta)$ is well known in the literature. By \cite{CGP2} the cohomology of the connected part of loop order $g$ with $n$ legs (for $2g+n\geq 3$) computes the symmetric weight 0 part of the compactly supported cohomology of the moduli spaces of curves $W_0H_c^\bullet(\M_{g,n})_{\bbS_n}$. Closely related complexes also compute the rational homology of spaces of long knots \cite{AroneTurchin}.

We then define a map of dg vector spaces
\begin{gather}\label{equ:K def}
    K \colon (\Q\oplus \Q\alpha \oplus \Q\beta \oplus \Q\alpha\beta )\otimes \fHGC \to (\sX, \delta_s), 
\end{gather}
with $\alpha$ representing an $(\epsilon-\epsilon)$-edge and $\beta$ representing an $(\epsilon-\omega)$-edge.
More concretely, the map $K$ acts on the four summands as follows:
\begin{itemize}
\item On the first summand $\fHGC$ the map $K$ just acts as the natural inclusion:
\[
\begin{tikzpicture}[yshift=-.5cm]
    \node[int] (v) at (0,.5) {};
    \node (w1) at (-.5,-.2) {};
    \node (w2) at (0,-.2) {};
    \node (w3) at (.5,-.2) {};
    \draw (v) edge (w1) edge (w2) edge (w3);
\end{tikzpicture}
\begin{tikzpicture}
    \node[int] (v1) at (45:.6) {};
    \node[int] (v2) at (135:.6) {};
    \node[int] (v3) at (-135:.6) {};
    \node[int] (v4) at (-45:.6) {};
    \node (w1) at (45:1.2) {};
    \node (w2) at (-45:1.2) {};
    \draw (v1) edge (w1) edge (v2) edge (v3) edge (v4) (v2) edge (v3) edge (v4) (v3) edge (v4) (v4) edge (w2);
  \end{tikzpicture}
\mapsto  
\begin{tikzpicture}[yshift=-.5cm]
    \node[int] (v) at (0,.5) {};
    \node (w1) at (-.5,-.2) {$\omega$};
    \node (w2) at (0,-.2) {$\omega$};
    \node (w3) at (.5,-.2) {$\omega$};
    \draw (v) edge (w1) edge (w2) edge (w3);
\end{tikzpicture}
\begin{tikzpicture}
    \node[int] (v1) at (45:.6) {};
    \node[int] (v2) at (135:.6) {};
    \node[int] (v3) at (-135:.6) {};
    \node[int] (v4) at (-45:.6) {};
    \node (w1) at (45:1.2) {$\omega$};
    \node (w2) at (-45:1.2) {$\omega$};
    \draw (v1) edge (w1) edge (v2) edge (v3) edge (v4) (v2) edge (v3) edge (v4) (v3) edge (v4) (v4) edge (w2);
  \end{tikzpicture}  \quad .
\]
\item On $\Q\alpha\otimes \fHGC$ the map $K$ applies the natural inclusion followed by $\nabla$:
\[
\begin{tikzpicture}[yshift=-.5cm]
    \node[int] (v) at (0,.5) {};
    \node (w1) at (-.5,-.2) {};
    \node (w2) at (0,-.2) {};
    \node (w3) at (.5,-.2) {};
    \draw (v) edge (w1) edge (w2) edge (w3);
\end{tikzpicture}
\begin{tikzpicture}
    \node[int] (v1) at (45:.6) {};
    \node[int] (v2) at (135:.6) {};
    \node[int] (v3) at (-135:.6) {};
    \node[int] (v4) at (-45:.6) {};
    \node (w1) at (45:1.2) {};
    \node (w2) at (-45:1.2) {};
    \draw (v1) edge (w1) edge (v2) edge (v3) edge (v4) (v2) edge (v3) edge (v4) (v3) edge (v4) (v4) edge (w2);
  \end{tikzpicture}
\mapsto  
\nabla
\left(
\begin{tikzpicture}[yshift=-.5cm]
    \node[int] (v) at (0,.5) {};
    \node (w1) at (-.5,-.2) {$\omega$};
    \node (w2) at (0,-.2) {$\omega$};
    \node (w3) at (.5,-.2) {$\omega$};
    \draw (v) edge (w1) edge (w2) edge (w3);
\end{tikzpicture}
\begin{tikzpicture}
    \node[int] (v1) at (45:.6) {};
    \node[int] (v2) at (135:.6) {};
    \node[int] (v3) at (-135:.6) {};
    \node[int] (v4) at (-45:.6) {};
    \node (w1) at (45:1.2) {$\omega$};
    \node (w2) at (-45:1.2) {$\omega$};
    \draw (v1) edge (w1) edge (v2) edge (v3) edge (v4) (v2) edge (v3) edge (v4) (v3) edge (v4) (v4) edge (w2);
  \end{tikzpicture} 
  \right) \quad .
\]
\item On $\Q\beta\otimes \fHGC$ the map $K$ adds one $(\epsilon-\omega)$-edge, plus connects an additional $\omega$-leg. 
\begin{align*}
    \begin{tikzpicture}
        \node[int] (v1) at (45:.6) {};
        \node[int] (v2) at (135:.6) {};
        \node[int] (v3) at (-135:.6) {};
        \node[int] (v4) at (-45:.6) {};
        \node (w1) at (45:1.2) {};
        \node (w2) at (-45:1.2) {};
        \draw (v1) edge (w1) edge (v2) edge (v3) edge (v4) (v2) edge (v3) edge (v4) (v3) edge (v4) (v4) edge (w2);
      \end{tikzpicture}
      &\mapsto
    \begin{tikzpicture}
        \node (x1) at (-2,0) {$\omega$};
        \node (x2) at (-1,0) {$\epsilon$};
        \node[int] (v1) at (45:.6) {};
        \node[int] (v2) at (135:.6) {};
        \node[int] (v3) at (-135:.6) {};
        \node[int] (v4) at (-45:.6) {};
        \node (w1) at (45:1.2) {$\omega$};
        \node (w2) at (-45:1.2) {$\omega$};
        \draw (v1) edge (w1) edge (v2) edge (v3) edge (v4) (v2) edge (v3) edge (v4) (v3) edge (v4) (v4) edge (w2) (x1) edge (x2);
      \end{tikzpicture} 
      +
      \sum 
      \begin{tikzpicture}
        \node (x1) at (-2,0) {$\omega$};
        \node[ext, minimum size=1.4cm, dashed] (x2) at (0,0) {};
        \node[int] (v1) at (45:.6) {};
        \node[int] (v2) at (135:.6) {};
        \node[int] (v3) at (-135:.6) {};
        \node[int] (v4) at (-45:.6) {};
        \node (w1) at (45:1.2) {$\omega$};
        \node (w2) at (-45:1.2) {$\omega$};
        \draw (v1) edge (w1) edge (v2) edge (v3) edge (v4) (v2) edge (v3) edge (v4) (v3) edge (v4) (v4) edge (w2) (x1) edge (x2);
      \end{tikzpicture} 
      \\&=
      \begin{tikzpicture}
        \node (x1) at (-2,0) {$\omega$};
        \node (x2) at (-1,0) {$\epsilon$};
        \node[int] (v1) at (45:.6) {};
        \node[int] (v2) at (135:.6) {};
        \node[int] (v3) at (-135:.6) {};
        \node[int] (v4) at (-45:.6) {};
        \node (w1) at (45:1.2) {$\omega$};
        \node (w2) at (-45:1.2) {$\omega$};
        \draw (v1) edge (w1) edge (v2) edge (v3) edge (v4) (v2) edge (v3) edge (v4) (v3) edge (v4) (v4) edge (w2) (x1) edge (x2);
      \end{tikzpicture} 
      +
      2
      \begin{tikzpicture}
        \node (x1) at (135:1.2) {$\omega$};
        \node[int] (v1) at (45:.6) {};
        \node[int] (v2) at (135:.6) {};
        \node[int] (v3) at (-135:.6) {};
        \node[int] (v4) at (-45:.6) {};
        \node (w1) at (45:1.2) {$\omega$};
        \node (w2) at (-45:1.2) {$\omega$};
        \draw (v1) edge (w1) edge (v2) edge (v3) edge (v4) (v2) edge (v3) edge (v4) (v3) edge (v4) (v4) edge (w2) 
        (x1) edge (v2);
      \end{tikzpicture} 
      +2 \ \ 
      \begin{tikzpicture}
        \node (x1) at (0,.85) {$\omega$};
        \node[int] (v1) at (45:.6) {};
        \node[int] (v2) at (135:.6) {};
        \node[int] (v3) at (-135:.6) {};
        \node[int] (v4) at (-45:.6) {};
        \node (w1) at (45:1.2) {$\omega$};
        \node (w2) at (-45:1.2) {$\omega$};
        \draw (v1) edge (w1) edge (v2) edge (v3) edge (v4) (v2) edge (v3) edge (v4) (v3) edge (v4) (v4) edge (w2) 
        (x1) edge (v1);
      \end{tikzpicture} 
\end{align*}
      Here the sum is over all ways of connecting a new $\omega$-leg to an internal vertex. 
      \item On the summand $\Q\alpha\beta\otimes \fHGC$ the map $K$ acts as above, followed by $\nabla$:
    \[
        \begin{tikzpicture}
            \node[int] (v1) at (45:.6) {};
            \node[int] (v2) at (135:.6) {};
            \node[int] (v3) at (-135:.6) {};
            \node[int] (v4) at (-45:.6) {};
            \node (w1) at (45:1.2) {};
            \node (w2) at (-45:1.2) {};
            \draw (v1) edge (w1) edge (v2) edge (v3) edge (v4) (v2) edge (v3) edge (v4) (v3) edge (v4) (v4) edge (w2);
          \end{tikzpicture}
          \mapsto
          \nabla
          \left(
        \begin{tikzpicture}
            \node (x1) at (-2,0) {$\omega$};
            \node (x2) at (-1,0) {$\epsilon$};
            \node[int] (v1) at (45:.6) {};
            \node[int] (v2) at (135:.6) {};
            \node[int] (v3) at (-135:.6) {};
            \node[int] (v4) at (-45:.6) {};
            \node (w1) at (45:1.2) {$\omega$};
            \node (w2) at (-45:1.2) {$\omega$};
            \draw (v1) edge (w1) edge (v2) edge (v3) edge (v4) (v2) edge (v3) edge (v4) (v3) edge (v4) (v4) edge (w2) (x1) edge (x2);
          \end{tikzpicture} 
          +
          \sum 
          \begin{tikzpicture}
            \node (x1) at (-2,0) {$\omega$};
            \node[ext, minimum size=1.4cm, dashed] (x2) at (0,0) {};
            \node[int] (v1) at (45:.6) {};
            \node[int] (v2) at (135:.6) {};
            \node[int] (v3) at (-135:.6) {};
            \node[int] (v4) at (-45:.6) {};
            \node (w1) at (45:1.2) {$\omega$};
            \node (w2) at (-45:1.2) {$\omega$};
            \draw (v1) edge (w1) edge (v2) edge (v3) edge (v4) (v2) edge (v3) edge (v4) (v3) edge (v4) (v4) edge (w2) (x1) edge (x2);
          \end{tikzpicture} 
          \right)
          \]
\end{itemize}

\begin{prop}\label{prop:map K}
The map $K$ is a quasi-isomorphism of dg vector spaces.
\end{prop}
\begin{proof}
First one checks that $K$ commutes with the differentials.
For the summands $\fHGC$ and  
$\Q\beta\otimes \fHGC$ this is straightforward. For the other two summands, one uses the fact that $\nabla$ commutes with the differentials.

It remains to check that the dg map $K$ is a quasi-isomorphism.
That is, we want to check that the mapping cone of $K$ is acyclic.
We filter both domain and target, and hence also the mapping cone by the number of connected components of graphs.
On the associated graded we see only those parts of the differential that leave the number of connected components the same. 
A close variant of the resulting complex has been studied by Turchin and the second author \cite{TWspherical}. We recall their main result in Appendix \ref{app:TW recollection} below, along with a slight variation  
(Corollary \ref{cor:TW tadpolefree}) that implies that the $E^1$ page of our spectral sequence has the form 
\[
E^1 = ( \Q T \otimes \Sym^{\geq 0}\Q L^\omega \oplus   \Sym^{\geq 2}\Q L^\omega ) \otimes (\Q \oplus \Q L^\epsilon) \otimes H(\fHGC)  \Rightarrow H(cone(K)),
\]
with $T, L^\epsilon, L^\omega$ the graphs of Appendix \ref{app:TW recollection}.
The differential on the $E^1$ page corresponds to those terms of $\delta_s$ that reduce the number of connected components by exactly one.
The key observation is that the component mapping the first tensor factor above to itself, 
\[
  \Sym^{\geq 2}\Q L^\omega  \to  \Q T \otimes \Sym^{\geq 0}\Q L^\omega 
\]
has the form 
\[
 \underbrace{L^\omega\cup \cdots  \cup L^\omega}_{k} 
 \mapsto 
 -\binom{k}{2} \underbrace{L^\omega\cup \cdots  \cup L^\omega}_{k-2}  \cup T.
\]
This obviously makes the first tensor factor acyclic.
By a simple spectral sequence argument it then follows that $H(E^1)=0$ and hence $H(cone(K))=0$ as desired.
\end{proof}

\subsection{Injectivity on cohomology and proof of Theorem \ref{thm:emb1}}

The projection to the subspace $C_{g,10\omega}\subset C_{g}$ spanned by graphs with exactly 10 $\omega$-legs is a map of dg vector spaces
\[
\pi \colon (C_g,\delta_s+\delta_\omega) \to (C_{g,10\omega},\delta_s).
\]
To show that $F$ induces an injective map on cohomology it hence suffices to show that $\pi\circ F$ induces an injective map on cohomology.
But $\pi\circ F$ is precisely the same as the map $K$ from the previous section restricted to a subspace of the $10$-hair part 
of the summand $(\Q\oplus\Q\alpha)\otimes \fHGC$.
More precisely, $\pi\circ F$ fits into a commutative diagram
\[
\adjustbox{scale=.93,center}{
\begin{tikzcd}
  \Sym^{10}(\GC_0)[-11] \oplus \Sym^{10}(\GC_0)[-12]
  \ar {rr}{\pi\circ F} \ar{dr}{\iota\oplus \iota}
  & & \bigoplus_g C_{g,10\omega} \\
  &\left( \fHGC[-11] \oplus \Q\alpha\otimes \fHGC[-12]\right)_{\text{10-hair}} \ar{ur}{K} & 
\end{tikzcd}
}
\]
with the map $\iota$ sending $\gamma_1\cdots\gamma_{10}\in  \Sym^{10}(\GC_0)$ to the hairy graph 
\[
\begin{tikzpicture}
\node[ext] (v) at (0,.3) {$\gamma_1$};
\draw (v) edge +(0,-.6);
\end{tikzpicture}
\cdots 
\begin{tikzpicture}
  \node[ext] (v) at (0,.3) {$\gamma_{10}$};
  \draw (v) edge +(0,-.6);
\end{tikzpicture}
\in \left(\fHGC\right)_{\text{10-hair}}
\]
The map $\iota$ induces an injection on cohomology by \cite[Theorem 1]{TWcommutative}.

But by Proposition \ref{prop:map K} the (restriction of the) map $K$ is also an injection on cohomology, and hence so is $\pi\circ F$.
\hfill\qed

\section{Case $n=0$ -- second injection} \label{sec:second inj}

In this section we shall describe a second family of nontrivial cohomology classes in $\gr_{11}H_c^\bullet(\M_g)$ that are built from cocycles in $\GC_0$.
More concretely, we will show the following result.

\begin{thm}\label{thm:emb 2}
There is a 
 map of dg vector spaces 
\[
E \colon 
\Sym^{9}(\GC_0)^{(g-3)}[-22]
\oplus
\Sym^{6}(\GC_0)^{(g-5)}[-22]
\oplus
\Sym^{3}(\GC_0)^{(g-7)}[-22]
\to
B_g
\]
that gives rise to an injective 
map on cohomology
\begin{align*}
    E\colon 
&H^{k-22}(\Sym^{9}(\GC_0))^{(g-3)} 
\oplus
H^{k-22}(\Sym^{6}(\GC_0))^{(g-5)} 
\oplus
H^{k-22}(\Sym^{3}(\GC_0))^{(g-7)} 
\to H^k(B_g).
\end{align*}
\end{thm}

\noindent The remainder of this section is concerned with the construction of the map $E$ and the proof that it induces an injection on cohomology. 

Since the construction is fairly technical and ad hoc, we shall first describe the idea here.
Recall that in the previous section we were able to construct explicit cocycles (say $x\in C_g$) in the  graph complex $(C_g, \delta_s+\delta_\omega)$ consisting of graphs with at most ten $\omega$-legs.
The idea to show that $x$ indeed represents a non-trivial cohomology class was to consider the projection $\pi: (C_g, \delta_s+\delta_\omega)\to (C_{g,10\omega}, \delta_{s})$, and use that the cohomology of the latter complex is computable. Concretely,by Proposition \ref{prop:map K} it agrees with the genus $g$- and $10$-hair part of 
\begin{gather}\label{equ:motivation C10}
  (\Q\oplus \Q\alpha \oplus \Q\beta \oplus \Q\alpha\beta )\otimes H(\fHGC),
\end{gather}
and we know many non-trivial classes in the hairy graph cohomology $H(\fHGC)$ from previous work in the literature.

Let us split $x=\sum_{j=0}^{10}x_j$ into components $x_j$ with $j$ many $\omega$-legs, then the top piece $x_{10}$ represents a non-trivial cohomology class in $C_{g,10\omega}$ (i.e., in \eqref{equ:motivation C10}).
Conversely, we may ask for a given cocycle $x_{10}\in C_{g,10\omega}$ whether it can be extended into 
a cocycle $x_{10}+x_9+\dots+x_0\in (C_g, \delta_s+\delta_\omega)$. Unfortunately, this extension problem is non-trivial, and we could only provide a solution for specific types of $x_{10}$ in the previous section above, and these come from the summand $\Q\oplus \Q\alpha$ in \eqref{equ:motivation C10}. 
The idea underlying Theorem \ref{thm:emb 2} is to consider the summand $\Q \alpha\beta$ instead.
In that case the most natural approach turned out to not try to construct the $x_0,\dots,x_9$, but rather construct a cocycle $x_{12}+x_{11}\in (B_g, \delta_s+\delta_\omega)$ whose image under the map $\delta_\omega:B_g\to C_g$ is the required $x_{10}$. 
This is the idea of the construction of the present section.
The advantage is that we only need to consider two summands, $x_{12}$ and $x_{11}$, and they will have a more natural combinatorial form than $\delta_\omega(x_{12}+x_{11})$.

\subsection{Some combinatorial constructions}
Our map $E$ will be a linear combination of several pieces, that we shall introduce next.
First, for $k=0,1,\dots$ we consider the maps
\begin{gather*}
\Phi_k \colon \Sym(\GC_0) \to \sX \\
\Phi_k(\gamma_1\cdots \gamma_r) = 
\underbrace{\tau\cup \cdots\cup \tau}_{k} \cup \gamma_1^\omega \cup\cdots \cup\gamma_r^\omega 
=
\begin{tikzpicture}
    \node[int] (i) at (0,.5) {};
    \node (v1) at (-.5,-.2) {$\omega$};
    \node (v2) at (0,-.2) {$\omega$};
    \node (v3) at (.5,-.2) {$\omega$};
  \draw (i) edge (v1) edge (v2) edge (v3);
  \end{tikzpicture}
  \cdots 
  \begin{tikzpicture}
    \node[int] (i) at (0,.5) {};
    \node (v1) at (-.5,-.2) {$\omega$};
    \node (v2) at (0,-.2) {$\omega$};
    \node (v3) at (.5,-.2) {$\omega$};
  \draw (i) edge (v1) edge (v2) edge (v3);
  \end{tikzpicture}  
  \begin{tikzpicture}
    \node[ext] (i) at (0,.5) {$\gamma_1$};
    \node (v1) at (0,-.2) {$\omega$};
  \draw (i) edge (v1);
  \end{tikzpicture}
  \cdots 
  \begin{tikzpicture}
    \node[ext] (i) at (0,.5) {$\gamma_r$};
    \node (v1) at (0,-.2) {$\omega$};
  \draw (i) edge (v1);
  \end{tikzpicture}
  ,
\end{gather*}
with 
\[
\tau = 
\begin{tikzpicture}
    \node[int] (i) at (0,.5) {};
    \node (v1) at (-.5,-.2) {$\omega$};
    \node (v2) at (0,-.2) {$\omega$};
    \node (v3) at (.5,-.2) {$\omega$};
  \draw (i) edge (v1) edge (v2) edge (v3);
  \end{tikzpicture}
\]
the tripod graph.
Since the differential $\delta_s$ distributes over the operation $\cup$ as long as there are no $\epsilon$-legs we have that 
\begin{equation}\label{equ:Phik deltas}
\delta_s \Phi_k(\gamma_1\cdots \gamma_r)
= 
\sum_{j=1}^r (-1)^{|\gamma_1|+\cdots+|\gamma_{j-1}|} \Phi_k(\gamma_1\cdots (\delta\gamma_j)\cdots \gamma_r).
\end{equation}

On the other hand, we have 
\begin{equation}\label{equ:Phik deltaomega}
\resizebox{.92\hsize}{!}{
$
  \delta_\omega \Phi_k(\gamma_1\cdots \gamma_r)
  =
  3k\, 
  \begin{tikzpicture}
    \node[int] (i) at (0,.5) {};
    \node (v1) at (-.5,-.2) {$\epsilon$};
    \node (v2) at (0,-.2) {$\omega$};
    \node (v3) at (.5,-.2) {$\omega$};
  \draw (i) edge (v1) edge (v2) edge (v3);
  \end{tikzpicture}
  \cdots 
  \begin{tikzpicture}
    \node[int] (i) at (0,.5) {};
    \node (v1) at (-.5,-.2) {$\omega$};
    \node (v2) at (0,-.2) {$\omega$};
    \node (v3) at (.5,-.2) {$\omega$};
  \draw (i) edge (v1) edge (v2) edge (v3);
  \end{tikzpicture}  
  \begin{tikzpicture}
    \node[ext] (i) at (0,.5) {$\gamma_1$};
    \node (v1) at (0,-.2) {$\omega$};
  \draw (i) edge (v1);
  \end{tikzpicture}
  \cdots 
  \begin{tikzpicture}
    \node[ext] (i) at (0,.5) {$\gamma_r$};
    \node (v1) at (0,-.2) {$\omega$};
  \draw (i) edge (v1);
  \end{tikzpicture}
  \, +\, 
  \sum_{j=1}^r\pm
  \begin{tikzpicture}
    \node[int] (i) at (0,.5) {};
    \node (v1) at (-.5,-.2) {$\omega$};
    \node (v2) at (0,-.2) {$\omega$};
    \node (v3) at (.5,-.2) {$\omega$};
  \draw (i) edge (v1) edge (v2) edge (v3);
  \end{tikzpicture}
  \cdots 
  \begin{tikzpicture}
    \node[int] (i) at (0,.5) {};
    \node (v1) at (-.5,-.2) {$\omega$};
    \node (v2) at (0,-.2) {$\omega$};
    \node (v3) at (.5,-.2) {$\omega$};
  \draw (i) edge (v1) edge (v2) edge (v3);
  \end{tikzpicture}  
  \begin{tikzpicture}
    \node[ext] (i) at (0,.5) {$\gamma_1$};
    \node (v1) at (0,-.2) {$\omega$};
  \draw (i) edge (v1);
  \end{tikzpicture}
  \cdots
  \begin{tikzpicture}
    \node[ext] (i) at (0,.5) {$\gamma_j$};
    \node (v1) at (0,-.2) {$\epsilon$};
  \draw (i) edge (v1);
  \end{tikzpicture}
  \cdots 
  \begin{tikzpicture}
    \node[ext] (i) at (0,.5) {$\gamma_r$};
    \node (v1) at (0,-.2) {$\omega$};
  \draw (i) edge (v1);
  \end{tikzpicture}\, .
  $}
\end{equation}

We also define the similar operation 

\begin{gather*}
  \hat \Phi_k \colon \Sym(\GC_0) \to \sX \\
  \hat \Phi_k(\gamma_1\cdots \gamma_r) = 
  \begin{tikzpicture}
      \node[int] (i) at (0,.5) {};
      \node (v1) at (-1,-.2) {$\omega$};
      \node (v2) at (-.5,-.2) {$\omega$};
      \node (v3) at (0,-.2) {$\omega$};
      \node (v4) at (.5,-.2) {$\omega$};
      \node (v5) at (1,-.2) {$\omega$};
    \draw (i) edge (v1) edge (v2) edge (v3) edge (v4) edge (v5);
    \end{tikzpicture}
    \underbrace{
    \begin{tikzpicture}
      \node[int] (i) at (0,.5) {};
      \node (v1) at (-.5,-.2) {$\omega$};
      \node (v2) at (0,-.2) {$\omega$};
      \node (v3) at (.5,-.2) {$\omega$};
    \draw (i) edge (v1) edge (v2) edge (v3);
    \end{tikzpicture} 
    \cdots 
    \begin{tikzpicture}
      \node[int] (i) at (0,.5) {};
      \node (v1) at (-.5,-.2) {$\omega$};
      \node (v2) at (0,-.2) {$\omega$};
      \node (v3) at (.5,-.2) {$\omega$};
    \draw (i) edge (v1) edge (v2) edge (v3);
    \end{tikzpicture}  
    }_{k}
    \begin{tikzpicture}
      \node[ext] (i) at (0,.5) {$\gamma_1$};
      \node (v1) at (0,-.2) {$\omega$};
    \draw (i) edge (v1);
    \end{tikzpicture}
    \cdots 
    \begin{tikzpicture}
      \node[ext] (i) at (0,.5) {$\gamma_r$};
      \node (v1) at (0,-.2) {$\omega$};
    \draw (i) edge (v1);
    \end{tikzpicture}
    ,
  \end{gather*}
It satisfies 
\begin{equation}\label{equ:hatPhik deltas}
  \begin{aligned}
  &\delta_s \hat \Phi_k(\gamma_1\cdots \gamma_r)
  - 
  \sum_{j=1}^r (-1)^{|\gamma_1|+\cdots+|\gamma_{j-1}|} \hat \Phi_k(\gamma_1\cdots (\delta\gamma_j)\cdots \gamma_r)
  \\&= 10
  \begin{tikzpicture}
    \node[int] (i1) at (-.3,.5) {};
    \node[int] (i2) at (.3,.5) {};
    \node (v1) at (-1,-.2) {$\omega$};
    \node (v2) at (-.5,-.2) {$\omega$};
    \node (v3) at (0,-.2) {$\omega$};
    \node (v4) at (.5,-.2) {$\omega$};
    \node (v5) at (1,-.2) {$\omega$};
  \draw (i1) edge (v1) edge (v2) edge (v3) edge (i2) (i2) edge (v4) edge (v5);
  \end{tikzpicture}
  \underbrace{
  \begin{tikzpicture}
    \node[int] (i) at (0,.5) {};
    \node (v1) at (-.5,-.2) {$\omega$};
    \node (v2) at (0,-.2) {$\omega$};
    \node (v3) at (.5,-.2) {$\omega$};
  \draw (i) edge (v1) edge (v2) edge (v3);
  \end{tikzpicture} 
  \cdots 
  \begin{tikzpicture}
    \node[int] (i) at (0,.5) {};
    \node (v1) at (-.5,-.2) {$\omega$};
    \node (v2) at (0,-.2) {$\omega$};
    \node (v3) at (.5,-.2) {$\omega$};
  \draw (i) edge (v1) edge (v2) edge (v3);
  \end{tikzpicture}  
  }_{k}
  \begin{tikzpicture}
    \node[ext] (i) at (0,.5) {$\gamma_1$};
    \node (v1) at (0,-.2) {$\omega$};
  \draw (i) edge (v1);
  \end{tikzpicture}
  \cdots 
  \begin{tikzpicture}
    \node[ext] (i) at (0,.5) {$\gamma_r$};
    \node (v1) at (0,-.2) {$\omega$};
  \draw (i) edge (v1);
  \end{tikzpicture}
  .
  \end{aligned}
\end{equation}

Next we define the degree zero operation 
\begin{gather*}
  \Psi \colon \sX \to \sX \\
  \Gamma \mapsto \sum 
  \begin{tikzpicture}
  \node[ext] (v) at (0,0) {$\Gamma$};
  \node (v1) at (-.7,0) {$\omega$};
  \node (v2) at (0.7,0) {$\omega$};
  \draw (v) edge (v1) edge (v2);
  \end{tikzpicture},
  \end{gather*}
  where we sum over all ways of attaching two $\omega$-legs to the graph $\Gamma$, in the blown-up picture. This means that the half-edges are attached to an internal vertex, or become an $\epsilon$-leg. (This can be seen as attaching to the special vertex.)
  
  \begin{lemma}\label{lem:Psidelta}
  If a graph $\Gamma\in X$ does not contain any $\epsilon$-legs, then we have that 
  \[
    \delta_s \Psi(\Gamma)-\Psi(\delta_s \Gamma) 
    = 
    -
    \sum 
    \begin{tikzpicture}
    \node[ext] (v) at (0,0) {$\Gamma$};
    \node[int] (w) at (0,.7) {};
    \node (v1) at (-.7,.7) {$\omega$};
    \node (v2) at (0.7,.7) {$\omega$};
    \draw (w) edge (v1) edge (v2) edge (v);
    \end{tikzpicture}
    =: -\Xi(\Gamma),
  \]
  where on the right-hand side we again sum over all ways of attaching the leg to $\Gamma$, with the attachment to the special vertex being the same as introducing an $\epsilon$-marking at the leg.
  \end{lemma}
  \begin{proof}
  The computation is similar to the proof of Lemma \ref{lem:some ops}. 
  \end{proof}

Furthermore, we use the pre-Lie product $\bullet$ on $\GC_0$. For $\gamma_1,\gamma_2\in \GC_0$
\[
  \gamma_1\bullet \gamma_2
  =
  \sum
  \begin{tikzpicture}
    \node[ext] (v) at (0,0.4) {$\gamma_1$};
    \node[ext] (w) at (0,-0.4) {$\gamma_2$};
    \draw (v) edge (w) (w) edge (v.west) edge (v.east);
    \end{tikzpicture}
    \in \GC_0,
\]
where the sum is over all ways of inserting $\gamma_2$ into a vertex of $\gamma_1$.
We shall only need to use the following property of $\bullet$, which is a special case of Proposition \ref{prop:pois trivial} (proved in  \cite{Willwachertrivial}):
\begin{equation} \label{eq:preLie eqn}
  \delta (\gamma_1\bullet \gamma_2)
  -
  (\delta \gamma_1)\bullet \gamma_2
  - (-1)^{|\gamma_1|}
   \gamma_1\bullet (\delta\gamma_2)
   =
   \{\gamma_1,\gamma_2\}.
\end{equation}

Finally, we define the operation 
\begin{gather*}
  \tilde \Phi_k \colon \Sym(\GC_0) \to \sX \\
  \tilde \Phi_k(\gamma_1\cdots \gamma_r) 
  =
  \sum_{j=1}^r
  \pm
  \begin{tikzpicture}
      \node[int] (i) at (0,.5) {};
      \node (v1) at (-.5,-.2) {$\omega$};
      \node (v2) at (0,-.2) {$\omega$};
      \node (v3) at (.5,-.2) {$\omega$};
    \draw (i) edge (v1) edge (v2) edge (v3);
    \end{tikzpicture}
    \cdots 
    \begin{tikzpicture}
      \node[int] (i) at (0,.5) {};
      \node (v1) at (-.5,-.2) {$\omega$};
      \node (v2) at (0,-.2) {$\omega$};
      \node (v3) at (.5,-.2) {$\omega$};
    \draw (i) edge (v1) edge (v2) edge (v3);
    \end{tikzpicture}  
    \begin{tikzpicture}
      \node[ext] (i) at (0,.5) {$\gamma_1$};
      \node (v1) at (0,-.2) {$\omega$};
    \draw (i) edge (v1);
    \end{tikzpicture}
    \cdots 
    \begin{tikzpicture}
      \node[ext] (i) at (0,.5) {$\gamma_j$};
      \node (v1) at (-.5,-.2) {$\omega$};
      \node (v2) at (0,-.2) {$\omega$};
      \node (v3) at (.5,-.2) {$\omega$};
    \draw (i) edge (v1) edge (v2) edge (v3);
    \end{tikzpicture}
    \cdots 
    \begin{tikzpicture}
      \node[ext] (i) at (0,.5) {$\gamma_r$};
      \node (v1) at (0,-.2) {$\omega$};
    \draw (i) edge (v1);
    \end{tikzpicture}\, .
  \end{gather*}
In other words, $\tilde \Phi_k$ is defined similarly to $\Phi_k$, except that we attach 3 $\omega$-legs to one of the $\gamma_j$ instead of one.
\begin{lemma}
The expression  $\delta_s \tilde\Phi_k(\gamma_1\cdots \gamma_r)
  -
  \sum_{j=1}^r (-1)^{|\gamma_1|+\cdots+|\gamma_{j-1}|} \tilde\Phi_k(\gamma_1\cdots (\delta\gamma_j)\cdots \gamma_r)$ equals:
\begin{equation}\label{equ:tPhik deltas}
 \resizebox{.92\hsize}{!}{$
  \begin{aligned}
  \sum_{j=1}^r
  \pm
  \begin{tikzpicture}
      \node[int] (i) at (0,.5) {};
      \node (v1) at (-.5,-.2) {$\omega$};
      \node (v2) at (0,-.2) {$\omega$};
      \node (v3) at (.5,-.2) {$\omega$};
    \draw (i) edge (v1) edge (v2) edge (v3);
    \end{tikzpicture}
    \cdots 
    \begin{tikzpicture}
      \node[int] (i) at (0,.5) {};
      \node (v1) at (-.5,-.2) {$\omega$};
      \node (v2) at (0,-.2) {$\omega$};
      \node (v3) at (.5,-.2) {$\omega$};
    \draw (i) edge (v1) edge (v2) edge (v3);
    \end{tikzpicture}  
    \begin{tikzpicture}
      \node[ext] (i) at (0,.5) {$\gamma_1$};
      \node (v1) at (0,-.2) {$\omega$};
    \draw (i) edge (v1);
    \end{tikzpicture}
    \cdots 
    \begin{tikzpicture}
      \node[ext] (e) at (0,1.2) {$\gamma_j$};
      \node[int] (i) at (0,.5) {};
      \node (v1) at (-.5,-.2) {$\omega$};
      \node (v2) at (0,-.2) {$\omega$};
      \node (v3) at (.5,-.2) {$\omega$};
    \draw (i) edge (v1) edge (v2) edge (v3) edge (e);
    \end{tikzpicture}
    \cdots 
    \begin{tikzpicture}
      \node[ext] (i) at (0,.5) {$\gamma_r$};
      \node (v1) at (0,-.2) {$\omega$};
    \draw (i) edge (v1);
    \end{tikzpicture}
    +
    3
    \begin{tikzpicture}
      \node[int] (i) at (0,.5) {};
      \node (v1) at (-.5,-.2) {$\omega$};
      \node (v2) at (0,-.2) {$\omega$};
      \node (v3) at (.5,-.2) {$\omega$};
    \draw (i) edge (v1) edge (v2) edge (v3);
    \end{tikzpicture}
    \cdots 
    \begin{tikzpicture}
      \node[int] (i) at (0,.5) {};
      \node (v1) at (-.5,-.2) {$\omega$};
      \node (v2) at (0,-.2) {$\omega$};
      \node (v3) at (.5,-.2) {$\omega$};
    \draw (i) edge (v1) edge (v2) edge (v3);
    \end{tikzpicture}  
    \begin{tikzpicture}
      \node[ext] (i) at (0,.5) {$\gamma_1$};
      \node (v1) at (0,-.2) {$\omega$};
    \draw (i) edge (v1);
    \end{tikzpicture}
    \cdots 
    \begin{tikzpicture}
      \node[ext] (e) at (0,1.2) {$\gamma_j$};
      \node[int] (i) at (0,.5) {};
      \node (v1) at (-.5,-.2) {$\omega$};
      \node (v2) at (0,-.2) {$\omega$};
      \node (v3) at (.5,-.2) {$\omega$};
    \draw (i) edge (v1) edge (v2) edge (e) (e) edge (v3);
    \end{tikzpicture}
    \cdots 
    \begin{tikzpicture}
      \node[ext] (i) at (0,.5) {$\gamma_r$};
      \node (v1) at (0,-.2) {$\omega$};
    \draw (i) edge (v1);
    \end{tikzpicture}
    \, .
  \end{aligned}
  $}
\end{equation}
\end{lemma}
\begin{proof}
The verification is a similar graphical computation to those above. We omit the details. 
\end{proof}

\subsection{Definition of the map $E$}
We then define the maps of graded vector spaces
\begin{gather*}
    E_k \colon \Sym(\GC_0) \to \sX \\
    \gamma_1\cdots \gamma_r \mapsto 
    \Phi_k(\gamma_1\cdots \gamma_r )
    +
    k \tilde \Phi_{k-1}(\gamma_1\cdots \gamma_r )
    +
    \frac{3k}{10} \hat \Phi_{k-1}(\gamma_1\cdots \gamma_r )
    +
    3k \Psi( \Phi_{k-1}(\gamma_1\cdots \gamma_r ))
    \\-
    \sum_{i=1}^r \gamma_i^1 \circ \Phi_k(\gamma_1\cdots \hat \gamma_i \cdots \gamma_r )
    + 
    \sum_{i<j}\pm
    \Phi_k(\gamma_1 \cdots (\gamma_i \bullet \gamma_j) \cdots \gamma_r)
    .
\end{gather*}
The map $E_k$ is not a morphism of dg vector spaces; it does not commute with the differentials.
However, we have the following result:
\begin{lemma}\label{lem:Ek differential}
  For $\gamma_1, \dots, \gamma_r\in \GC_0,$ the commutator of $E_k$ with the differential
  \[
    \delta E_k(\gamma_1\cdots \gamma_r)
    -\sum_j\pm
    E_k(\gamma_1\cdots \delta\gamma_j\cdots \gamma_r)
  \]
  is a linear combination of graphs with at most $r+3k-2$ legs decorated by $\omega$.
\end{lemma}
\begin{proof}
First note that $\Phi_k(\gamma_1\cdots \gamma_r)$ is the only term in the definition of $E_k$ that has $r+3k$ $\omega$-legs. The remaining terms, call them $X(\gamma_1\cdots \gamma_r)$ temporarily, all have $r+3k-1$ many $\omega$-legs.
Given \eqref{equ:Phik deltas} the assertion of the lemma hence is equivalent to the statement that 
\begin{equation}\label{equ:Ek differential tbs}
  [\delta_s,X](\gamma_1\cdots \gamma_r)
  :=
\delta_s X(\gamma_1\cdots \gamma_r)
-
\sum_j\pm
    X(\gamma_1\cdots \delta\gamma_j\cdots \gamma_r)
  = -\delta_\omega \Phi_k(\gamma_1\cdots \gamma_r).
\end{equation}
To show this we investigate the terms contributing to $[\delta_s,X]$ separately.
First, by Lemma~\ref{lem:Psidelta}
\[
  [\delta_s, \Psi\circ \Phi_{k-1}](\gamma_1\cdots \gamma_r)
  =
  -\Xi(\Phi_{k-1}(\gamma_1\cdots \gamma_r)).
\]
The terms contributing to $\Xi(\Phi_{k-1}(\gamma_1\cdots \gamma_r))$ are of three sorts: ($\Xi_1$) the terms for which there is a new $\epsilon$-leg and ($\Xi_2$) the terms for which the new leg is connected to a vertex of one $\gamma_j$ and ($\Xi_3$) the terms for which the new leg is connected to a vertex of another tripod.
\begin{align*}
  \Xi_1&= 
  \begin{tikzpicture}
    \node[int] (i) at (0,.5) {};
    \node (v1) at (-.5,-.2) {$\epsilon$};
    \node (v2) at (0,-.2) {$\omega$};
    \node (v3) at (.5,-.2) {$\omega$};
  \draw (i) edge (v1) edge (v2) edge (v3);
  \end{tikzpicture}
  \begin{tikzpicture}
    \node[int] (i) at (0,.5) {};
    \node (v1) at (-.5,-.2) {$\omega$};
    \node (v2) at (0,-.2) {$\omega$};
    \node (v3) at (.5,-.2) {$\omega$};
  \draw (i) edge (v1) edge (v2) edge (v3);
  \end{tikzpicture}
  \cdots 
  \begin{tikzpicture}
    \node[int] (i) at (0,.5) {};
    \node (v1) at (-.5,-.2) {$\omega$};
    \node (v2) at (0,-.2) {$\omega$};
    \node (v3) at (.5,-.2) {$\omega$};
  \draw (i) edge (v1) edge (v2) edge (v3);
  \end{tikzpicture}  
  \begin{tikzpicture}
    \node[ext] (i) at (0,.5) {$\gamma_1$};
    \node (v1) at (0,-.2) {$\omega$};
  \draw (i) edge (v1);
  \end{tikzpicture}
  \cdots 
  \begin{tikzpicture}
    \node[ext] (i) at (0,.5) {$\gamma_r$};
    \node (v1) at (0,-.2) {$\omega$};
  \draw (i) edge (v1);
  \end{tikzpicture}
  \\
  \Xi_2&=
  \sum
  \begin{tikzpicture}
    \node[int] (i) at (0,.5) {};
    \node (v1) at (-.5,-.2) {$\omega$};
    \node (v2) at (0,-.2) {$\omega$};
    \node (v3) at (.5,-.2) {$\omega$};
  \draw (i) edge (v1) edge (v2) edge (v3);
  \end{tikzpicture}
  \cdots 
  \begin{tikzpicture}
    \node[int] (i) at (0,.5) {};
    \node (v1) at (-.5,-.2) {$\omega$};
    \node (v2) at (0,-.2) {$\omega$};
    \node (v3) at (.5,-.2) {$\omega$};
  \draw (i) edge (v1) edge (v2) edge (v3);
  \end{tikzpicture}  
  \begin{tikzpicture}
    \node[ext] (i) at (0,.5) {$\gamma_1$};
    \node (v1) at (0,-.2) {$\omega$};
  \draw (i) edge (v1);
  \end{tikzpicture}
  \cdots 
  \begin{tikzpicture}
    \node[ext] (i) at (0,.5) {$\gamma_j$};
    \node (v1) at (0,-.2) {$\omega$};
  \draw (i) edge (v1);
  \end{tikzpicture}
  \cdots
  \begin{tikzpicture}
    \node[int] (w) at (0,1) {};
    \node (w1) at (-.7,1) {$\omega$};
    \node (w2) at (0.7,1) {$\omega$};
    \node[ext] (i) at (0,.5) {$\gamma_r$};
    \node (v1) at (0,-.2) {$\omega$};
  \draw (i) edge (v1) edge (w)
  (w) edge (w1) edge (w2);
  \end{tikzpicture}
\\
\Xi_3&=(k-1)
\begin{tikzpicture}
  \node[int] (i1) at (-.3,.5) {};
  \node[int] (i2) at (.3,.5) {};
  \node (v1) at (-1,-.2) {$\omega$};
  \node (v2) at (-.5,-.2) {$\omega$};
  \node (v3) at (0,-.2) {$\omega$};
  \node (v4) at (.5,-.2) {$\omega$};
  \node (v5) at (1,-.2) {$\omega$};
\draw (i1) edge (v1) edge (v2) edge (v3) edge (i2) (i2) edge (v4) edge (v5);
\end{tikzpicture}
\begin{tikzpicture}
  \node[int] (i) at (0,.5) {};
  \node (v1) at (-.5,-.2) {$\omega$};
  \node (v2) at (0,-.2) {$\omega$};
  \node (v3) at (.5,-.2) {$\omega$};
\draw (i) edge (v1) edge (v2) edge (v3);
\end{tikzpicture}
\cdots 
\begin{tikzpicture}
  \node[int] (i) at (0,.5) {};
  \node (v1) at (-.5,-.2) {$\omega$};
  \node (v2) at (0,-.2) {$\omega$};
  \node (v3) at (.5,-.2) {$\omega$};
\draw (i) edge (v1) edge (v2) edge (v3);
\end{tikzpicture}  
\begin{tikzpicture}
  \node[ext] (i) at (0,.5) {$\gamma_1$};
  \node (v1) at (0,-.2) {$\omega$};
\draw (i) edge (v1);
\end{tikzpicture}
\cdots 
\begin{tikzpicture}
  \node[ext] (i) at (0,.5) {$\gamma_r$};
  \node (v1) at (0,-.2) {$\omega$};
\draw (i) edge (v1);
\end{tikzpicture}
\end{align*} 

The terms $\Xi_1$ cancel the left-hand terms of $\delta\Phi_k(\cdots)$ in \eqref{equ:Phik deltaomega}, in which one $\epsilon$ is put on a tripod leg.
The terms $\Xi_3$ cancel with the terms $[\delta_s, \hat \Phi_{k-1}](\gamma_1\cdots \gamma_r)$ by \eqref{equ:hatPhik deltas}.

Next denote temporarily 
\[
  Y(\gamma_1\cdots\gamma_r):=
  \sum_{i=1}^r \gamma_i^1 \circ \Phi_k(\gamma_1\cdots \hat \gamma_i \cdots \gamma_r ).
\]
Then we use \eqref{equ:deltagamma1} to compute that 
\[
  [\delta_s,Y](\gamma_1\cdots\gamma_r)
  = 
  \sum_{i=1}^r \gamma_i^\circ \circ \Phi_k(\gamma_1\cdots \hat \gamma_i \cdots \gamma_r ).
\]
The terms on the right-hand side may be again split into terms ($Y_1$) in which $\gamma_i$ is attached to an $\epsilon$-leg and terms ($Y_2$) for which $\gamma_i$ is attached to a vertex of some other $\gamma_j$ and ($Y_3$) terms for which $\gamma_i$ is attached to a vertex of some tripod.
\begin{align*}
  Y_1&=\sum_i 
\begin{tikzpicture}
  \node[int] (i) at (0,.5) {};
  \node (v1) at (-.5,-.2) {$\omega$};
  \node (v2) at (0,-.2) {$\omega$};
  \node (v3) at (.5,-.2) {$\omega$};
\draw (i) edge (v1) edge (v2) edge (v3);
\end{tikzpicture}
\cdots 
\begin{tikzpicture}
  \node[int] (i) at (0,.5) {};
  \node (v1) at (-.5,-.2) {$\omega$};
  \node (v2) at (0,-.2) {$\omega$};
  \node (v3) at (.5,-.2) {$\omega$};
\draw (i) edge (v1) edge (v2) edge (v3);
\end{tikzpicture}  
\begin{tikzpicture}
  \node[ext] (i) at (0,.5) {$\gamma_1$};
  \node (v1) at (0,-.2) {$\omega$};
\draw (i) edge (v1);
\end{tikzpicture}
\cdots
\begin{tikzpicture}
  \node[ext] (i) at (0,.5) {$\gamma_i$};
  \node (v1) at (0,-.2) {$\epsilon$};
\draw (i) edge (v1);
\end{tikzpicture}
\cdots
\begin{tikzpicture}
  \node[ext] (i) at (0,.5) {$\gamma_r$};
  \node (v1) at (0,-.2) {$\omega$};
\draw (i) edge (v1);
\end{tikzpicture}
\\
Y_2&=\sum_{i,j} 
\begin{tikzpicture}
  \node[int] (i) at (0,.5) {};
  \node (v1) at (-.5,-.2) {$\omega$};
  \node (v2) at (0,-.2) {$\omega$};
  \node (v3) at (.5,-.2) {$\omega$};
\draw (i) edge (v1) edge (v2) edge (v3);
\end{tikzpicture}
\cdots 
\begin{tikzpicture}
  \node[int] (i) at (0,.5) {};
  \node (v1) at (-.5,-.2) {$\omega$};
  \node (v2) at (0,-.2) {$\omega$};
  \node (v3) at (.5,-.2) {$\omega$};
\draw (i) edge (v1) edge (v2) edge (v3);
\end{tikzpicture}  
\begin{tikzpicture}
  \node[ext] (i) at (0,.5) {$\gamma_1$};
  \node (v1) at (0,-.2) {$\omega$};
\draw (i) edge (v1);
\end{tikzpicture}
\cdots
\begin{tikzpicture}
  \node[ext] (i) at (0,1.3) {$\gamma_i$};
  \node[ext] (j) at (0,.5) {$\gamma_j$};
  \node (v1) at (0,-.2) {$\omega$};
\draw (j) edge (v1) edge (i);
\end{tikzpicture}
\cdots
\begin{tikzpicture}
  \node[ext] (i) at (0,.5) {$\gamma_r$};
  \node (v1) at (0,-.2) {$\omega$};
\draw (i) edge (v1);
\end{tikzpicture}
\\
Y_3&=\sum_{i} 
\begin{tikzpicture}
  \node[ext] (ii) at (0,1) {$\gamma_i$};
  \node[int] (i) at (0,.5) {};
  \node (v1) at (-.5,-.2) {$\omega$};
  \node (v2) at (0,-.2) {$\omega$};
  \node (v3) at (.5,-.2) {$\omega$};
\draw (i) edge (v1) edge (v2) edge (v3) edge (ii);
\end{tikzpicture}
\begin{tikzpicture}
  \node[int] (i) at (0,.5) {};
  \node (v1) at (-.5,-.2) {$\omega$};
  \node (v2) at (0,-.2) {$\omega$};
  \node (v3) at (.5,-.2) {$\omega$};
\draw (i) edge (v1) edge (v2) edge (v3);
\end{tikzpicture} 
\cdots 
\begin{tikzpicture}
  \node[int] (i) at (0,.5) {};
  \node (v1) at (-.5,-.2) {$\omega$};
  \node (v2) at (0,-.2) {$\omega$};
  \node (v3) at (.5,-.2) {$\omega$};
\draw (i) edge (v1) edge (v2) edge (v3);
\end{tikzpicture}  
\begin{tikzpicture}
  \node[ext] (i) at (0,.5) {$\gamma_1$};
  \node (v1) at (0,-.2) {$\omega$};
\draw (i) edge (v1);
\end{tikzpicture}
\cdots
\begin{tikzpicture}
  \node[ext] (i) at (0,.5) {$\gamma_r$};
  \node (v1) at (0,-.2) {$\omega$};
\draw (i) edge (v1);
\end{tikzpicture}
\end{align*} 
The terms $Y_1$ cancel the remaining terms of $\delta\Phi_k(\cdots)$, see \eqref{equ:Phik deltaomega}, in which one $\epsilon$-leg is attached to $\gamma_i$.
The terms $Y_2$ cancel the terms arising from the commutator  of $\sum_{i<j}
\Phi_k(\gamma_1 \cdots (\gamma_i \bullet \gamma_j) \cdots \gamma_r)$ with the differential by \eqref{eq:preLie eqn}.

The commutator $Z:=[\delta_s,\tilde \Phi_k](\gamma_1\cdots\gamma_r)$ is computed in \eqref{equ:tPhik deltas} and we denote by $Z_1$ the first summand and by $Z_2$ the second summand on the right-hand side of \eqref{equ:tPhik deltas}.
Then the second summand $Z_2$ cancels the terms $\Xi_2$ above. At the same time the terms $Z_1$ cancel the terms $Y_3$ above.
Thus all terms have been taken care of and \eqref{equ:Ek differential tbs} and the lemma is proved.
\end{proof}

As an immediate consequence we can finally define the map $E$ of Theorem \ref{thm:emb 2}.

\begin{corollary}\label{cor:E definition}
The map 
\[
E \colon 
\Sym^{9}(\GC_0)^{(g-3)}[-22]
\oplus
\Sym^{6}(\GC_0)^{(g-5)}[-22]
\oplus
\Sym^{3}(\GC_0)^{(g-7)}[-22]
\to
B_g
\]
defined such that 
\[
(\gamma_1\cdots \gamma_9) \oplus (\mu_1\cdots \mu_6) \oplus (\nu_1\nu_2\nu_3)
\mapsto 
\pi \left(  E_1(\gamma_1\cdots \gamma_9) + E_2(\mu_1\cdots \mu_6)  +  E_3(\nu_1\nu_2\nu_3) \right) \in B_g\,,
\]
with $\pi \colon X_g[-22]\to B_g$ the projection, is a map of dg vector spaces.
\end{corollary}
\begin{proof}
By Lemma \ref{lem:Ek differential} the commutator of $E_k$ with the differential only has terms with $\leq 10$ $\omega$-legs, and these are killed by $\pi$.
\end{proof}

\subsection{Proof of Theorem \ref{thm:emb 2}}
It remains to check that the cohomology map of the map $E$ of Theorem \ref{thm:emb 2} (see the definition in Corollary \ref{cor:E definition}) is in fact injective.
Similarly to the proof of Theorem \ref{thm:emb1} above it suffices to check that the composition $\pi\circ\delta_\omega\circ E$ of $E$ with the map $\pi \colon H(C_g, \delta) \to H(C_{g.10\omega},\delta_s)$ and the quasi-isomorphism $\delta_\omega \colon B_g\to C_g$ is injective.
However, in contrast with the proof of Theorem \ref{thm:emb1}, the composition $\pi\circ \delta_\omega\circ E$ does not factorize through the cohomology isomorphism $K$ of Proposition \ref{prop:map K}.
Hence we need to trace through the proof of Proposition \ref{prop:map K}, in which the cohomology $H(C_{g,10\omega},\delta_s)$ is computed, and identify the subspace of the cohomology that is in the image of $\pi\circ\delta_\omega\circ E$.
As in that proof, we hence consider the filtration by the number of connected components in graphs, and the corresponding spectral sequence.
We need to trace our images $\pi\circ\delta_\omega\circ E(\gamma_1\cdots \gamma_r)$ through this spectral sequence.
Hence we consider the leading order term of $\pi\circ\delta_\omega\circ E(\gamma_1\cdots \gamma_r)$, i.\ e., the term with the most connected components.
This is easily seen to be 
\[
  \begin{tikzpicture}
    \node (v) at (0,0) {$\omega$};
    \node (w) at (1,0) {$\epsilon$};
    \draw (v) edge (w);
  \end{tikzpicture}
  \begin{tikzpicture}
    \node (v) at (0,0) {$\epsilon$};
    \node (w) at (1,0) {$\epsilon$};
    \draw (v) edge (w);
  \end{tikzpicture}
  \begin{tikzpicture}
    \node[int] (i) at (0,.5) {};
    \node (v1) at (-.5,-.2) {$\omega$};
    \node (v2) at (0,-.2) {$\omega$};
    \node (v3) at (.5,-.2) {$\omega$};
  \draw (i) edge (v1) edge (v2) edge (v3);
  \end{tikzpicture}
  \cdots 
  \begin{tikzpicture}
    \node[int] (i) at (0,.5) {};
    \node (v1) at (-.5,-.2) {$\omega$};
    \node (v2) at (0,-.2) {$\omega$};
    \node (v3) at (.5,-.2) {$\omega$};
  \draw (i) edge (v1) edge (v2) edge (v3);
  \end{tikzpicture}  
  \begin{tikzpicture}
    \node[ext] (i) at (0,.5) {$\gamma_1$};
    \node (v1) at (0,-.2) {$\omega$};
  \draw (i) edge (v1);
  \end{tikzpicture}
  \cdots 
  \begin{tikzpicture}
    \node[ext] (i) at (0,.5) {$\gamma_r$};
    \node (v1) at (0,-.2) {$\omega$};
  \draw (i) edge (v1);
  \end{tikzpicture}
\]
But this leading order term is the same as produced by the map $K$ of Proposition \ref{prop:map K}, acting on the summand $\Q \alpha\beta \otimes \fHGC$ on the left-hand side of \eqref{equ:K def}, and specifically on the element 
\[
  \iota(\gamma_1\cdots\gamma_r):=
  \begin{tikzpicture}
    \node[int] (i) at (0,.5) {};
    \node (v1) at (-.5,-.2) {$\omega$};
    \node (v2) at (0,-.2) {$\omega$};
    \node (v3) at (.5,-.2) {$\omega$};
  \draw (i) edge (v1) edge (v2) edge (v3);
  \end{tikzpicture}
  \cdots 
  \begin{tikzpicture}
    \node[int] (i) at (0,.5) {};
    \node (v1) at (-.5,-.2) {$\omega$};
    \node (v2) at (0,-.2) {$\omega$};
    \node (v3) at (.5,-.2) {$\omega$};
  \draw (i) edge (v1) edge (v2) edge (v3);
  \end{tikzpicture}  
  \begin{tikzpicture}
    \node[ext] (i) at (0,.5) {$\gamma_1$};
    \node (v1) at (0,-.2) {$\omega$};
  \draw (i) edge (v1);
  \end{tikzpicture}
  \cdots 
  \begin{tikzpicture}
    \node[ext] (i) at (0,.5) {$\gamma_r$};
    \node (v1) at (0,-.2) {$\omega$};
  \draw (i) edge (v1);
  \end{tikzpicture}
  \in \fHGC.
\]
But since the map $\iota$ is an injection on cohomology, and the map $K$ is an injection on cohomology, so must be $\pi\circ\delta_\omega\circ E$. Hence also $E$ is an injection on cohomology as claimed.
\hfill \qed

\subsection{Proof of Theorem \ref{thm:emb intro}}
To show Theorem \ref{thm:emb intro} of the introduction we need to check that the images of the maps $E$ and $F$ of Theorems \ref{thm:emb 2} and Theorem \ref{thm:emb1} inside the cohomology of $(C_g,\delta_s+\delta_\omega)$ are linearly independent. As in the preceding proof, it suffices to check that the images of $\pi\circ E$ and $\pi\circ F$ are linearly independent in $H(C_{g,10\omega},\delta_s)$. But the latter cohomology is computed in Propsition \ref{prop:map K} and identified with the genus $g$ part of
\begin{multline*}
  \resizebox{.96\hsize}{!}{
  $
  H(\fHGC)_{10\text{-hair}}
  \oplus 
  \Q\alpha\otimes H(\fHGC)_{10\text{-hair}} 
  \Q\beta\otimes H(\fHGC)_{9\text{-hair}} 
  \oplus \Q\alpha\beta \otimes H(\fHGC)_{9\text{-hair}}. 
  $
  }
\end{multline*}
Under this identification we saw in the proof of Theorem \ref{thm:emb1} that the image of the map $\pi\circ F$ is a subspace of the first two summands
\[
H(\fHGC)_{10\text{-hair}}
  \oplus 
  \Q\alpha\otimes H(\fHGC)_{10\text{-hair}}.
\]
Likewise, we saw in the proof of Theorem \ref{thm:emb 2} that the composition of $\pi\circ E$ with a projection to the summand $\Q\alpha\beta \otimes H(\fHGC)_{9\text{-hair}}$ is an injection.
Hence, since the two previous subspaces are linearly independent by Proposition \ref{prop:map K}, so must be the images of $\pi\circ E$  and $\pi\circ F$. 
\hfill \qed

\section{Euler characteristic}\label{sec:euler}

The $\bS_n$-equivariant Euler characteristic of a complex very similar to $X_{g,n}$ was computed in \cite{PayneWillwacherEuler}; the difference here is that the $\omega$-legs are odd instead of even. We recall the results from \cite{PayneWillwacherEuler} and then apply the necessary modifications to account for this degree shift.

First, we introduce the functions 
\begin{align*}
    E_\ell&:= \frac 1 \ell \sum_{d\mid \ell}\mu(\ell/d)\frac 1 {u^d},
    &
    \lambda_\ell &:= u^\ell (1-u^\ell)\ell,
\end{align*}
\begin{equation*} 
    B(z) := \sum_{r\geq 2}\frac{B_r}{r(r-1)} \frac 1 {z^{r-1}},
  \end{equation*}
and $U_\ell(X,u)$ such that
\begin{align*}
    \log U_\ell(X,u)
   &=   
    \log \frac {(-\lambda_\ell)^X \Gamma(-E_\ell+X) }{\Gamma(-E_\ell)}
    \\&=
        X\left(\log(\lambda_\ell E_\ell)-1 \right)+(-E_\ell+X-\textstyle{\frac 1 2} )\log(1-\textstyle{\frac X{E_\ell}}) + B(-E_\ell+X)- B(-E_\ell).
\end{align*}

Denote by $\widetilde X^{ev}_{g,n}$ a graded vector space defined in the same manner as $X_{g,n}$ but with even $\omega$-decorations instead of odd, and with all edges (structural or not) odd.
Let $\widetilde X^{ev}_{g,n,r\omega}\subset \widetilde X^{ev}_{g,n}$ be the subspace with $r$ legs decorated by $\omega$.
By \cite[\S4.3]{PayneWillwacherEuler}, the generating function for the equivariant Euler characteristic is\footnote{We subtract 1 relative to loc. cit. since we do not include the empty graph in our complex.}
\[
  \sum_{g,n,r}
  u^{g+n-1} w^r
  \chi_{\bbS_n}(\tilde X^{ev}_{g,n,r\omega} )
  =
  \prod_\ell 
  \frac { 
    U_\ell(\frac 1 \ell \sum_{d\mid \ell} \mu(\ell/d) (p_d  +1+w^d ), u )
  }
  { 
    U_\ell(\frac 1 \ell \sum_{d\mid \ell} \mu(\ell/d) p_d, u )
  }
  -1 \, ,
\]
with $w$ the formal variable taking care of the number of $\omega$-legs and $u$ the formal variable counting genus plus the number of punctures.
Looking at the derivation in loc. cit. one sees that the only change required from even to odd $\omega$ decorations is the sign in front of the term $w^d$. We obtain 
\[
  \sum_{g,n,r}
  u^{g+n-1} w^r
  \chi_{\bbS_n}(\tilde X_{g,n,r\omega} )
  =
  \prod_\ell 
  \frac { 
    U_\ell(\frac 1 \ell \sum_{d\mid \ell} \mu(\ell/d) (p_d  +1-w^d ), u )
  }
  { 
    U_\ell(\frac 1 \ell \sum_{d\mid \ell} \mu(\ell/d) p_d, u )
  }\, -1,
\]
with $\tilde X_{g,n}$ being a slightly modified version of $X_{g,n}$ in which all edges, are considered odd, not just structural ones.
This mis-treatment of non-structural edges may be undone by replacing $p_d\to -p_d$, which is the equivalent on the character of multiplying the underlying representation of the symmetric group by a degree shifted sign representation.
We hence obtain:
\[
  \sum_{g,n,r}
  u^{g+n-1} w^r
  \chi_{\bbS_n}(X_{g,n,r\omega} )
  =
  \prod_\ell 
  \frac { 
    U_\ell(\frac 1 \ell \sum_{d\mid \ell} \mu(\ell/d) (-p_d  +1-w^d ), u )
  }
  { 
    U_\ell(\frac 1 \ell \sum_{d\mid \ell} \mu(\ell/d) (-p_d), u )
  }\, -1.
\]

We are interested in the truncation of the complex $X_{g,n}$, concretely in the subcomplex 
$$C_{g,n}=\bigoplus_{r=0}^{10}X_{g,n,r\omega}[-21]$$
spanned by graphs with at most ten $\omega$-legs.
Let 
\[
T_{\leq 10} \left( \sum_{j=0}^\infty a_j  w^j \right)
=
\sum_{j=0}^{10} a_j
\]
be the operator that sums the first 10 coefficients of a formal power series in $w$.
Thus we find:

\[
  \sum_{g,n}
  u^{g+n}
  \chi_{\bbS_n}(C_{g,n} )
  =
  -u \ T_{\leq 10}\left(
  \prod_\ell 
  \frac { 
    U_\ell(\frac 1 \ell \sum_{d\mid \ell} \mu(\ell/d) (-p_d  +1-w^d ), u )
  }
  { 
    U_\ell(\frac 1 \ell \sum_{d\mid \ell} \mu(\ell/d) (-p_d), u )
  }-1\right)\, ,
\]
with the factor $-u$ accounting for the overall degree and genus shift due to the special vertex.

\begin{thm}
    The equivariant Euler characteristic of the weight 11 compactly supported cohomology of the moduli space of curves is computed by the following generating function:
\begin{multline*}
\frac12 \sum_{g,n\geq 0 \atop 2g+n\geq 3} u^{g+n} \chi_{\bbS_n}(\gr_{11}H_c(\M_{g,n}))
  =
  \sum_{g,n}
  u^{g+n}
  \chi_{\bbS_n}(C_{g,n} )
  \\=
  -u\  T_{\leq 10}\bigg(
  \prod_{\ell\geq 1} 
  \frac { 
    U_\ell(\frac 1 \ell \sum_{d\mid \ell} \mu(\ell/d) (-p_d  +1-w^d ), u )
  }
  { 
    U_\ell(\frac 1 \ell \sum_{d\mid \ell} \mu(\ell/d) (-p_d), u )
  }-1\bigg)\, .
\end{multline*}
\end{thm}

Specializing to $n=0$, we obtain: 
\begin{align*}
&\frac12\sum_{g\geq 2} u^g \chi(\gr_{11}H_c(\M_{g}))
=   
u^9-2 u^{10}+2 u^{11}+8 u^{13}-17 u^{14}-14 u^{15}-20 u^{16}+29 u^{17}+85 u^{18}
\\ &\quad
+178 u^{19}+123 u^{20}-311 u^{21}-1049 u^{22}-2443 u^{23}-776 u^{24}+6027 u^{25}+7200 u^{26}-34892
   u^{27}
   \\ &\quad  +196735 u^{28}+1215236 u^{29}-3230856 u^{30}-26415680 u^{31}
    +O(u^{32})  
\end{align*}
\noindent The following graphs show $\log(\frac12|\chi(\gr_{11}H_c(\M_{g}))|)$ and $\sgn(\chi(\gr_{11}H_c(\M_{g})))$ for $g$ up to 70.
\begin{center}
\includegraphics[width=.45\textwidth]{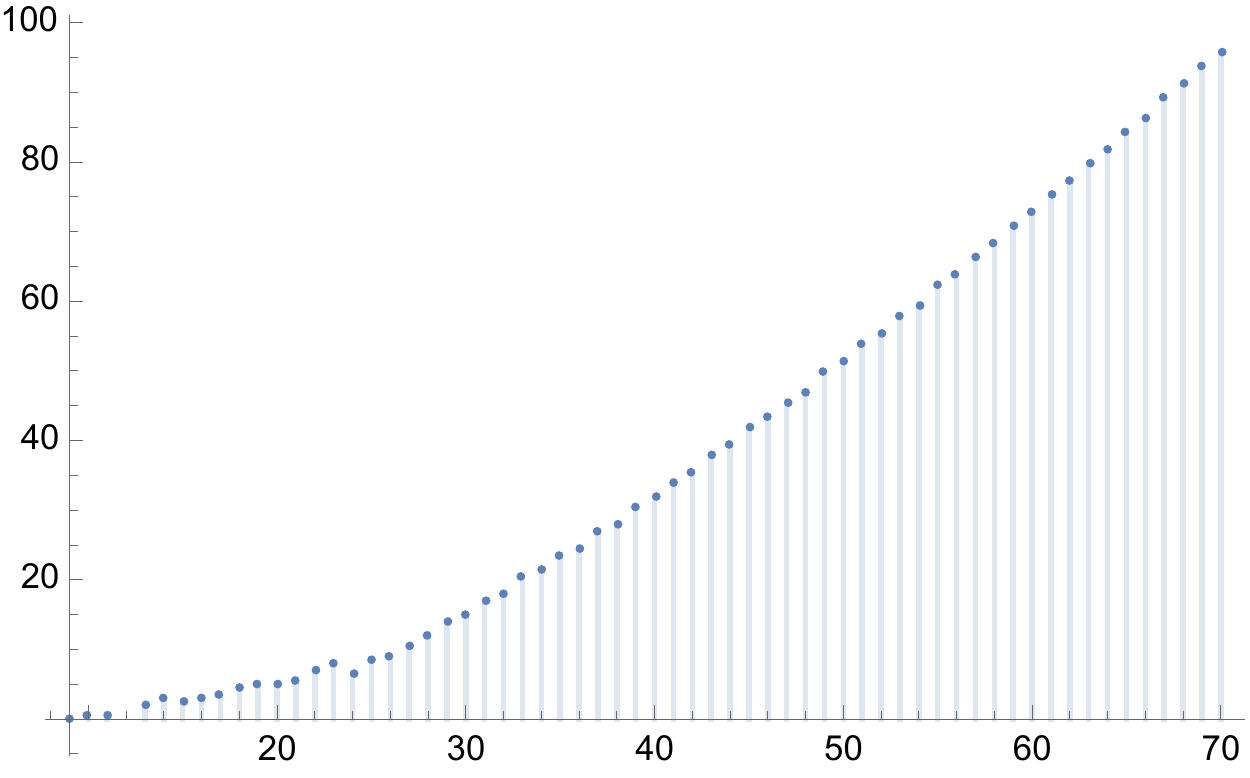} \quad \quad
\includegraphics[width=.45\textwidth]{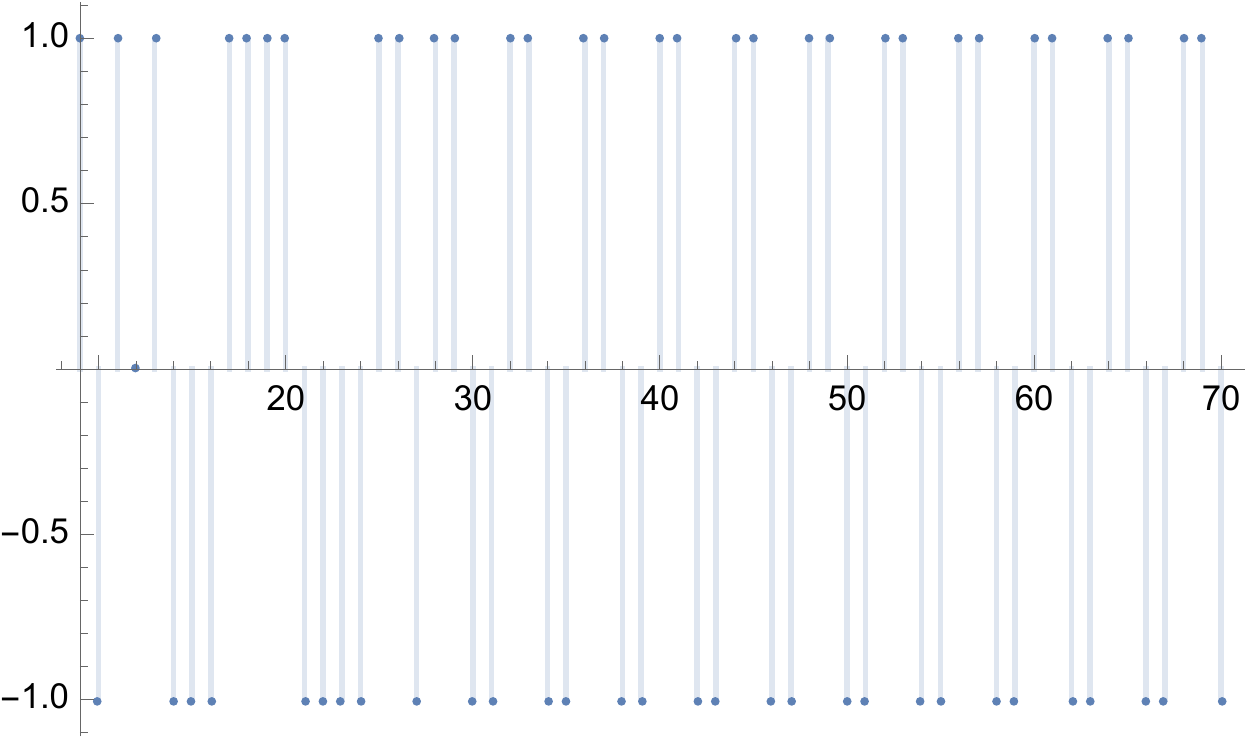}
\end{center}

We also obtain interesting numerical data for $n > 0$.  Recall that $\gr_{11}H^\bullet_c(\M_{g,n})$ vanishes for $g \leq 4$ and $n \leq 6$, by Corollary~\ref{cor:Eneg}.  In Figure~\ref{fig:EC table}, we present the $\bS_n$-equivariant Euler characteristic for $5 \leq g \leq 16$ and $n \leq 6$, expressed in the Schur polynomial basis for symmetric functions.  As mentioned in the introduction, our computations agree with those of Bergstr\"om and Faber for $g = 2$ and $3$.  In Figure~\ref{fig:EC table 2} we present the Euler characteristic for $g = 4$ and $7 \leq n \leq 15$.  More extensive data, for $g + n \leq 24$ is available at \url{https://github.com/wilthoma/weight11mgn}.

\begin{figure}
  \begin{adjustbox}{angle=90, scale=.7}
\begin{tabular}{|g|M|M|M|M|M|M|M|} \hline \rowcolor{Gray} g,n & 0 & 1 & 2 & 3 & 4 & 5 & 6\\ 
\hline
 5 & $ 0 $ & $ 0 $ & $ 0 $ & $ 0 $ & $ 0 $ & $ -s_{1,1,1,1,1} $ & $ -s_{1,1,1,1,1,1} + 2s_{3,1,1,1} $ \\  
\hline
 6 & $ 0 $ & $ 0 $ & $ 0 $ & $ 0 $ & $ -s_{2,1,1} $ & $ -s_{2,1,1,1} + s_{3,1,1} + s_{3,2} + 2s_{4,1} $ & $ -4s_{1,1,1,1,1,1} - 2s_{2,1,1,1,1} - 3s_{2,2,1,1} - 3s_{2,2,2} + s_{3,1,1,1} - 2s_{3,2,1} - 2s_{3,3} + 2s_{4,1,1} - 3s_{4,2} - 2s_{5,1} - 3s_{6} $ \\  
\hline
 7 & $ 0 $ & $ 0 $ & $ s_{1,1} $ & $ s_{1,1,1} - 2s_{3} $ & $ 2s_{2,2} - 2s_{3,1} $ & $ -s_{1,1,1,1,1} - 7s_{2,1,1,1} - 6s_{3,1,1} + 2s_{3,2} + s_{4,1} + 4s_{5} $ & $ -3s_{1,1,1,1,1,1} - 14s_{2,1,1,1,1} + 3s_{2,2,1,1} - 5s_{2,2,2} + 2s_{3,1,1,1} + 11s_{3,2,1} + 10s_{3,3} + 20s_{4,1,1} + 3s_{4,2} + 6s_{5,1} $ \\  
\hline
 8 & $ 0 $ & $ 0 $ & $ s_{2} $ & $ 4s_{1,1,1} $ & $ 10s_{1,1,1,1} + 2s_{2,1,1} - 2s_{2,2} - 10s_{3,1} - 3s_{4} $ & $ 16s_{1,1,1,1,1} - s_{2,1,1,1} + 8s_{2,2,1} - 29s_{3,1,1} + 7s_{3,2} + 13s_{5} $ & $ 15s_{1,1,1,1,1,1} - 18s_{2,1,1,1,1} + 4s_{2,2,1,1} - 8s_{2,2,2} - 75s_{3,1,1,1} + s_{3,2,1} + 35s_{3,3} - 8s_{4,1,1} + 23s_{4,2} + 51s_{5,1} + 14s_{6} $ \\  
\hline
 9 & $ s_{} $ & $ s_{1} $ & $ 4s_{2} $ & $ s_{1,1,1} + 11s_{2,1} + 4s_{3} $ & $ 9s_{1,1,1,1} + 25s_{2,1,1} - 6s_{2,2} - 9s_{3,1} - 26s_{4} $ & $ 37s_{1,1,1,1,1} + 45s_{2,1,1,1} + 26s_{2,2,1} - 32s_{3,1,1} - 21s_{3,2} - 60s_{4,1} - 19s_{5} $ & $ 90s_{1,1,1,1,1,1} + 40s_{2,1,1,1,1} + 10s_{2,2,1,1} + 95s_{2,2,2} - 223s_{3,1,1,1} - 63s_{3,2,1} + 56s_{3,3} - 199s_{4,1,1} + 100s_{4,2} + 73s_{5,1} + 94s_{6} $ \\  
\hline
 10 & $ -2s_{} $ & $ -6s_{1} $ & $ -9s_{1,1} - 2s_{2} $ & $ -14s_{1,1,1} + 8s_{2,1} + 30s_{3} $ & $ -s_{1,1,1,1} + 57s_{2,1,1} - 18s_{2,2} + 59s_{3,1} + 7s_{4} $ & $ 30s_{1,1,1,1,1} + 177s_{2,1,1,1} - 18s_{2,2,1} + 104s_{3,1,1} - 188s_{3,2} - 177s_{4,1} - 138s_{5} $ & $ 133s_{1,1,1,1,1,1} + 422s_{2,1,1,1,1} + 145s_{2,2,1,1} + 331s_{2,2,2} + 41s_{3,1,1,1} - 3s_{3,2,1} - 170s_{3,3} - 456s_{4,1,1} + 72s_{4,2} - 134s_{5,1} + 85s_{6} $ \\  
\hline
 11 & $ 2s_{} $ & $ s_{1} $ & $ -12s_{1,1} - 18s_{2} $ & $ -64s_{1,1,1} - 8s_{2,1} + 33s_{3} $ & $ -138s_{1,1,1,1} - 36s_{2,1,1} + 55s_{2,2} + 239s_{3,1} + 127s_{4} $ & $ -196s_{1,1,1,1,1} + 176s_{2,1,1,1} - 81s_{2,2,1} + 667s_{3,1,1} - 138s_{3,2} + 2s_{4,1} - 285s_{5} $ & $ -124s_{1,1,1,1,1,1} + 624s_{2,1,1,1,1} + 190s_{2,2,1,1} + 138s_{2,2,2} + 1548s_{3,1,1,1} - 199s_{3,2,1} - 965s_{3,3} - 62s_{4,1,1} - 1132s_{4,2} - 1456s_{5,1} - 590s_{6} $ \\  
\hline
 12 & $ 0 $ & $ 3s_{1} $ & $ -6s_{1,1} - 59s_{2} $ & $ -38s_{1,1,1} - 162s_{2,1} - 85s_{3} $ & $ -231s_{1,1,1,1} - 433s_{2,1,1} + 114s_{2,2} + 185s_{3,1} + 418s_{4} $ & $ -586s_{1,1,1,1,1} - 504s_{2,1,1,1} - 145s_{2,2,1} + 1023s_{3,1,1} + 847s_{3,2} + 1441s_{4,1} + 475s_{5} $ & $ -1308s_{1,1,1,1,1,1} - 713s_{2,1,1,1,1} - 430s_{2,2,1,1} - 2126s_{2,2,2} + 3745s_{3,1,1,1} - 447s_{3,2,1} - 1780s_{3,3} + 2713s_{4,1,1} - 3378s_{4,2} - 2540s_{5,1} - 2111s_{6} $ \\  
\hline
 13 & $ 8s_{} $ & $ 48s_{1} $ & $ 83s_{1,1} + 25s_{2} $ & $ 111s_{1,1,1} - 228s_{2,1} - 400s_{3} $ & $ -93s_{1,1,1,1} - 830s_{2,1,1} + 131s_{2,2} - 705s_{3,1} - 2s_{4} $ & $ -719s_{1,1,1,1,1} - 3101s_{2,1,1,1} - 328s_{2,2,1} - 1608s_{3,1,1} + 2717s_{3,2} + 3077s_{4,1} + 2311s_{5} $ & $ -2258s_{1,1,1,1,1,1} - 5169s_{2,1,1,1,1} - 963s_{2,2,1,1} - 3914s_{2,2,2} + 2600s_{3,1,1,1} + 4429s_{3,2,1} + 3361s_{3,3} + 10416s_{4,1,1} + 2197s_{4,2} + 3934s_{5,1} - 973s_{6} $ \\  
\hline
 14 & $ -17s_{} $ & $ 11s_{1} $ & $ 169s_{1,1} + 188s_{2} $ & $ 598s_{1,1,1} + 26s_{2,1} - 458s_{3} $ & $ 1259s_{1,1,1,1} + 35s_{2,1,1} - 833s_{2,2} - 3024s_{3,1} - 1759s_{4} $ & $ 1347s_{1,1,1,1,1} - 4226s_{2,1,1,1} - 853s_{2,2,1} - 9176s_{3,1,1} + 1051s_{3,2} - 143s_{4,1} + 3440s_{5} $ & $ 432s_{1,1,1,1,1,1} - 9102s_{2,1,1,1,1} - 2017s_{2,2,1,1} + 53s_{2,2,2} - 15086s_{3,1,1,1} + 12424s_{3,2,1} + 16880s_{3,3} + 10326s_{4,1,1} + 25310s_{4,2} + 26059s_{5,1} + 10237s_{6} $ \\  
\hline
 15 & $ -14s_{} $ & $ -34s_{1} $ & $ 123s_{1,1} + 499s_{2} $ & $ 740s_{1,1,1} + 1676s_{2,1} + 825s_{3} $ & $ 2623s_{1,1,1,1} + 3353s_{2,1,1} - 1733s_{2,2} - 3869s_{3,1} - 4669s_{4} $ & $ 6680s_{1,1,1,1,1} + 6159s_{2,1,1,1} + 3305s_{2,2,1} - 10919s_{3,1,1} - 7565s_{3,2} - 14792s_{4,1} - 4291s_{5} $ & $ 11039s_{1,1,1,1,1,1} - 8455s_{2,1,1,1,1} - 18627s_{2,2,1,1} + 10104s_{2,2,2} - 71325s_{3,1,1,1} - 28303s_{3,2,1} + 15022s_{3,3} - 49340s_{4,1,1} + 26884s_{4,2} + 28282s_{5,1} + 25714s_{6} $ \\  
\hline
 16 & $ -20s_{} $ & $ -268s_{1} $ & $ -491s_{1,1} + 55s_{2} $ & $ -317s_{1,1,1} + 2881s_{2,1} + 3441s_{3} $ & $ 2482s_{1,1,1,1} + 9040s_{2,1,1} - 429s_{2,2} + 4942s_{3,1} - 1635s_{4} $ & $ 11500s_{1,1,1,1,1} + 34542s_{2,1,1,1} + 8547s_{2,2,1} + 10598s_{3,1,1} - 27078s_{3,2} - 36549s_{4,1} - 24154s_{5} $ & $ 26795s_{1,1,1,1,1,1} + 37235s_{2,1,1,1,1} - 16753s_{2,2,1,1} + 23971s_{2,2,2} - 85912s_{3,1,1,1} - 113620s_{3,2,1} - 43822s_{3,3} - 163569s_{4,1,1} - 58533s_{4,2} - 55546s_{5,1} + 11555s_{6} $ \\  
\hline
\end{tabular}
  \end{adjustbox}
  \caption{\label{fig:EC table} The $\mathbb S_n$-equivariant Euler characteristic of $C_{g,n}$.} 
\end{figure}

\begin{figure}
\begin{adjustbox}{scale=.3435}
{
\begin{tabular}{|g|N|N|N|N|N|}
\hline 
\rowcolor{Gray} 
g,n & 7 & 8 & 9 & 10 & 11 \\
\hline
 4 & 
 $ s_{2,1,1,1,1,1} $ & $ s_{2,1,1,1,1,1,1} - s_{3,1,1,1,1,1} - s_{3,2,1,1,1} - 2s_{4,1,1,1,1} $ & $ 3s_{1,1,1,1,1,1,1,1,1} + 2s_{2,1,1,1,1,1,1,1} + 3s_{2,2,1,1,1,1,1} + 3s_{2,2,2,1,1,1} - s_{3,1,1,1,1,1,1} + 2s_{3,2,1,1,1,1} + s_{3,2,2,1,1} + 3s_{3,3,1,1,1} + s_{3,3,2,1} - s_{4,1,1,1,1,1} + 3s_{4,2,1,1,1} + 2s_{4,3,1,1} + 2s_{5,1,1,1,1} + 2s_{5,2,1,1} + 2s_{6,1,1,1} $ & $ 5s_{1,1,1,1,1,1,1,1,1,1} - 2s_{2,2,2,2,1,1} - 12s_{3,1,1,1,1,1,1,1} - 10s_{3,2,1,1,1,1,1} - 12s_{3,2,2,1,1,1} - 5s_{3,2,2,2,1} - 5s_{3,3,2,1,1} - s_{3,3,2,2} - 3s_{3,3,3,1} - 9s_{4,1,1,1,1,1,1} - 11s_{4,2,1,1,1,1} - 13s_{4,2,2,1,1} - s_{4,2,2,2} - 6s_{4,3,1,1,1} - 7s_{4,3,2,1} - 2s_{4,3,3} - 2s_{4,4,1,1} - s_{5,1,1,1,1,1} - 7s_{5,2,1,1,1} - 4s_{5,2,2,1} - 8s_{5,3,1,1} - 2s_{5,3,2} - 2s_{5,4,1} + s_{6,1,1,1,1} - 5s_{6,2,1,1} - 3s_{6,3,1} - 2s_{7,1,1,1} - 2s_{7,2,1} - 2s_{8,1,1} $ & $ 5s_{1,1,1,1,1,1,1,1,1,1,1} + s_{2,1,1,1,1,1,1,1,1,1} + 13s_{2,2,1,1,1,1,1,1,1} + 14s_{2,2,2,1,1,1,1,1} + 11s_{2,2,2,2,1,1,1} + 7s_{2,2,2,2,2,1} - 10s_{3,1,1,1,1,1,1,1,1} + 22s_{3,2,1,1,1,1,1,1} + 21s_{3,2,2,1,1,1,1} + 22s_{3,2,2,2,1,1} + 7s_{3,2,2,2,2} + 42s_{3,3,1,1,1,1,1} + 37s_{3,3,2,1,1,1} + 26s_{3,3,2,2,1} + 5s_{3,3,3,1,1} + 6s_{3,3,3,2} + 7s_{4,1,1,1,1,1,1,1} + 34s_{4,2,1,1,1,1,1} + 33s_{4,2,2,1,1,1} + 27s_{4,2,2,2,1} + 46s_{4,3,1,1,1,1} + 46s_{4,3,2,1,1} + 22s_{4,3,2,2} + 12s_{4,3,3,1} + 23s_{4,4,1,1,1} + 22s_{4,4,2,1} + 4s_{4,4,3} + 27s_{5,1,1,1,1,1,1} + 35s_{5,2,1,1,1,1} + 40s_{5,2,2,1,1} + 15s_{5,2,2,2} + 23s_{5,3,1,1,1} + 32s_{5,3,2,1} + 9s_{5,3,3} + 15s_{5,4,1,1} + 11s_{5,4,2} + 5s_{5,5,1} + 17s_{6,1,1,1,1,1} + 21s_{6,2,1,1,1} + 23s_{6,2,2,1} + 14s_{6,3,1,1} + 13s_{6,3,2} + 8s_{6,4,1} + 2s_{6,5} + 2s_{7,1,1,1,1} + 9s_{7,2,1,1} + 5s_{7,2,2} + 10s_{7,3,1} + 3s_{7,4} - s_{8,1,1,1} + 5s_{8,2,1} + 3s_{8,3} + 2s_{9,1,1} + 2s_{9,2} + 2s_{10,1} $ \\
 \hline
 \end{tabular}
 }
    \end{adjustbox}

\begin{adjustbox}{scale=.32}
{
  \begin{tabular}{|g|N|N|N|N|}
\hline 
\rowcolor{Gray} 
g,n & 12 & 13 & 14 & 15\\ 
\hline
 4 & 
  $ 3s_{1,1,1,1,1,1,1,1,1,1,1,1} - 5s_{2,1,1,1,1,1,1,1,1,1,1} - 23s_{2,2,2,1,1,1,1,1,1} - 21s_{2,2,2,2,1,1,1,1} - 15s_{2,2,2,2,2,1,1} - 9s_{2,2,2,2,2,2} - 32s_{3,1,1,1,1,1,1,1,1,1} - 41s_{3,2,1,1,1,1,1,1,1} - 112s_{3,2,2,1,1,1,1,1} - 91s_{3,2,2,2,1,1,1} - 60s_{3,2,2,2,2,1} - 21s_{3,3,1,1,1,1,1,1} - 123s_{3,3,2,1,1,1,1} - 96s_{3,3,2,2,1,1} - 43s_{3,3,2,2,2} - 102s_{3,3,3,1,1,1} - 63s_{3,3,3,2,1} - 3s_{3,3,3,3} - 27s_{4,1,1,1,1,1,1,1,1} - 98s_{4,2,1,1,1,1,1,1} - 172s_{4,2,2,1,1,1,1} - 130s_{4,2,2,2,1,1} - 64s_{4,2,2,2,2} - 126s_{4,3,1,1,1,1,1} - 241s_{4,3,2,1,1,1} - 170s_{4,3,2,2,1} - 133s_{4,3,3,1,1} - 61s_{4,3,3,2} - 64s_{4,4,1,1,1,1} - 116s_{4,4,2,1,1} - 61s_{4,4,2,2} - 70s_{4,4,3,1} - 14s_{4,4,4} - s_{5,1,1,1,1,1,1,1} - 116s_{5,2,1,1,1,1,1} - 139s_{5,2,2,1,1,1} - 106s_{5,2,2,2,1} - 191s_{5,3,1,1,1,1} - 225s_{5,3,2,1,1} - 115s_{5,3,2,2} - 71s_{5,3,3,1} - 102s_{5,4,1,1,1} - 126s_{5,4,2,1} - 36s_{5,4,3} - 20s_{5,5,1,1} - 22s_{5,5,2} - 20s_{6,1,1,1,1,1,1} - 96s_{6,2,1,1,1,1} - 90s_{6,2,2,1,1} - 57s_{6,2,2,2} - 124s_{6,3,1,1,1} - 124s_{6,3,2,1} - 22s_{6,3,3} - 67s_{6,4,1,1} - 58s_{6,4,2} - 18s_{6,5,1} - 3s_{6,6} - 39s_{7,1,1,1,1,1} - 58s_{7,2,1,1,1} - 58s_{7,2,2,1} - 42s_{7,3,1,1} - 41s_{7,3,2} - 26s_{7,4,1} - 8s_{7,5} - 20s_{8,1,1,1,1} - 24s_{8,2,1,1} - 24s_{8,2,2} - 16s_{8,3,1} - 9s_{8,4} - 2s_{9,1,1,1} - 9s_{9,2,1} - 8s_{9,3} + s_{10,1,1} - 5s_{10,2} - 2s_{11,1} - 2s_{12} $ & $ s_{1,1,1,1,1,1,1,1,1,1,1,1,1} - 6s_{2,1,1,1,1,1,1,1,1,1,1,1} + 23s_{2,2,1,1,1,1,1,1,1,1,1} + 23s_{2,2,2,1,1,1,1,1,1,1} + 73s_{2,2,2,2,1,1,1,1,1} + 63s_{2,2,2,2,2,1,1,1} + 38s_{2,2,2,2,2,2,1} - 14s_{3,1,1,1,1,1,1,1,1,1,1} + 83s_{3,2,1,1,1,1,1,1,1,1} + 145s_{3,2,2,1,1,1,1,1,1} + 281s_{3,2,2,2,1,1,1,1} + 203s_{3,2,2,2,2,1,1} + 87s_{3,2,2,2,2,2} + 199s_{3,3,1,1,1,1,1,1,1} + 354s_{3,3,2,1,1,1,1,1} + 507s_{3,3,2,2,1,1,1} + 303s_{3,3,2,2,2,1} + 184s_{3,3,3,1,1,1,1} + 360s_{3,3,3,2,1,1} + 133s_{3,3,3,2,2} + 144s_{3,3,3,3,1} + 45s_{4,1,1,1,1,1,1,1,1,1} + 197s_{4,2,1,1,1,1,1,1,1} + 430s_{4,2,2,1,1,1,1,1} + 523s_{4,2,2,2,1,1,1} + 318s_{4,2,2,2,2,1} + 388s_{4,3,1,1,1,1,1,1} + 896s_{4,3,2,1,1,1,1} + 953s_{4,3,2,2,1,1} + 377s_{4,3,2,2,2} + 581s_{4,3,3,1,1,1} + 624s_{4,3,3,2,1} + 128s_{4,3,3,3} + 342s_{4,4,1,1,1,1,1} + 656s_{4,4,2,1,1,1} + 546s_{4,4,2,2,1} + 418s_{4,4,3,1,1} + 288s_{4,4,3,2} + 91s_{4,4,4,1} + 128s_{5,1,1,1,1,1,1,1,1} + 312s_{5,2,1,1,1,1,1,1} + 650s_{5,2,2,1,1,1,1} + 569s_{5,2,2,2,1,1} + 241s_{5,2,2,2,2} + 457s_{5,3,1,1,1,1,1} + 1096s_{5,3,2,1,1,1} + 844s_{5,3,2,2,1} + 701s_{5,3,3,1,1} + 377s_{5,3,3,2} + 475s_{5,4,1,1,1,1} + 825s_{5,4,2,1,1} + 418s_{5,4,2,2} + 467s_{5,4,3,1} + 76s_{5,4,4} + 238s_{5,5,1,1,1} + 269s_{5,5,2,1} + 75s_{5,5,3} + 97s_{6,1,1,1,1,1,1,1} + 321s_{6,2,1,1,1,1,1} + 541s_{6,2,2,1,1,1} + 372s_{6,2,2,2,1} + 472s_{6,3,1,1,1,1} + 825s_{6,3,2,1,1} + 385s_{6,3,2,2} + 385s_{6,3,3,1} + 392s_{6,4,1,1,1} + 521s_{6,4,2,1} + 176s_{6,4,3} + 186s_{6,5,1,1} + 134s_{6,5,2} + 40s_{6,6,1} + 27s_{7,1,1,1,1,1,1} + 237s_{7,2,1,1,1,1} + 279s_{7,2,2,1,1} + 137s_{7,2,2,2} + 365s_{7,3,1,1,1} + 396s_{7,3,2,1} + 84s_{7,3,3} + 224s_{7,4,1,1} + 169s_{7,4,2} + 66s_{7,5,1} + 9s_{7,6} + 32s_{8,1,1,1,1,1} + 138s_{8,2,1,1,1} + 116s_{8,2,2,1} + 167s_{8,3,1,1} + 101s_{8,3,2} + 79s_{8,4,1} + 12s_{8,5} + 43s_{9,1,1,1,1} + 64s_{9,2,1,1} + 40s_{9,2,2} + 40s_{9,3,1} + 12s_{9,4} + 20s_{10,1,1,1} + 22s_{10,2,1} + 5s_{10,3} + 2s_{11,1,1} + 4s_{11,2} - s_{12,1} $ & $ -s_{1,1,1,1,1,1,1,1,1,1,1,1,1,1} - 17s_{2,1,1,1,1,1,1,1,1,1,1,1,1} - 13s_{2,2,1,1,1,1,1,1,1,1,1,1} - 93s_{2,2,2,1,1,1,1,1,1,1,1} - 111s_{2,2,2,2,1,1,1,1,1,1} - 174s_{2,2,2,2,2,1,1,1,1} - 129s_{2,2,2,2,2,2,1,1} - 43s_{2,2,2,2,2,2,2} - 47s_{3,1,1,1,1,1,1,1,1,1,1,1} - 120s_{3,2,1,1,1,1,1,1,1,1,1} - 460s_{3,2,2,1,1,1,1,1,1,1} - 657s_{3,2,2,2,1,1,1,1,1} - 800s_{3,2,2,2,2,1,1,1} - 459s_{3,2,2,2,2,2,1} - 169s_{3,3,1,1,1,1,1,1,1,1} - 895s_{3,3,2,1,1,1,1,1,1} - 1319s_{3,3,2,2,1,1,1,1} - 1337s_{3,3,2,2,2,1,1} - 490s_{3,3,2,2,2,2} - 978s_{3,3,3,1,1,1,1,1} - 1429s_{3,3,3,2,1,1,1} - 1192s_{3,3,3,2,2,1} - 523s_{3,3,3,3,1,1} - 400s_{3,3,3,3,2} - 38s_{4,1,1,1,1,1,1,1,1,1,1} - 396s_{4,2,1,1,1,1,1,1,1,1} - 1047s_{4,2,2,1,1,1,1,1,1} - 1608s_{4,2,2,2,1,1,1,1} - 1490s_{4,2,2,2,2,1,1} - 530s_{4,2,2,2,2,2} - 912s_{4,3,1,1,1,1,1,1,1} - 2692s_{4,3,2,1,1,1,1,1} - 3753s_{4,3,2,2,1,1,1} - 2691s_{4,3,2,2,2,1} - 2399s_{4,3,3,1,1,1,1} - 3441s_{4,3,3,2,1,1} - 1680s_{4,3,3,2,2} - 1173s_{4,3,3,3,1} - 787s_{4,4,1,1,1,1,1,1} - 2412s_{4,4,2,1,1,1,1} - 2828s_{4,4,2,2,1,1} - 1236s_{4,4,2,2,2} - 2439s_{4,4,3,1,1,1} - 2641s_{4,4,3,2,1} - 539s_{4,4,3,3} - 837s_{4,4,4,1,1} - 533s_{4,4,4,2} - 50s_{5,1,1,1,1,1,1,1,1,1} - 749s_{5,2,1,1,1,1,1,1,1} - 1586s_{5,2,2,1,1,1,1,1} - 2245s_{5,2,2,2,1,1,1} - 1472s_{5,2,2,2,2,1} - 1790s_{5,3,1,1,1,1,1,1} - 4135s_{5,3,2,1,1,1,1} - 4897s_{5,3,2,2,1,1} - 2038s_{5,3,2,2,2} - 2974s_{5,3,3,1,1,1} - 3530s_{5,3,3,2,1} - 805s_{5,3,3,3} - 1738s_{5,4,1,1,1,1,1} - 4214s_{5,4,2,1,1,1} - 3726s_{5,4,2,2,1} - 3353s_{5,4,3,1,1} - 2214s_{5,4,3,2} - 867s_{5,4,4,1} - 772s_{5,5,1,1,1,1} - 1803s_{5,5,2,1,1} - 922s_{5,5,2,2} - 1257s_{5,5,3,1} - 239s_{5,5,4} - 193s_{6,1,1,1,1,1,1,1,1} - 942s_{6,2,1,1,1,1,1,1} - 1820s_{6,2,2,1,1,1,1} - 2011s_{6,2,2,2,1,1} - 759s_{6,2,2,2,2} - 1848s_{6,3,1,1,1,1,1} - 3973s_{6,3,2,1,1,1} - 3471s_{6,3,2,2,1} - 2446s_{6,3,3,1,1} - 1670s_{6,3,3,2} - 1890s_{6,4,1,1,1,1} - 3684s_{6,4,2,1,1} - 1881s_{6,4,2,2} - 2166s_{6,4,3,1} - 303s_{6,4,4} - 1052s_{6,5,1,1,1} - 1581s_{6,5,2,1} - 596s_{6,5,3} - 271s_{6,6,1,1} - 227s_{6,6,2} - 289s_{7,1,1,1,1,1,1,1} - 826s_{7,2,1,1,1,1,1} - 1520s_{7,2,2,1,1,1} - 1138s_{7,2,2,2,1} - 1287s_{7,3,1,1,1,1} - 2561s_{7,3,2,1,1} - 1243s_{7,3,2,2} - 1299s_{7,3,3,1} - 1306s_{7,4,1,1,1} - 1833s_{7,4,2,1} - 632s_{7,4,3} - 738s_{7,5,1,1} - 530s_{7,5,2} - 178s_{7,6,1} - 9s_{7,7} - 177s_{8,1,1,1,1,1,1} - 545s_{8,2,1,1,1,1} - 846s_{8,2,2,1,1} - 334s_{8,2,2,2} - 789s_{8,3,1,1,1} - 1093s_{8,3,2,1} - 333s_{8,3,3} - 637s_{8,4,1,1} - 456s_{8,4,2} - 253s_{8,5,1} - 31s_{8,6} - 47s_{9,1,1,1,1,1} - 302s_{9,2,1,1,1} - 286s_{9,2,2,1} - 416s_{9,3,1,1} - 252s_{9,3,2} - 200s_{9,4,1} - 27s_{9,5} - 35s_{10,1,1,1,1} - 142s_{10,2,1,1} - 49s_{10,2,2} - 127s_{10,3,1} - 23s_{10,4} - 43s_{11,1,1,1} - 44s_{11,2,1} - 8s_{11,3} - 18s_{12,1,1} - s_{12,2} + s_{13,1} + 3s_{14} $ & $ -2s_{1,1,1,1,1,1,1,1,1,1,1,1,1,1,1} - 19s_{2,1,1,1,1,1,1,1,1,1,1,1,1,1} + 17s_{2,2,1,1,1,1,1,1,1,1,1,1,1} + 35s_{2,2,2,1,1,1,1,1,1,1,1,1} + 223s_{2,2,2,2,1,1,1,1,1,1,1} + 293s_{2,2,2,2,2,1,1,1,1,1} + 379s_{2,2,2,2,2,2,1,1,1} + 217s_{2,2,2,2,2,2,2,1} - 2s_{3,1,1,1,1,1,1,1,1,1,1,1,1} + 190s_{3,2,1,1,1,1,1,1,1,1,1,1} + 578s_{3,2,2,1,1,1,1,1,1,1,1} + 1472s_{3,2,2,2,1,1,1,1,1,1} + 1847s_{3,2,2,2,2,1,1,1,1} + 1750s_{3,2,2,2,2,2,1,1} + 617s_{3,2,2,2,2,2,2} + 567s_{3,3,1,1,1,1,1,1,1,1,1} + 1788s_{3,3,2,1,1,1,1,1,1,1} + 3880s_{3,3,2,2,1,1,1,1,1} + 4411s_{3,3,2,2,2,1,1,1} + 3109s_{3,3,2,2,2,2,1} + 1837s_{3,3,3,1,1,1,1,1,1} + 4749s_{3,3,3,2,1,1,1,1} + 4857s_{3,3,3,2,2,1,1} + 2270s_{3,3,3,2,2,2} + 2990s_{3,3,3,3,1,1,1} + 2821s_{3,3,3,3,2,1} + 392s_{3,3,3,3,3} + 162s_{4,1,1,1,1,1,1,1,1,1,1,1} + 713s_{4,2,1,1,1,1,1,1,1,1,1} + 2356s_{4,2,2,1,1,1,1,1,1,1} + 4213s_{4,2,2,2,1,1,1,1,1} + 4894s_{4,2,2,2,2,1,1,1} + 3262s_{4,2,2,2,2,2,1} + 1857s_{4,3,1,1,1,1,1,1,1,1} + 6912s_{4,3,2,1,1,1,1,1,1} + 12144s_{4,3,2,2,1,1,1,1} + 12039s_{4,3,2,2,2,1,1} + 5043s_{4,3,2,2,2,2} + 7514s_{4,3,3,1,1,1,1,1} + 14445s_{4,3,3,2,1,1,1} + 11810s_{4,3,3,2,2,1} + 7115s_{4,3,3,3,1,1} + 4289s_{4,3,3,3,2} + 2362s_{4,4,1,1,1,1,1,1,1} + 7633s_{4,4,2,1,1,1,1,1} + 12023s_{4,4,2,2,1,1,1} + 9125s_{4,4,2,2,2,1} + 9046s_{4,4,3,1,1,1,1} + 15134s_{4,4,3,2,1,1} + 7371s_{4,4,3,2,2} + 6154s_{4,4,3,3,1} + 4118s_{4,4,4,1,1,1} + 5397s_{4,4,4,2,1} + 1589s_{4,4,4,3} + 356s_{5,1,1,1,1,1,1,1,1,1,1} + 1543s_{5,2,1,1,1,1,1,1,1,1} + 4714s_{5,2,2,1,1,1,1,1,1} + 7084s_{5,2,2,2,1,1,1,1} + 7013s_{5,2,2,2,2,1,1} + 2735s_{5,2,2,2,2,2} + 3736s_{5,3,1,1,1,1,1,1,1} + 13040s_{5,3,2,1,1,1,1,1} + 19142s_{5,3,2,2,1,1,1} + 14726s_{5,3,2,2,2,1} + 13195s_{5,3,3,1,1,1,1} + 19847s_{5,3,3,2,1,1} + 10006s_{5,3,3,2,2} + 7087s_{5,3,3,3,1} + 5395s_{5,4,1,1,1,1,1,1} + 16118s_{5,4,2,1,1,1,1} + 20683s_{5,4,2,2,1,1} + 9489s_{5,4,2,2,2} + 17338s_{5,4,3,1,1,1} + 20971s_{5,4,3,2,1} + 4814s_{5,4,3,3} + 7051s_{5,4,4,1,1} + 5028s_{5,4,4,2} + 3753s_{5,5,1,1,1,1,1} + 9149s_{5,5,2,1,1,1} + 9010s_{5,5,2,2,1} + 8265s_{5,5,3,1,1} + 6064s_{5,5,3,2} + 3173s_{5,5,4,1} + 382s_{5,5,5} + 356s_{6,1,1,1,1,1,1,1,1,1} + 2291s_{6,2,1,1,1,1,1,1,1} + 5785s_{6,2,2,1,1,1,1,1} + 7855s_{6,2,2,2,1,1,1} + 5795s_{6,2,2,2,2,1} + 5511s_{6,3,1,1,1,1,1,1} + 15366s_{6,3,2,1,1,1,1} + 18722s_{6,3,2,2,1,1} + 8604s_{6,3,2,2,2} + 13005s_{6,3,3,1,1,1} + 15010s_{6,3,3,2,1} + 3131s_{6,3,3,3} + 6987s_{6,4,1,1,1,1,1} + 18013s_{6,4,2,1,1,1} + 17157s_{6,4,2,2,1} + 15909s_{6,4,3,1,1} + 10864s_{6,4,3,2} + 4813s_{6,4,4,1} + 5143s_{6,5,1,1,1,1} + 10920s_{6,5,2,1,1} + 6177s_{6,5,2,2} + 7526s_{6,5,3,1} + 1606s_{6,5,4} + 1963s_{6,6,1,1,1} + 2937s_{6,6,2,1} + 1083s_{6,6,3} + 324s_{7,1,1,1,1,1,1,1,1} + 2497s_{7,2,1,1,1,1,1,1} + 5085s_{7,2,2,1,1,1,1} + 6004s_{7,2,2,2,1,1} + 2541s_{7,2,2,2,2} + 5588s_{7,3,1,1,1,1,1} + 12371s_{7,3,2,1,1,1} + 11526s_{7,3,2,2,1} + 8139s_{7,3,3,1,1} + 5722s_{7,3,3,2} + 6073s_{7,4,1,1,1,1} + 12670s_{7,4,2,1,1} + 6942s_{7,4,2,2} + 8043s_{7,4,3,1} + 1287s_{7,4,4} + 3892s_{7,5,1,1,1} + 6361s_{7,5,2,1} + 2531s_{7,5,3} + 1551s_{7,6,1,1} + 1330s_{7,6,2} + 308s_{7,7,1} + 431s_{8,1,1,1,1,1,1,1} + 2025s_{8,2,1,1,1,1,1} + 3573s_{8,2,2,1,1,1} + 3103s_{8,2,2,2,1} + 3774s_{8,3,1,1,1,1} + 7025s_{8,3,2,1,1} + 3834s_{8,3,2,2} + 3441s_{8,3,3,1} + 3693s_{8,4,1,1,1} + 5576s_{8,4,2,1} + 1936s_{8,4,3} + 2041s_{8,5,1,1} + 1786s_{8,5,2} + 611s_{8,6,1} + 81s_{8,7} + 431s_{9,1,1,1,1,1,1} + 1250s_{9,2,1,1,1,1} + 1974s_{9,2,2,1,1} + 921s_{9,2,2,2} + 1830s_{9,3,1,1,1} + 2748s_{9,3,2,1} + 844s_{9,3,3} + 1588s_{9,4,1,1} + 1292s_{9,4,2} + 710s_{9,5,1} + 116s_{9,6} + 219s_{10,1,1,1,1,1} + 626s_{10,2,1,1,1} + 734s_{10,2,2,1} + 762s_{10,3,1,1} + 634s_{10,3,2} + 462s_{10,4,1} + 115s_{10,5} + 51s_{11,1,1,1,1} + 261s_{11,2,1,1} + 140s_{11,2,2} + 269s_{11,3,1} + 73s_{11,4} + 31s_{12,1,1,1} + 81s_{12,2,1} + 50s_{12,3} + 29s_{13,1,1} + 13s_{13,2} + 7s_{14,1} $ \\  
\hline
\end{tabular}
}
    \end{adjustbox}
    \caption{\label{fig:EC table 2} The $\mathbb S_n$-equivariant Euler characteristic of $C_{g,n}$ for $g = 4$ and $n \leq 15$.}
\end{figure}

\appendix 

\section{Recollection and variant of \cite{TWspherical}}
\label{app:TW recollection}

We denote by $\sX^{conn}\subset \sX$ the subcomplex spanned by graphs that are connected in the blown-up picture.
We denote by $\sX^{tp}$ a graph complex defined just like $\sX$, just allowing tadpoles at all vertices, and by $\sX^{tp,conn}$ its connected subcomplex.
Similarly, we denote by $\HGC\subset \fHGC$ the connected part, so that 
\[
  \fHGC =\Sym(\HGC)
\]
is a symmetric product. We denote by $\HGC^{tp}$ the variant of the hairy graph complex $\HGC$ generated by graphs that may have tadpoles at vertices.
We shall recall the following result:

\begin{thm}[Theorem 3.1 of \cite{TWspherical}]\label{thm:TW mapping cone}
The mapping cone of the inclusion $\HGC^{tp}\to (\sX^{tp,conn},\delta_s)$ has two-dimensional cohomology, spanned by one class whose projection on $\HGC^{tp}$ is 
\[
    D=
    \begin{tikzpicture}[baseline=-.65ex,every loop/.style={}]
    \node (v) at (0,-.3) {};
    \node [int] (w) at (0,.3) {};
    \draw (v) edge (w) (w) edge[loop] (w);
    \end{tikzpicture}
\]
and by one class whose projection to $(\sX^{tp,conn},\delta_s)$ is 
\[
    T=
\begin{tikzpicture}[baseline=0]
\node (v1) at (-.6,0) {$\epsilon$};
\node (v2) at (0,0) {$\omega$};
\node (v3) at (.6,0) {$\omega$};
\node [int] (w) at (0,.6) {};
\draw (w) edge (v1) edge (v2) edge (v3);
\end{tikzpicture}.
\]
\end{thm}

From this we can easily deduce the following tadpole-free variant:
\begin{cor}\label{cor:TW tadpolefree}
    The mapping cone of the inclusion $\HGC\to (\sX^{conn},\delta_s)$ has three-dimensional cohomology, spanned by three classes whose projections to $(\sX^{conn},\delta_s)$ are $T$ as above, and 
    \begin{align*}
        L^\omega&=
\begin{tikzpicture}[baseline=-.65ex]
\node (v) at (0,0) {$\epsilon$};
\node (w) at (1,0) {$\omega$};
\draw (v) edge (w);
\end{tikzpicture}
&
L^\epsilon&=
\begin{tikzpicture}[baseline=-.65ex]
\node (v) at (0,0) {$\epsilon$};
\node (w) at (1,0) {$\epsilon$};
\draw (v) edge (w);
\end{tikzpicture}\, .
    \end{align*}
    respectively.
\end{cor}
\begin{proof}
We first compare the tadpole-free and tadpole-carrying versions of our complexes, that is, we study the mapping cones of the projections 
\begin{align*}
    \HGC^{tp}&\to \HGC & \sX^{tp, conn}&\to \sX^{conn}.
\end{align*}
in each case the map is surjective, so the mapping cone is quasi-isomorphic to the kernel, which is spanned by graphs that have at least one tadpole.
Following the arguments of \cite[Lemma 5]{AWZ} one sees that this complex has cohomology spanned by graphs that have a single vertex, carrying a tadpole, of valence 3.
The complete list of such graphs is as follows:
\begin{align*}
  D&= 
  \begin{tikzpicture}[baseline=-.65ex,every loop/.style={}]
      \node (v) at (0,0) {};
      \node [int] (w) at (1,0) {};
      \draw (v) edge (w) (w) edge[loop] (w);
      \end{tikzpicture}
  &
    D^\omega&= 
\begin{tikzpicture}[baseline=-.65ex,every loop/.style={}]
    \node (v) at (0,0) {$\omega$};
    \node [int] (w) at (1,0) {};
    \draw (v) edge (w) (w) edge[loop] (w);
    \end{tikzpicture}
&
D^\epsilon&=
\begin{tikzpicture}[baseline=-.65ex,every loop/.style={}]
    \node (v) at (0,0) {$\epsilon$};
    \node [int] (w) at (1,0) {};
    \draw (v) edge (w) (w) edge[loop] (w);
    \end{tikzpicture}.
\end{align*}
The cocycle $D$ is not exact in $\HGC^{tp}$.
Hence the cohomology of $\HGC^{tp}$ is one dimension larger than that of $\HGC$, with the additional dimension spanned by the class of $D^\omega$.

The cocycles $D^\omega$ and $D^\epsilon$ are both exact in $\sX^{tp, conn}$ since $D^\omega=\delta_s L^\omega$  and $D^\epsilon=\delta_s L^\epsilon$. 
Hence the cohomology of $\HGC$ is two dimensions larger than that of $\HGC^{tp}$, with the additional generators $L^\omega$ and $L^\epsilon$, which are closed elements in $\sX^{conn}$ but not in $\sX^{tp, conn}$.

Accounting for these (small) differences the corollary then follows easily from Theorem \ref{thm:TW mapping cone}.
\end{proof}

\bibliographystyle{amsalpha}

\end{document}